\documentclass[11pt,oneside,reqno]{amsart}

\usepackage{amsmath}
\usepackage{amssymb}  
\usepackage{mathrsfs}
\usepackage{pifont}
\usepackage{color,xcolor}
\usepackage{cite}
\usepackage{graphicx}
\usepackage{wrapfig}
\usepackage{indentfirst}
\usepackage{verbatim}
\usepackage{makecell}
\usepackage{subfigure}
\usepackage{float}
\usepackage{amsthm}
\usepackage{enumerate}
\usepackage{bbm}
\usepackage{mathtools}
\mathtoolsset{showonlyrefs} % only number referenced equations
\usepackage{geometry}
\usepackage{appendix}
\usepackage{tcolorbox}
\usepackage{pgfplots}
\usepackage{tikz}
\usepackage{tikz-network}

\usepackage{xcolor}
\usepackage[
pdfstartview=FitH,
CJKbookmarks=true,
bookmarksnumbered=true,
bookmarksopen=true,
colorlinks=true,
citecolor=magenta,% magenta , cyan
linkcolor=blue,
]{hyperref}

\newcommand\N{\ensuremath{\mathbb{N}}}
\newcommand\T{\ensuremath{\mathbb{T}}}
\newcommand\R{\ensuremath{\mathbb{R}}}
\newcommand\C{\ensuremath{\mathbb{C}}}

\newcommand\Z{\ensuremath{\mathbb{Z}}}

\newcommand\moduli{\ensuremath{\mathscr{M}_{\Gamma}}}

\newcommand\bloch{\ensuremath{\mathcal{B}_n}}
\newcommand\Oph{\ensuremath{\mathrm{Op}_h}}

\newtheorem{theorem}{Theorem}[section]
\newtheorem{lemma}[theorem]{Lemma}
\newtheorem{corollary}[theorem]{Corollary}
\newtheorem{proposition}[theorem]{Proposition}
\newtheorem{conjecture}[theorem]{Conjecture}
\theoremstyle{definition}
\newtheorem{definition}{Definition}[section]
\newtheorem{example}[definition]{Example}
\newtheorem{question}[definition]{Question}

\newtheorem{remark}[definition]{Remark}

\numberwithin{equation}{section}

\theoremstyle{plain}
\newtheorem{assumptionalt}{Assumption}
\newenvironment{assumption}[1]{
  \renewcommand\theassumptionalt{#1}
  \assumptionalt
}{\endassumptionalt}
\title[Observability and Semiclassical Control for Schrödinger Equations]{Observability and Semiclassical Control for Schrödinger Equations on Non-compact Hyperbolic Surfaces}
\author[X.~Fu]{Xin Fu}
\address[X.~Fu]{School of Science, Institute for Theoretical Sciences, Westlake University, Hangzhou, Zhejiang Province, 310030, P.R. China}
\email{1550862645cf@gmail.com}

\author[Y.~Gong]{Yulin Gong}
\address[Y.~Gong]{School of Mathematics, University of Bristol, Bristol, BS8 1UG, U.K.}
\email{dz25829@bristol.ac.uk}

\author[Y. Wang]{Yunlei Wang}
\address[Y. Wang]{Department of Mathematics, Louisiana State University, Baton Rouge, LA 70803, USA}
\email{ywang30@lsu.edu}

\date{\today}
\begin{document}

\begin{abstract}
We study the observability of the Schrödinger equation on $X$, a non-compact covering space of a compact hyperbolic surface $M$. Using a generalized Bloch theory, functions on $X$ are identified as sections of flat Hilbert bundles over $M$. We develop a semiclassical analysis framework for such bundles and generalize the result of semiclassical control estimates in [Dyatlov and Jin, Acta Math., 220 (2018), pp. 297-339] to all flat Hilbert bundles over $M$, with uniform constants with respect to the choice of bundle. Furthermore, when the Riemannian cover $X \to M$ is a normal cover with a virtually Abelian deck transformation group $\Gamma$, we combine the uniform semiclassical control estimates on flat Hilbert bundles with the generalized Bloch theory to derive observability from any $\Gamma$-periodic open subsets of $X$. We also discuss applications of the uniform semiclassical control estimates in spectral geometry.
\end{abstract}

\maketitle
\vspace{-1em}
\tableofcontents

\section{Introduction}
In this article, we study observability for the Schrödinger equation on a non-compact connected Riemannian manifold $(X,g)$. The corresponding Cauchy problem is
\begin{equation}\label{Schrödinger}
\left\{\begin{aligned}
    &i\partial_t u +\Delta_g  u =0 && \text{in } (0,\infty)\times X,\\
    &u|_{t=0}=u_0 && \mathrm{on}\  X,
\end{aligned}\right.
\end{equation}
where $\Delta_g$ denotes the Laplace–Beltrami operator on $(X,g)$, locally given by
\begin{equation*}
    \Delta_g f=\frac{1}{\sqrt{G}}\partial_i\big( \sqrt{G}g^{ij}\partial_j f\big), \qquad G=\mathrm{det}(g_{ij}),
\end{equation*}
and $u_0\in L^2(X)$. 

Throughout this work, we use $C(\theta_1,\cdots,\theta_n)$ to denote a constant depending only on the parameters $\theta_1,\cdots,\theta_n$, whose value may change from line to line.

The notion of observability is given as follows:
\begin{definition}
    Let $T>0$ and $S$ be a measurable subset of $X$ with positive measure. Problem \eqref{Schrödinger} is said to be \emph{observable from the set $S$ in time $T>0$} if there exists a constant $C=C(X,S,T)>0$ such that, for any $u_0\in L^2(X)$, the mild solution $u(t,x)$ of problem \eqref{Schrödinger} satisfies
    \begin{equation}\label{observability}
        \| u_0\|^2_{L^2(X)}\le C \int_0^T \int_S| u(t,x) |^2 \,dxdt.
    \end{equation}
\end{definition}

By the Hilbert uniqueness method \cite{Lions88vol1, Lions88vol2}, observability is equivalent to exact controllability. More precisely, problem \eqref{Schrödinger} is observable from the set $S$ in time $T>0$ if and only if, for any $u_0 , u_1\in L^2(X)$, there exists a control $f \in L^2((0, T) \times S)$ such that the following problem is solvable:
\begin{equation*}
\left\{\begin{aligned}
    &i\partial_t u +\Delta_g  u =f && \text{in } (0,\infty)\times X,\\
    &(u|_{t=0},u|_{t=T}) =(u_0, u_1) && \mathrm{on}\  X.
\end{aligned}\right.
\end{equation*}

Observability has been extensively studied over the past decades, especially on compact Riemannian manifolds with and without boundary. We shall present a brief review in Section \ref{review}. In the compact setting, if $S$ satisfies the Geometric Control Condition (GCC), first introduced by \cite{bardo1992}, that is, every (generalized) geodesic meets $S$ within a fixed time $T_0>0$, then the observability inequality \eqref{observability} holds for any $T>0$ by Lebeau's approach \cite{lebeau1992controle}. However, while GCC is necessary and sufficient for observability of wave equations, it is only a sufficient condition for observability of Schrödinger equations on standard compact domains, such as disks, tori and compact hyperbolic surfaces. 

By contrast, very little is known for non-compact domains in which GCC fails, since the usual compactness arguments break down. In this case one must develop new techniques to overcome the lack of compactness. The main goal of this work is to establish observability for Schrödinger equations on certain kinds of non-compact Riemannian manifolds from observation sets that do not satisfy GCC.

\subsection{Problem setting}

Throughout this article, we assume that $(M,g)$ is a compact connected Riemannian manifold, and $\pi: X \rightarrow M$ is a Riemannian covering map with the set of right cosets
\begin{equation*}
    \Gamma = \pi_{*}(\pi_1(X)) \backslash \pi_1(M).
\end{equation*}
See \eqref{defineV} for the explicit meaning of $\Gamma$. We note that if $\pi$ is normal, then $\Gamma$ is isomorphic to the deck transformation group of $\pi$. We assume that $\Omega \subset M$ is a non-empty open set and $S = \pi^{-1}(\Omega)$. In other words, $S$ is a $\Gamma$-periodic set with the fundamental domain $\Omega$.

%A subset $S\subset X$ is called $\Gamma$-periodic if $\gamma(S)=S$ for every $\gamma \in \Gamma$. Equivalently, $S=\pi^{-1}(\Omega)$ for some $\Omega\subset M$. 

The following question arises naturally:

\begin{question}\label{question}
If problem~\eqref{Schrödinger} on the base space $M$ is observable from $\Omega$ in time $T>0$, does it follow that the corresponding problem on $X$ is observable from $S=\pi^{-1}(\Omega)$ in time $T$, and how does the control constant $C$ in \eqref{observability} depend on $\Gamma$?
\end{question}

For the Euclidean case $X= \mathbb{R}^d$, $\Gamma = \mathbb{Z}^d$, and $M= \mathbb{T}^d$, this question has been answered affirmatively in \cite{wunsch2017periodic, taufer2023controllability, balch2023obs}. In this work, we treat the case where $M$ is a compact hyperbolic surface and $\Gamma$ may be non-commutative.

\subsection{Main results}

%Throughout this article we assume that $(X,g)$ is a non-compact, smooth, connected Riemannian manifold. 

%In the following, we assume that $M$ is a compact hyperbolic surface. Since the fundamental group $\pi_{1}(M)$ is complex, the covering spaces of $M$ provide a wide class of non-compact hyperbolic surfaces.

Under suitable assumptions on the covering map $\pi$ and the group $\Gamma$, we provide a positive resolution to Question \ref{question}, as stated in Theorem \ref{thm-main-obs}. The proof of the observability inequality \eqref{observability} is established in two main stages:
\begin{enumerate}
    \item Establishing a semiclassical (high-frequency) control estimate; \label{step1}
    \item Eliminating the remaining low-frequency terms via a compactness argument.
\end{enumerate}

Our primary contributions are summarized in the following results. First, we establish a semiclassical control estimate that is uniform for any covering spaces:

\begin{theorem}\label{highfrequencycontrol}
    Let $M$ be a compact hyperbolic surface and $\Omega \subset M$ be a nonempty open subset. 
    There exist constants $C = C(M, \Omega) > 0$ and $h_0 = h_0(M, \Omega) > 0$ such that for any Riemannian cover $\pi: X \to M$, any $0 < h < h_0$, and any $u \in H^2(X)$,
    \begin{equation*}
        \|u\|_{L^2(X)} \leq C\left(\|u\|_{L^2\left(\pi^{-1}(\Omega)\right)} + \frac{|\log h|}{h}\left\|\left(-h^2 \Delta_g - I\right) u\right\|_{L^2(X)}\right).
    \end{equation*}
\end{theorem}

To derive the full observability result from the high-frequency estimate, we require the following structural assumption:

\begin{assumption}{H}\label{normal-and-type1-assumption}
    The Riemannian covering map $\pi: X \rightarrow M$ is normal, and the associated deck transformation group $\Gamma$ is of Type I.
\end{assumption}

A detailed justification for this requirement is provided in Remark \ref{reasonoftype1}. Let $d_{\Gamma}$ denote the supremum of the dimensions of the irreducible unitary representations of $\Gamma$. Under Assumption \ref{normal-and-type1-assumption}, we obtain the following observability result for Schrödinger equations on Type I covers of compact hyperbolic surfaces:

\begin{theorem}\label{thm-main-obs}
    Under Assumption \ref{normal-and-type1-assumption}, there exists a constant $C = C(M, \Omega, T, d_\Gamma) > 0$ such that for any $u \in L^2(X)$:
    \begin{equation*}
        \|u\|_{L^2(X)}^2 \leq C \int_{0}^{T}\|\mathrm{e}^{it \Delta_g}u\|^2_{L^2(\pi^{-1}(\Omega))} \, dt.
    \end{equation*}
\end{theorem}
Next, we outline the proof strategy for Theorems \ref{highfrequencycontrol} and \ref{thm-main-obs}.

To study the observability problem on the covering space $X$, we employ a generalized Bloch theory that reduces this problem to an observability problem on flat Hilbert bundles over the base space $M$. In Section \ref{subs:noncBloch}, we extend the non-commutative Bloch transform (see \cite[(5.1)]{nagy2024hyperbolic}) to an arbitrary Riemannian covering map $\pi: X\to M$ and construct a unitary isometry $\mathcal{B}_{\mathrm{NC}}: L^2(X)\rightarrow L^2(M;F^{\rho_{\pi}})$, where $F^{\rho_{\pi}}$ denotes a flat Hilbert bundle over $M$. The non-commutative Bloch transform $\mathcal{B}_{\mathrm{NC}}$ satisfies 
\begin{equation*}
    \mathcal{B}_{\mathrm{NC}}\circ \Delta_g=\Delta^{\rho_{\pi}}\circ \mathcal{B}_{\mathrm{NC}},
\end{equation*}
where $\Delta^{\rho_{\pi}}$ denotes the twisted Bochner–Laplace operator; see Section \ref{subsec: twisted laplacian}.

The discussion above leads us to study observability for Schrödinger equations on flat Hilbert bundles over $M$. For this problem, we use the microlocal approach initiated in \cite{ralston1969solutions, rauch1974exponetial, bardo1992, Burq1997control}. 

In the scalar case on compact hyperbolic surfaces, the semiclassical control estimate was obtained by Dyatlov and Jin \cite{dyatlov2018semiclassical}. In this work, we develop a semiclassical calculus on flat Hilbert bundles over general compact Riemannian manifolds and extend the results of \cite{dyatlov2018semiclassical} to the setting of flat Hilbert bundles over compact hyperbolic surfaces, which are of independent interest.

% Given the manifold $(X,g_X)$ and the symmetry group $\Gamma$, there exists a closed fundamental domain $D\subset X$ such that $\pi(D)=M$ and $\left.\pi\right|_{D^{o}}$ is injective map on the interior of $D$. The whole manifold $X$ is tiling of copies of $D\subset X$, i.e. $X=\bigcup\limits_{\gamma\in \Gamma}\gamma(D)$.

% \begin{remark}
%     We may change the condition of type I covering to  covering without changing the full facts presented by this result. Indeed, {\color{red} to be continued...}
% \end{remark}
% In the study of lattice-periodic Schrodinger operators on Euclidean spaces. The Floquet-Bloch theory provides a powerful tool to decompose the wave function into a family quasi-periodic functions on the fundamental domain.

\subsubsection{Semiclassical control estimates on flat Hilbert bundles}

Let $\pi_1(M)$ be the fundamental group of $M$. We denote
\begin{equation}
\begin{aligned}
    \mathcal{C}_m := \{(\mathcal{H},\rho): \ &\mathcal{H} \mathrm{\ is\ a\ Hilbert\ space\ with}\ \mathrm{dim}\,\mathcal{H}\leq m,  \\
    & \rho:\pi_1(M) \rightarrow \mathrm{U}(\mathcal{H}) \mathrm{\ is\ a\ unitary\ representation}\},
\end{aligned}
\end{equation}
and
\begin{equation}
\begin{aligned}
    \mathcal{C} := \{(\mathcal{H},\rho): \ & \mathcal{H} \mathrm{\ is\ a\ separable\  Hilbert\ space},  
    \\
    &\rho:\pi_1(M) \rightarrow \mathrm{U}(\mathcal{H}) \mathrm{\ is\ a\ unitary\ representation}\}.
\end{aligned}
\end{equation}
For each $(\mathcal{H},\rho)\in \mathcal{C}$, the associated flat Hilbert bundle $F^\rho$ is defined in \eqref{defofflatbundle} (by the Riemann-Hilbert correspondence, any flat Hilbert bundle over $M$ has the form of $F^\rho$ for some $(\mathcal{H},\rho)\in \mathcal{C}$). 

For each $(\mathcal{H},\rho)\in \mathcal{C}$ and $0<h \leq 1$, we introduce, in Section \ref{subsec:quansymb}, the quantization procedure  
\begin{equation*}
    \Oph^{\rho}(a)\in \Psi^{\mathrm{comp},sc}_h(M;F^\rho): L^2(M;F^\rho)\to L^2(M;F^{\rho})
\end{equation*}
for $a=a(x,\xi;h) \in C_{0}^{\infty}(T^*M)$. Here the superscript $sc$ indicates that the symbol $a$ is scalar-valued. The key observation is that the operator-norms of $\Oph^{\rho}(a)$ are \emph{uniformly} bounded in $(\mathcal{H},\rho)\in \mathcal{C}$ and $0<h\leq 1$; see Proposition \ref{uniformboundedtheorem}. Conversely, for any $A \in \Psi^{\mathrm{comp},sc}_h(M;F^\rho)$, we define the principal symbol $\sigma^{\rho}_h(A)$ in Definition \ref{defofprincipalsymbol} such that $A-\Oph^{\rho}(\sigma^{\rho}_h(A))\in h\Psi^{\mathrm{comp},sc}_h(M;F^\rho)$.

To define the cutoff Schrödinger propagator $\exp\left(-itP/h\right)$ in \eqref{schordingerpropagator}, as defined in \cite[(2.13)]{dyatlov2018semiclassical}, we need to show that the functional calculus of the semiclassical twisted Laplacian $-h^2\Delta^{\rho}$ remains in $\Psi^{\mathrm{comp},sc}_h(M;F^\rho)$:

\begin{proposition}
    \label{functionalcalculus}
For any $f\in C_{0}^{\infty}(\R)$, we have
\begin{equation}
f(-h^2\Delta^{\rho}) \in \Psi^{\mathrm{comp},sc}_{h}(M;F^\rho) \quad \mathrm{with} \  \sigma_{h}^{\rho}(f(-h^2\Delta^{\rho}))=f(|\xi|^2_g).
\end{equation}
Moreover, $\mathrm{WF}_{h}(f(-h^2 \Delta^\rho))\subset \{(x,\xi)\in T^*M: |\xi|_g^2\in \mathrm{supp}\,f\}$.
\end{proposition}

Following the proof of the semiclassical control estimate in \cite[Theorem 2]{dyatlov2018semiclassical}, we extend the Fourier integral operators to the setting of flat Hilbert bundles and introduce an anisotropic quantization associated with a Lagrangian foliation $L$ on $T^*M$. More precisely, for any anisotropic symbol $a \in S^{\mathrm{comp}}_{L,\mu}$, we define
\begin{equation*}
    \Oph^{\rho,L}(a): L^2(M;F^{\rho})\mapsto L^2(M;F^{\rho});
\end{equation*}
see Sections \ref{subsec:fio} and \ref{subsec: aniso}. Notably, the operator-norms of $\Oph^{\rho,L}(a)$ are also uniformly bounded in $(\mathcal{H},\rho)\in \mathcal{C}$ and $0<h\leq 1$. Moreover, we show that the long-time propagation of symbol $a\in C_{0}^{\infty}(T^*M)$ belongs to the anisotropic symbol class $S^{\mathrm{comp}}_{L,\mu}$, which yields a uniform long-time Egorov's Theorem \ref{longtimeegorovtheorem}.

Notice that the key feature of the quantization procedures introduced in Section \ref{sec:quantization on flat bundles} is the uniformity of the constants appearing in the relevant inequalities with respect to $(\mathcal{H},\rho)\in \mathcal{C}$. As a consequence, we extend the semiclassical control estimates in \cite{dyatlov2018semiclassical} to uniform semiclassical control estimates on flat Hilbert bundles. The following theorem is a direct corollary of Theorem \ref{generalizedmicrolocalcontrolonhilbertbundle}.

\begin{theorem}\label{microlocalcontrolonhilbertbundle}
Let $M$ be a compact hyperbolic surface and $a \in C_{0}^{\infty}(T^* M)$ such that $ a |_{S^* M} \not \equiv 0$. There exist constants $C=C(M,a)>0$ and $h_{0}  = h_0(M,a)>0$ such that for any $(\mathcal{H},\rho)\in \mathcal{C}$, any $0<h<h_{0}$ and any $u \in H^2(M;F^{\rho})$,
\begin{equation}\label{microlocalcontrolonhilbertbundleq}
\|u\|_{L^2(M;F^{\rho})} \leq C\left( \left\|\Oph^{\rho}(a) u\right\|_{L^2(M;F^{\rho})}+\frac{|\log h|}{h}\left\|\left(-h^2 \Delta^{\rho}-I\right) u\right\|_{L^2(M;F^{\rho})} \right).
\end{equation}
\end{theorem}

A direct corollary of Theorem \ref{microlocalcontrolonhilbertbundle} is the \emph{uniform} lower bound of eigensections. For any $(\mathcal{H},\rho)\in \mathcal{C}$, let $u_{\lambda,\rho}$ be a normalized eigensection of $-\Delta^{\rho}: L^2(M;F^\rho)\to L^2(M;F^{\rho})$ with respect to eigenvalue $\lambda$, \emph{i.e.},
\begin{equation}\label{formeigensection}
    \Delta^\rho u_{\lambda,\rho}  +\lambda  u_{\lambda,\rho} =0, \qquad \|u_{\lambda,\rho}\|_{L^2(M;F^{\rho})}=1.
\end{equation}

\begin{corollary}\label{controlofeigensection}
Let $M$ be a compact hyperbolic surface and $\Omega \subset M$ be a nonempty open subset. There exist constants $c(M,\Omega)>0$ and $\lambda_{0}(M,\Omega)>0$ such that for any $(\mathcal{H},\rho)\in \mathcal{C}$, any $\lambda >\lambda_{0}(M,\Omega)$ and any eigensection $u_{\lambda,\rho} $ satisfying \eqref{formeigensection},
\begin{equation*}
\| u_{\lambda,\rho} \|_{L^2(\Omega; F^{\rho})}\geq c(M,\Omega) .
\end{equation*}
\end{corollary}

\begin{remark}
A particular case is when $\mathcal{H}$ is finite-dimensional. In this case, the spectrum of the Laplacian $\Delta^{\rho}: L^2(M;F^\rho)\to L^2(M;F^{\rho})$ is discrete and consists of only  eigenvalues. Corollary \ref{controlofeigensection} together with the unique continuation principle yield: for any $m \in \N$ and any $(\mathcal{H},\rho) \in \mathcal{C}_m$, any eigensection $u_\rho$ of $\Delta^\rho$ satisfies
\begin{equation*}
    \| u_\rho \|_{L^2(\Omega; F^{\rho})}\geq c(M,\Omega,m) \|u_\rho \|_{L^2(M;F^{\rho})}.
\end{equation*}
\end{remark}

We now apply Theorem \ref{microlocalcontrolonhilbertbundle} to the flat Hilbert bundle $F^{\rho_\pi}$ associated with the Riemannian covering map $\pi:X\to M$. Combining this result with the non-commutative Bloch transform $\mathcal{B}_{\mathrm{NC}}$, we prove Theorem \ref{highfrequencycontrol}, which gives a uniform semiclassical control estimate on any Riemannian cover $X$ of $M$.

%\begin{theorem}\label{highfrequencycontrol} Let $M$ be a compact hyperbolic surface and $\Omega \subset M$ be a nonempty open subset. There exist constants $C=C(M,\Omega)>0$ and $h_0=h_0(M,\Omega)>0$ such that for any $0<h< h_0$ and any $u \in H^2(X)$, \begin{equation*} \|u\|_{L^2(X)} \leq C\left(\|u\|_{L^2\left(\pi^{-1}(\Omega)\right)}+\frac{|\log h|}{h}\left\|\left(-h^2 \Delta_g-I\right) u\right\|_{L^2(X)}\right) . \end{equation*} \end{theorem} 

Proofs of Corollary \ref{controlofeigensection} and other applications of Theorem \ref{microlocalcontrolonhilbertbundle} in quantum chaos are given in Section~\ref{subsec:application}.

\subsubsection{Observability inequality}

At this point, we have obtained the semiclassical control estimate for any Riemannian covering map $\pi:X\rightarrow M$. To obtain the observability inequality, the second (final) step is to remove the remaining low-frequency term through a compactness argument. 

However, when $X$ is non-compact, the compactness argument in the usual proof of observability fails because $\Delta_g$ is not a Fredholm operator on $L^2(X)$. Additionally, $\Delta^{\rho}$ is not a Fredholm operator on $L^2(M;F^{\rho})$ for any infinite-dimensional unitary representation $\rho$.

To address the difficulties arising from non-compactness, we impose Assumption \ref{normal-and-type1-assumption}.

\begin{remark}\label{reasonoftype1}
We require Assumption~\ref{normal-and-type1-assumption} for three main reasons:
\begin{itemize}
    \item[(a)] When $\pi$ is normal, the set of cosets $\Gamma$ forms a group. This allows us to define the Fourier transform on $\Gamma$ and, consequently, the generalized Bloch transform.
    \item[(b)] When $\Gamma$ is a type I group, the Fourier transform $\mathcal{F}: \ell^2(\Gamma)\to L^2(\widehat{\Gamma},d\mu)$ is an isometry, where $d\mu$ is the Plancherel measure on the dual space $\widehat{\Gamma}$ (see \eqref{uniF}).
    \item[(c)] For a type I group $\Gamma$, the dimensions of its irreducible unitary representations are uniformly bounded, \emph{i.e.}, $d_\Gamma<\infty$ (see \cite[Theorem 4]{kocabova2008generalized}). 
\end{itemize}
\end{remark}

It follows from (a) and (b) in Remark \ref{reasonoftype1} that the generalized Bloch transform 
\begin{equation}\label{introBiso}
    \mathcal{B}: L^2(X) \to \int_{\widehat{\Gamma}}^{\oplus}L^2(M;\mathrm{End}(F^{\rho}))\,d\mu(\rho)
\end{equation}
is an isometry. This transform can be viewed as the composition of the fiber-wise Fourier transform and the non-commutative Bloch transform (see Definition \ref{defbloch1}). It decomposes functions on $X$ into a family of sections of flat Hilbert bundles over $M$ associated with the irreducible unitary representations of $\Gamma$. Hence, by part (c) of Remark \ref{reasonoftype1}, we may apply the standard compactness argument to prove observability.

We refer to Section \ref{sec:mapBlochtypeI} for the definition and properties of type I groups. Note that every Abelian group is type I. We also give a non-commutative example of type I groups in Section \ref{sec:typeIcover}.

Assuming Assumption~\ref{normal-and-type1-assumption}, our strategy for proving the observability inequality is as follows. Applying the generalized Bloch transform and using the isometry property, we reduce the observability for the observable set $S$ to the uniform observability for the observable set $\Omega$ on a family of finite-dimensional flat Hilbert bundles. Then, applying Theorem \ref{microlocalcontrolonhilbertbundle} and the usual compactness argument to remove the low-frequency error term.

%We now assume that $\Gamma$ is a countable discrete group acting freely and properly discontinuously on $X$ by isometries, making the metric $\Gamma$-invariant. Then the quotient $M:=X / \Gamma$ is a connected Riemannian manifold endowed with the induced metric, still denoted by $g$. Equivalently, the natural projection $\pi: X \rightarrow M$ is a Riemannian normal covering map with the covering transformation group $\Gamma$. We assume that $\Omega \subset M$ is a non-empty open set and $S = \pi^{-1}(\Omega)$. In other words, $S$ is a $\Gamma$-periodic set with the fundamental domain $\Omega$. We assume that $\Gamma$ is the type {\rm I} group. 

\begin{theorem}\label{thm-uniform-obs}
 Let $T>0$, $M$ be a compact hyperbolic surface and $\Omega \subset M$ be a nonempty open subset. For any $(\mathcal{H},\rho) \in \mathcal{C}_m$, there exists a constant $K=K(M, \Omega, T, m)>0$ such that for any $u \in L^2(M;F^{\rho})$, 
\begin{equation}\label{eqn-obs}
\|u\|_{L^2(M;F^{\rho})}^2 \leq K\int_{0}^{T}\|\mathrm{e}^{it \Delta^{\rho}}u\|^2_{L^2(\Omega;F^{\rho})} \, dt. 
\end{equation}
\end{theorem}

By taking $\mathcal{H}=\mathrm{End}(\mathbb{C}^{d_\Gamma})$ and $m=(d_\Gamma)^2$ in Theorem \ref{thm-uniform-obs} and utilizing the unitary property of the generalized Bloch transform (Theorem \ref{thmunitB}), we proved the desired observability Theorem \ref{thm-main-obs} with a control constant $C=C(M,\Omega,T,d_\Gamma)>0$, which depends only on $M,\Omega,T$ and the maximum dimension of the irreducible unitary representations of $\Gamma$.

%\begin{theorem}\label{thm-main-obs} Under the conditions of Theorem \ref{thm-uniform-obs}, we can find a constant $C=C(M,\Omega,T,d_\Gamma)>0 $ such that for any $u \in L^2(X)$, \begin{equation*} \|u\|_{L^2(X)}^2 \leq C\int_{0}^{T}\|\mathrm{e}^{it \Delta_g}u\|^2_{L^2(\pi^{-1}(\Omega))} \, dt.  \end{equation*} \end{theorem}

\begin{remark}
Some remarks are in order.
\begin{itemize}
\item[(a)] We may consider the $\Gamma$-periodic Schrödinger operator $H=-\Delta+P_1+P_0$, where $H$ is self-adjoint on $L^2(X)$ and $P_i$ are differential operators of order $i$ on $X$. By Theorem \ref{generalizedmicrolocalcontrolonhilbertbundle}, Theorem \ref{thm-main-obs} remains valid for $H$.
\item[(b)] In the spirit of the work of Dyatlov, Jin, and Nonnenmacher \cite{dyatlov2022control}, we expect that all results above still hold if $M$ is a negatively curved surface. The key point is that one can select uniform constants in the semiclassical calculus for scalar symbols on the flat Hilbert bundle; see Section \ref{sec:quantization on flat bundles}.
\item[(c)] With the restriction of type I groups, some interesting covering transformation groups (\emph{e.g.}, the Fuchsian subgroup) are not covered. There are two main difficulties in extending our results to non type I groups: 
\begin{itemize}
    \item the mapping properties of the generalized Bloch transform for non type I groups are hard to study; we refer to \cite{nagy2024hyperbolic, katsuda2025extension} for some recent progress on Bloch analysis for non type I groups;
    \item when the dual group $\widehat{\Gamma}$ contains infinite-dimensional irreducible representations, it is unclear how to apply the usual compactness argument to remove the low-frequency error term from the semiclassical control estimate.
\end{itemize}
%If $\Gamma$ is a type $\mathrm{I}$ group, then by Thomas’s theorem \cite{thoma1964uber,tonti2019thoma} the observability statement in Theorem~\ref{thm-main-obs} reduces to the  case. The only difference is that the resulting observability constants may differ, reflecting the dependence of the constant in \eqref{eqn-obs}.
\end{itemize}
\end{remark}

By Theorem \ref{highfrequencycontrol}, we obtain a uniform semiclassical control estimate for all Riemannian covering spaces of a compact hyperbolic surface. However, the observability inequality in Theorem \ref{thm-main-obs} holds only for the type \rm{I} group covering. Therefore, it is reasonable to propose the following conjecture:

\begin{conjecture}
Let $X$ be a $\Gamma$-covering of a compact hyperbolic surface $M$, $S \subset X$ be a nonempty, $\Gamma$-periodic, open subset. Then, for any $T>0$, there exists a constant $C=C(X,S,T)>0$ such that the observability inequality \eqref{observability} holds. Furthermore, denoted by $\Omega=\pi(S)$, then the constant $C$ in \eqref{observability} can be chosen such that it depends only on $M,\Omega$ and $T$.
\end{conjecture}

\subsection{Comments on history}\label{review}

In this section, we review observability results for Schrödinger type equations. We begin with the compact setting and then discuss recent developments on non-compact domains. 

The observability problem on compact Riemannian manifolds has been extensively studied over the past decades; we refer to the surveys \cite{laurent2014internal, macia2015high} for comprehensive overviews. As mentioned earlier, Lebeau \cite{lebeau1992controle} showed that the Geometric Control Condition (GCC) is sufficient to guarantee observability of the Schrödinger equation at any time $T>0$. The GCC was originally introduced for wave equations by Rauch and Taylor \cite{rauch1974exponetial}, and by Bardos, Lebeau and Rauch \cite{bardo1992}. However, GCC is in general not a necessary condition for observability of Schrödinger type equations. 

A notable example is provided by the $d$-dimensional rational flat torus $\T^d$. For the Schrödinger equation on $\T^d$, observability from any nonempty open subset of $\T^d$, which does not necessarily satisfy GCC, was established by Haraux and Jaffard \cite{haraux1989series, jaffard1990controle} for $d=2$, and by Komornik \cite{komornik1992exact} for $d\geq 2$. These results were later extended in several directions by: adding a $C^\infty$ potential by Burq and Zworski \cite{burq2012control}, adding a broad class of potentials including continuous ones by Anantharaman and Macià \cite{anatharaman2014semiclassical},  adding a $L^2$ potential by Bourgain, Burq and Zworski \cite{bourgain2013control} for $d=2$, extending to rough observable sets by Burq and Zworski \cite{burq2019roughcontrols} for $d=2$. For those interested in the observability of magnetic Schrödinger equations, we refer to the recent work of Le Balc'h, Niu, and Sun \cite{balch2025geometric}, which established an almost-sharp Magnetic Geometric Control Condition on the two-dimensional flat torus. One may also see \cite{anatharaman2015semiclassical,macia2021observabilityresults,tao2021exactcontrol} for other developments. 

Another important geometric setting is that of compact hyperbolic surfaces. In this case, Jin \cite{jin2018control} proved observability of the Schrödinger equation from any nonempty open subset at any time $T>0$, relying on the deep result of the semiclassical control estimate established by Dyatlov and Jin \cite{dyatlov2018semiclassical}. This was subsequently generalized to Anosov surfaces by Dyatlov, Jin, and Nonnenmacher \cite{dyatlov2022control}. For the disk, Anantharaman, Léautaud and Macià \cite{anantharaman2016delocalization,anantharaman2016wigner} characterized open subsets from which observability holds. 

Recently, observability of Schrödinger equations on the Euclidean space $\R^d$ has attracted considerable interest. In the absence of spatial compactness, the standard techniques break down and new challenges appear. In dimension one, Huang, Wang and Wang \cite{huang2022obs} and Martin, Pravda-Starov \cite{martin2021geometric} independently proved that the free Schrödinger equation is observable in some time $T>0$ from a subset $S\subset \R$ if and only if $S$ is \emph{thick}, meaning that
\begin{equation*}
    \exists \gamma, L>0,\ \text{ s.t. } \forall x\in\R,\  |S\cap(x+[0,L])|\ge \gamma.
\end{equation*}
Here $|E|$ denotes the Lebesgue measure of $E\subset\R$. In \cite{huang2022obs}, the authors further considered Schrödinger equations with confining potential $V=|x|^{2m}, m\in \mathbb{N}^+$, and showed that observability holds from $S$ in some time when $m=1$ (respectively, in any time when $m\ge 2$) if and only if $S$ is \emph{weakly thick}, that is,
\begin{equation*}
    \liminf_{x\to +\infty} \frac{|S\cap [-x,x]|}{|[-x,x]|}>0.
\end{equation*}
A generalization of this condition in higher dimensions, together with its necessity, was also given in \cite{martin2021geometric}. These results rely heavily on harmonic analysis tools. In \cite{prouff2025obs}, Prouff considered a large class of subquadratic confining potentials and derived a sufficient and almost necessary characterization of observable sets from which observability holds. His method relies on the establishment of a uniform Egorov's theorem in semiclassical analysis. In \cite{su2025quantitative}, Su, Sun and Yuan proved quantitative observability for one-dimensional Schrödinger equations with $V\in L^\infty$ from thick sets. The key idea in \cite{su2025quantitative} is to establish a spectral inequality at low frequency and a resolvent estimate at high frequency, and then merge them in a proper way.

There is another route to observability on non-compact domains that exploits symmetry. Given a Riemannian manifold $(X,g)$ and a symmetry group $\Gamma$, one may study observability from $\Gamma$-periodic sets. The central idea is to use the symmetry to reduce the problem to a family of observability estimates on compact spaces. This philosophy originates from physics, where it appears in the study of Schrödinger equations with periodic potentials and is known as Floquet–Bloch theory \cite{floquet1883sur, bloch1929quantenmechanik}. Using Bloch theory on $\R^d$, Täufer \cite{taufer2023controllability} applied Ingham-type inequalities to establish observability for the free Schrödinger equation from any nonempty periodic open set. Although not stated explicitly, this  result can also be deduced from the proof of Proposition 4 in \cite{wunsch2017periodic} with the help of Bloch theory. To our best knowledge, Wunsch \cite{wunsch2017periodic} was the first to apply Bloch theory to obtain polynomial energy decay for wave equations with periodic damping. In dimension two, Le Balc'h and Martin \cite{balch2023obs} extended these results to periodic measurable sets with positive measure, relying on semiclassical defect measures and adaptations of strategies developed in \cite{burq2012control, bourgain2013control, burq2019roughcontrols}. Similar strategy was used by Niu and Zhao \cite{niu2025obs} to prove observability in the semi-periodic setting. These works naturally raise the question of whether such techniques can be extended to other symmetric spaces than $\R^d$.

Kocábová and Šťovíček \cite{kocabova2008generalized} generalized Bloch theory to type I groups. More recently, hyperbolic lattices have attracted growing interest in physics, motivated by advances in engineered photonic structures and circuit-based realizations; see \cite{hyperboliclattices2019, hyperboliccircuit2022}. In parallel, Maciejko and Rayan \cite{maciejko2021hyperbolic,  maciejko2022automorphic} developed a mathematical theory of hyperbolic bands. Based on Magee's work \cite{magee2022random, magee2025random}, Nagy and Rayan \cite{nagy2024hyperbolic} introduced the hyperbolic Bloch transform and proved its injectivity and asymptotic unitary property. However, the problem regarding the $L^2$-boundedness of the hyperbolic Bloch transform is still open; see \cite[Remark 3.8]{nagy2024hyperbolic}. 

Non-abelian Bloch theory naturally leads to the study of twisted Bochner–Laplace operators on high-dimensional unitary flat bundles, including bundles with separable Hilbert fibers. Independently, geometric quantization has motivated the study of spectral theory on unitary flat bundles. Two independent groups, Ma, Ma \cite{mama2024semiclassical} and Cekić, Lefeuvre \cite{cekiclefeuvre2024semiclassical}, have developed a mixture of semiclassical and geometric quantization and proven the equidistribution property of high-frequency eigensections. The uniform constants for any finite-dimensional unitary representation in the semiclassical analysis on flat bundles play an important role in Ma and Ma \cite{mama2024semiclassical}.

\subsection{Notations}

We collect some notations and conventions used throughout this work:
\begin{itemize}
\item We use $C(\theta_1,\cdots,\theta_n)$ to denote a constant depending only on the parameters $\theta_1,\cdots,\theta_n$, whose value may change from line to line.
\item Given a Riemannian manifold $(M,g)$, we denote by $\pi_M:\widetilde{M} \rightarrow M$ the Riemannian universal covering map, and by $\pi_1(M)$ the fundamental group of $M$.
\item Unless otherwise specified, we write $A=\mathcal{O}(f(h))_X$ for some $f:\R\rightarrow \R$, where $X$ is a normed space, if there is a constant $C$ independent of $(\mathcal{H},\rho,h)$ such that $\|A\|_X \leq Cf(h)$.
\item We use the \emph{Einstein summation convention}.
\end{itemize}

The remainder of the paper is organized as follows. In Section \ref{sec:bloch theory}, we develop a generalized Bloch theory. In Section \ref{sec: dynamics of geodesic flow}, we review the geodesic flow and the anisotropic symbol classes on compact hyperbolic surfaces. In Section \ref{sec:quantization on flat bundles}, we introduce the quantization of scalar symbols and Fourier integral operators on flat Hilbert bundles. Section \ref{sec:uniform semiclassical control} is devoted to proving the uniform semiclassical control estimate on flat Hilbert bundles. Finally, we prove the observability inequality on non-compact hyperbolic surfaces with type I symmetry groups in Section \ref{sec: proof of observability}.

\section{A generalized Bloch theory}\label{sec:bloch theory}

In this section, we develop a generalized Bloch theory. While much of the material can be found in the existing literature (see \cite{nagy2024hyperbolic, kocabova2008generalized}), we present an almost self-contained treatment for the reader’s convenience, including all necessary details.

In this section, we always assume that $\mathcal{H}$ is a separable Hilbert space with a countable orthonormal basis $\{e_{i}\}_{i\in I}$ and $(M,g)$ is a compact connected Riemannian manifold.

\subsection{Laplacian on flat Hilbert bundles}\label{subsec: twisted laplacian}

A \emph{flat Hilbert bundle} is a Hilbert bundle with a flat metric-compatible connection. According to the Riemann–Hilbert correspondence (see \cite[Section 1.2-1.4]{Kobayashi87vectorbundle} for the finite-dimensional case), any unitary representation $\rho: \pi_{1}(M) \to \mathrm{U}(\mathcal{H})$ corresponds to a flat Hilbert bundle $F^\rho$ over $M$, which is constructed as follows. Let
\begin{equation}\label{defFrho}
    F^\rho:=\widetilde{M} \times_{\rho} \mathcal{H}
\end{equation}
denote the quotient of $\widetilde{M} \times \mathcal{H}$ by the action of $\pi_1(M)$ given by
\begin{equation}\label{defofflatbundle}
    \gamma: (\widetilde{x},v) \in \widetilde{M}\times \mathcal{H} \mapsto (\gamma \widetilde{x}\,,\,\rho(\gamma )v) \in \widetilde{M}\times \mathcal{H} , \qquad \gamma \in \pi_1(M).
\end{equation}
Here we identify $\pi_1(M)$ as the deck transformation group of $\pi_M:\widetilde{M}\to M$. 

We now check that $F^{\rho}$ defined in \eqref{defFrho} admits a flat Hilbert bundle structure over $M$. The bundle projection is defined as
\begin{equation}
    \pi_\rho:F^\rho \rightarrow M, \qquad [\widetilde{x},v] \mapsto \pi_M(\widetilde{x}).
\end{equation}
For any $x\in M$, choose a neighborhood $U\subset M $ of $x$ and a local lift $\iota_U:M\rightarrow \widetilde{M}$ such that $\pi_M\circ \iota_U = \mathrm{Id}_U$. The local trivialization is defined as
\begin{equation}
    \Phi_U:\pi_{\rho}^{-1} (U) \rightarrow U\times \mathcal{H}  , \qquad [\iota_U(y),v] \mapsto (y,v) .
\end{equation}
It is easy to see that the transition function between two local trivializations is constant (in fact, $\rho(\gamma)$ for some $\gamma\in \pi_1(M)$). The fiberwise inner product $h_x^\rho$ at $x\in M$ is defined as:
\begin{equation}\label{hilbertfiber}
    h_x^\rho \big( [\iota_U(x),v], [\iota_U(x),w] \big) := \langle v,w\rangle_{\mathcal{H}}.
\end{equation}
Note that the inner product \eqref{hilbertfiber} is well-defined because the representation $\rho$ is unitary. Therefore, $F^\rho$ defined in \eqref{defFrho} is a Hilbert bundle. We now describe the corresponding flat connection and twisted Laplacian from both global and local perspectives.

For the theory of differential and integration of Hilbert-space-valued functions on manifolds, we refer to \cite{HillePhillips1974semigroups, Kato1966pertubation}. Let $C^{\infty}(\widetilde{M} ; \mathcal{H})$ be the space of smooth $\mathcal{H}$-valued functions. The space of smooth sections of  $F^{\rho}$, denoted by $C^{\infty}(M; F^{\rho})$, is identified with the space of $\rho$-equivariant $\mathcal{H}$-valued functions:
\begin{equation*}
  C^{\infty}_{\rho}(\widetilde{M} ; \mathcal{H}) := \{f \in C^{\infty}(\widetilde{M} ; \mathcal{H}) : f(\gamma x) = \rho(\gamma)f(x), \ \forall \gamma \in \pi_1(M)\}.  
\end{equation*}
In the same manner, for $k\in \N$, we define the $\mathcal{H}$-valued $k$-form space $\Omega^{k}(\widetilde{M}; \mathcal{H})$, the $\rho$-equivariant $\mathcal{H}$-valued $k$-form space $\Omega^{k}_{\rho}(\widetilde{M}; \mathcal{H})$, and the $F^{\rho}$-valued $k$-form space $\Omega^{k}(M; F^{\rho})$. When $\rho=\rho_{\mathrm{triv}}$ is the trivial representation, the corresponding function and $k$-form spaces $C^{\infty}_{\rho}(\widetilde{M} ; \mathcal{H})$ and $\Omega^{k}_{\rho}(\widetilde{M}; \mathcal{H})$ are denoted by $C^{\infty}_{\Gamma}(\widetilde{M}; \mathcal{H})$ and $\Omega^{k}_{\Gamma}(\widetilde{M}; \mathcal{H})$, respectively. 

%The $\pi_1(M)$-equivalent function and differential form spaces are identified as the function and differential form spaces on $M$.

For any
\begin{equation*}
  u = u^i \otimes e_i \in C^{\infty}(\widetilde{M}) \widehat{\otimes} \mathcal{H} \cong  C^{\infty}(\widetilde{M}; \mathcal{H})   ,
\end{equation*}
where $\widehat{\otimes}$ denotes the completed tensor product, we define the operator $d \otimes \mathrm{Id}_{\mathcal{H}}$ by:
\begin{equation}\label{globalconnection}
    (d \otimes \mathrm{Id}_{\mathcal{H}})u := (du^i) \otimes e_i \in \Omega^1(\widetilde{M}; \mathcal{H}).
\end{equation}
Since $d$ is invariant under the $\pi_{1}(M)$-action, \eqref{globalconnection} induces a map from $C^{\infty}_{\rho}(\widetilde{M} ;\mathcal{H})$ to $\Omega^{1}_{\rho}(\widetilde{M}; \mathcal{H})$. Consequently, we obtain the flat connection $\nabla^{\rho}: C^{\infty}(M;F^{\rho}) \to \Omega^1(M;F^{\rho})$. The twisted Bochner–Laplace operator (short for twisted Laplacian) $\Delta^{\rho}$ is then defined as:
\begin{equation}
\Delta^\rho := \mathrm{tr}(\nabla^\rho\circ \nabla^\rho ): C^{\infty}(M;F^{\rho}) \to C^{\infty}(M;F^{\rho}),
\end{equation}
which coincides with $\Delta_g \otimes \mathrm{Id}_{\mathcal{H}}$ on $C_{\rho}^{\infty}(\widetilde{M}; \mathcal{H})$. 

The above provides a global description of the flat connection $\nabla^{\rho}$ and the twisted Laplacian $\Delta^{\rho}$. In the following, we characterize the flat connection and the twisted Laplacian locally. 

Let $\{e_i(x)\equiv e_i:M\rightarrow \mathcal{H}\}_{i\in I}$ be the orthonormal frame field of $M\times \mathcal{H}$. For any local trivialization $(\Phi_U,U)$ of $F^\rho$, $\{\Phi_U^{-1}(e_i)\}_{i\in I}$ is an orthonormal frame field of $\pi_{\rho}^{-1}(U)$. The flat connection $\nabla^\rho$ is defined so that $\Phi_U^{-1}(e_i)$ is parallel for every $i \in I$, \emph{i.e.}, 
\begin{equation}
    \nabla^\rho \Phi_U^{-1}(e_i) =0, \qquad \forall i\in I.
\end{equation}
Then, $\nabla^\rho$ is compatible with the inner product $h^\rho$: for any vector fields $u = \Phi_U^{-1}(e_i) u^i, \, v = \Phi_U^{-1}(e_i) v^i$ and $w = \Phi_U^{-1}(e_i) w^i$ on $U$, we have
\begin{equation*}
\begin{aligned}
    u (h^\rho (v,w)) &= u^j \partial_j ( \delta_{ik} v^iw^k ) \\
    &=  \delta_{ik} (\partial_j v^i u^j) w^k + \delta_{ik} v^i( \partial_jw^k u^j  ) \\
    &= h^\rho( \Phi_U^{-1}(e_i) \partial_j v^i u^j, \Phi_U^{-1}(e_k)w^k  ) + h^\rho( \Phi_U^{-1}(e_i)v^i, \Phi_U^{-1}(e_k)\partial_j w^ku^j  )  \\
    & = h^{\rho} (\nabla^{\rho}_u v ,w) + h^\rho( v,\nabla^\rho_u w ).
\end{aligned}
\end{equation*}
Therefore, $F^\rho$ is a Hilbert bundle over $M$ with a flat connection $\nabla^\rho$ such that $\nabla^\rho $ is $h^\rho$-compatible.

From the construction above, we have the following proposition:
\begin{proposition}
\label{defoflocaltrivialization}
Let $\pi:F\rightarrow M$ be a flat Hilbert bundle with fiber isomorphic to $\mathcal{H}$. Then, there exists a finite collection $\{(\Phi_\alpha, U_\alpha, \varphi_\alpha)\}_{\alpha \in \Lambda}$ such that:
\begin{itemize}
    \item $\{(U_\alpha, \varphi_\alpha)\}_{\alpha \in \Lambda}$ is an atlas of $M$.
    \item For each $\alpha \in \Lambda$, $\Phi_\alpha : \pi^{-1}(U_\alpha) \rightarrow U_\alpha \times \mathcal{H}$ is a local trivialization of $F$, and $\{ \Phi_\alpha^{-1}(e_i)\}_{i\in I}$ is a local parallel orthonormal frame of $F$.
\end{itemize}
We refer to the triple $(\Phi_\alpha, U_\alpha, \varphi_\alpha)$ as a parallel orthonormal local trivialization chart of $F$.
\end{proposition}

Using the local parallel orthonormal frame of $F^\rho$, the twisted Laplacian $\Delta^\rho$ can be diagonalized as follows: for any smooth section $u\in C^{\infty}(M;F^\rho)$, we can locally express $u$ as $u=u^j \Phi_\alpha^{-1}(e_j)$, then
\begin{equation}\label{diagofLaplacian1}
\begin{aligned}
    \Delta^\rho u &= \mathrm{tr}\,(\nabla^\rho \circ \nabla^\rho) u = \sum_i \nabla^\rho \circ \nabla^\rho (v_{\alpha,i},v_{\alpha,i}) u  \\
    &=\sum_{i\in I} \left(\nabla^\rho_{v_{\alpha,i}} \circ \nabla^\rho_{v_{\alpha,i}} - \nabla^\rho_{\nabla^{\mathrm{LC}}_{v_{\alpha,i}} v_{\alpha,i}} \right) \big(u^j \Phi_\alpha^{-1}(e_j) \big)   \\
    &=(\Delta_g u^j )\Phi_\alpha^{-1}(e_j),
\end{aligned}
\end{equation}
where $\{v_{\alpha,i}\}_{i\in I}$ is a local orthonormal frame of $TM$ and $\nabla^{\mathrm{LC}}$ denotes the Levi-Civita connection on $TM$.

We review the deformation of representations. To simplify, we only consider the finite-dimensional unitary representation $\rho: \pi_1(M)\to \mathrm{U}(n)$. Since $\pi_1(M)$ is generated by $2g$ generators 
\begin{equation*}
    a_1,\cdots,a_g,b_1,\cdots,b_g
\end{equation*}
under the relation 
\begin{equation*}
    [a_1,b_1]\cdots[a_g,b_g]=1,
\end{equation*}
the representation variety $\mathrm{Hom}(\pi_1(M),\mathrm{U}(n))$ can be identified as a compact algebraic variety of $\mathrm{U}(n)^{2g}$. Thus, the topology of the moduli space of representations $\mathrm{Hom}(\pi_1(M),\mathrm{U}(n))/\mathrm{U}(n)$ is naturally inherited from $\mathrm{U}(n)$.  If $[\rho_i]\to [\rho]$ in $\mathrm{Hom}(\pi_1(M),\mathrm{U}(n))/\mathrm{U}(n)$ as $i \rightarrow \infty$, then we can choose representative elements $\rho_i$ and $\rho$ of $[\rho_i]$ and $[\rho]$, respectively, such that 
\begin{equation*}
    \lim_{i\rightarrow \infty} \rho_i(a_j) = \rho(a_j) \quad \mathrm{and} \quad \lim_{i\rightarrow \infty} \rho_i(b_j) = \rho(b_j) \qquad \mathrm{in}\ \mathrm{U}(n)
\end{equation*}
for any $j=1,\cdots,g$. For more on the moduli space of unitary representations of the surface group, we refer to Atiyah and Bott \cite{atiyah1983yang} and Goldman \cite{goldman1984symplectic}. In the remainder of this subsection, we assume that the representation $\rho$ is finite-dimensional. 

Fix a reference point $x_0 \in \widetilde{M}$. There exists a smooth map $\mathfrak{U}_{\rho}: \widetilde{M} \to \mathrm{U}(\mathcal{H})$, defined by parallel transport on the flat bundle $(F^{\rho}, \nabla^{\rho})$ from $x_0$, such that
\begin{equation}\label{equivalent}
\mathfrak{U}_{\rho}(\gamma x) = \rho(\gamma)\mathfrak{U}_{\rho}(x), \qquad  \forall \gamma \in \pi_1(M), \, x \in \widetilde{M}.
\end{equation}
More precisely, $\mathfrak{U}_{\rho}$ is the local solution to the following initial value problem:
\begin{equation*}
\begin{aligned}
&d\mathfrak{U}_{\rho} + \pi_M^*(A_{\rho})\mathfrak{U}_{\rho} = 0, \\
&\mathfrak{U}_{\rho}(x_0) = \mathrm{Id}_{\mathcal{H}}.
\end{aligned}
\end{equation*}
Here, $A_{\rho}$ denotes the connection 1-form for the flat connection $\nabla^{\rho}$. In the case where $M$ is a compact hyperbolic surface, the construction of $\mathfrak{U}_{\rho}$ is referred to \cite{nagy2024hyperbolic}. 

We define the Laplacian $\Delta_\rho$ on the $\Gamma$-equivalent space $L^2_{\Gamma}(\widetilde{M};\mathcal{H})$:
\begin{equation}\label{diaglaplacian2}
   \Delta_{\rho}f(x):=\mathfrak{U}_{\rho}^{-1}(x)(\Delta_g\otimes \mathrm{Id}_{\mathcal{H}})(\mathfrak{U}_{\rho}f)(x), \qquad  \forall f\in C^{\infty}_{\Gamma}(\widetilde{M};\mathcal{H}), \, x\in \widetilde{M}.
\end{equation} 
Since $\mathfrak{U}_{\rho}$ gives an isometry between $L^2_{\Gamma}(\widetilde{M};\mathcal{H}) $ and $L^2_{\rho}(\widetilde{M};\mathcal{H})$, the twisted Laplacian $\Delta^\rho$ is unitarily equivalent to the Laplacian $\Delta_{\rho}$ defined as \eqref{diaglaplacian2}, which is considered as an operator on $L^2(M;\mathcal{H})$. Moreover, the following continuity holds:

\begin{proposition}\label{convergeofresolvent}
Assume that $\rho_i \to \rho$ in $\mathrm{Hom}(\pi_1(M), \mathrm{U}(\mathcal{H}))$. For any $\varepsilon > 0$, $(z-\Delta_{\rho_i})^{-1}$ converges to $(z-\Delta_{\rho})^{-1}$ uniformly in $|\mathrm{Im}\, z| \geq \varepsilon$:
\begin{equation*}
\lim_{\rho_i\rightarrow \rho} \sup_{|\mathrm{Im}\, z| \geq \varepsilon}\left\| (z - \Delta_{\rho_i})^{-1} - (z - \Delta_{\rho})^{-1} \right\|_{L^2(M;\mathcal{H})\rightarrow L^2(M;\mathcal{H})} =0.
\end{equation*}
\end{proposition}
\begin{proof}
    The proof is straightforward by using local coordinates and patching argument. We refer to \cite[Proposition~3.2]{balch2023obs} for the Euclidean space case.
\end{proof}

\subsection{Non-commutative Bloch transform}\label{subs:noncBloch}

In this section, we assume that $\pi:X\to M$ is an arbitrary Riemannian covering map. By the universal property of the universal cover, we have $\pi_M=\pi\circ \pi_{X}$. We write 
\begin{equation*}
    M=\Gamma_{M}\backslash\widetilde{M} \quad\text{and}\quad X=\Gamma_{X}\backslash\widetilde{M},
\end{equation*}
where $\Gamma_{X}\leqslant \Gamma_{M}\leqslant \mathrm{Iso}(\widetilde{M})$ are discrete subgroups of the isometry group of $\widetilde{M}$. We fix base points $x_{0} \in M$, $x_{0}^{\prime} \in \pi^{-1}(x_0)$, and $\widetilde{x}_0\in \pi_X^{-1}(x_0')$. Then, $\Gamma_{M}\cong \pi_{1}(M,x_0)$ and $\Gamma_{X}\cong \pi_{1}(X,x_0^{\prime})$. Let 
\begin{equation}\label{defineV}
    V:=\Gamma_{X}\backslash\Gamma_{M}=\left\{\left[\gamma\right]:=\Gamma_{X} \cdot \gamma \ : \ \gamma\in \Gamma_{M}\right\}
\end{equation}
be the set of right cosets of $\Gamma_{X}$ in $\Gamma_M$. Since $\Gamma_{M}$ is a countable set, $V$ is also countable. Moreover, $\pi$ is normal if and only if $V$ forms the deck transformation group of $\pi$. 

%{\color{red}[Yunlei: I added “assume” here, this is our assumption right?]}
%We now generalize the Non-commutative Bloch transform, introduced by Nagy and Rayan \cite{nagy2024hyperbolic}. 

Let $\mathcal{H}=\ell^2(V)$ denote the Hilbert space of square-summable functions on $V$. Let $\rho_{\pi}: \pi_1(M)\cong  \Gamma_{M}\to \mathrm{U}(\mathcal{H})$ be the quasiregular representation on $\mathcal H=\ell^{2}(V)$, defined by
\begin{equation*}
    \rho_{\pi}(\eta)(\delta_{[\gamma]}):=\delta_{[\gamma \eta^{-1} ]}, \qquad \forall \,\eta \in \Gamma_M,\, [\gamma]\in V,
\end{equation*}
where $\delta_{[\gamma]}\in \ell^2(V)$ is the Kronecker delta function supported at $[\gamma] \in V$. Let $F^{\rho_\pi}$ denote the flat Hilbert bundle over $M$ associated with the representation $\rho_\pi$, constructed as in Section \ref{subsec: twisted laplacian}.

%Notice that $\Gamma_{M}$ has natural right action on $V$ given by
%$$V\times \Gamma_{M} \to V,\quad [\eta]\cdot \gamma=[\eta\gamma].$$
%This gives a group homomorphism $\Gamma_{M}\to \mathrm{Sym}(V)$, where $\mathrm{Sym}(V)$ denotes the permutation group of $V$.
%\begin{remark}
%For any compact $M$, the $\mathrm{Hom}(\pi_1(M),S_{n})$ gives the $n$-cover random cover model with uniform probability measure.
%\end{remark}
%{\color{red}[Yunlei: This remark confuses me. we can remove it if it is not used later.]}

%We then define the associated flat $\ell^{2}(V)$-bundle over $M$ by
%\begin{equation*}
%    F^{\rho_{\pi}} := \Gamma_{M}\backslash\big(\widetilde M\times \ell^{2}(V)\big),
%\end{equation*}
%where $\Gamma_{M}$ acts via $\gamma\cdot (x,f) := \big(\gamma x,\rho_{\pi}(\gamma)f\big)$.

We now introduce the non-commutative Bloch transform. For $\psi \in L^2(X)$, we can view it as a $\Gamma_X$-invariant function on $\widetilde{M}$, that is, $\psi(\gamma x)=\psi(\gamma^{\prime}x)$ whenever $[\gamma]=[\gamma^{\prime}] \in V$.

\begin{definition}
For any $\psi :X\rightarrow \mathbb{C}$, the non-commutative Bloch transform $\mathcal{B}_{\mathrm{NC}} \psi$ is defined as a function on $\widetilde{M} \times \mathcal{H}$:
\begin{equation}\label{defofnonbloch}
(\mathcal{B}_{\mathrm{NC}}\psi)(x):=\sum_{[\gamma]\in V}\psi(\gamma x)\delta_{[\gamma]}, \qquad \forall \,\psi \in L^2(X),\,x\in \widetilde{M} .
\end{equation}
\end{definition}

%\textcolor{red}{(XF: I changed this definition so that it is consistent to that of the paper by Nagy and Ryan, and consistent to the classical Bloch transform on Euclidean space, see Example 2.4.)}

For any $\eta \in \Gamma_{M}$ and $x\in \widetilde{M}$, we have
\begin{equation*}
\begin{aligned}
    (\mathcal{B}_{\mathrm{NC}}\psi) (\eta x) &=\sum_{[\gamma] \in V}\psi(\gamma  \eta x)\rho_{\pi}(\eta)(\delta_{[ \gamma \eta]}) =\rho_{\pi}(\eta)\left( \sum_{[\gamma] \in V}\psi(\gamma  \eta x) \delta_{[\gamma  \eta ]}) \right) \\
    &=\rho_{\pi}(\eta)\big((\mathcal{B}_{\mathrm{NC}}\psi)(x)\big).
\end{aligned}
\end{equation*}
Therefore, $\mathcal{B}_{\mathrm{NC}}\psi$ is $\Gamma_M$-equivariant and can be identified as a $F^{\rho_\pi}$-valued section over $M$. With this identification, we have
\begin{equation*}
    (\mathcal{B}_{\mathrm{NC}}\psi)(x)=\sum_{[\gamma] \in V} \psi(\gamma x) \Phi^{-1}_U (\delta_{[\gamma]})
\end{equation*}
locally. By \eqref{diagofLaplacian1}, we get
\begin{equation}\label{nonblochtwisted}
    (\Delta^{\rho_\pi} \circ\mathcal{B}_{\mathrm{NC}}\psi)(x)=\sum_{[\gamma] \in V} \Delta_g \psi(\gamma x) \Phi^{-1}_U (\delta_{[\gamma]}) = (\mathcal{B}_{\mathrm{NC}}\circ \Delta_g  \psi)(x).
\end{equation}

The following proposition lists some properties of the non-commutative Bloch transform:

\begin{proposition}
For any $a\in C (M)$, $\pi^*a=a\circ \pi \in C (X)$. Let $M_a$ and $M_{\pi^*a}$ be the multiplication operators on $L^2(M;F^{\rho})$ and $L^2(X)$, respectively. We have 
\begin{equation*}
    (\Delta^{\rho_\pi} \circ\mathcal{B}_{\mathrm{NC}}\psi)(x)= (\mathcal{B}_{\mathrm{NC}}\circ \Delta_g  \psi)(x),
\end{equation*}
\begin{equation*}
    (M_a \circ\mathcal{B}_{\mathrm{NC}}\psi)(x)=(\mathcal{B}_{\mathrm{NC}}\circ M_{\pi^*a}  \psi)(x),
\end{equation*}
and $\mathcal{B}_{\mathrm{NC}}:L^2(X)\to L^2(M;F^{\rho_{\pi}})$ is an isometry, \emph{i.e.},
\begin{equation}\label{unitaryofnonbloch}
\|\mathcal{B}_{\mathrm{NC}}\psi\|_{L^2(M;F^{\rho_{\pi}})}^2=\|\psi\|_{L^2(X)}^2.
\end{equation}
\end{proposition}

%The tensor extension $\Delta_{X} \otimes \mathrm{Id}_{\mathcal{H}}: C^{\infty}(X,\mathcal{H})\to C^{\infty}(X,\mathcal{H})$ descends to the bundle Laplacian $\Delta^{\rho}: C^{\infty}(M,\mathcal{E}^{\Gamma})\to C^{\infty}(M,\mathcal{E}^{\Gamma})$. Nagy and Rayan \cite[Lemma 5.1]{nagy2024hyperbolic} show Thus, problems for the Laplacian on the non-compact manifold $X$ can be transferred to the corresponding problems for the Laplacian acting on the flat Hilbert bundle $\mathcal{E}^\Gamma$ over the compact base $M$.

\subsection{Generalized Bloch transform}\label{gbtrans}

Although the non-commutative Bloch transform maps functions on the covering space $X$ to sections of the vector bundle $F^{\rho_{\pi}}$ over the base manifold $M$, the fiber of $F^{\rho_{\pi}}$ (isomorphic to $\ell^2(V)$) may be infinite-dimensional, posing challenges for further analysis. 
To address this issue, we compose the fiber-wise Fourier transform with the non-commutative Bloch transform to obtain sections of finite-dimensional Hermitian bundles over the representation space. This results in the so-called generalized Bloch transform, which decomposes the original function into a direct integral of sections of finite-dimensional Hermitian bundles.

To define the generalized Bloch transform, we assume that the covering map $\pi:X\rightarrow M$ is normal. It follows that the set $V$ defined in \eqref{defineV} is isomorphic to the deck transformation group of $\pi$, which is denoted by $\Gamma$. For each $n \in \mathbb{N}^+$, let $\mathrm{Hom}_{\mathrm{irr}}(\Gamma,\mathrm{U}(n))$ be the space of all $n$-dimensional irreducible unitary representations of $\Gamma$. Let
\begin{equation}
    \mathcal{M}^n_{\Gamma}: = \mathrm{Hom}_{\mathrm{irr}}(\Gamma,\mathrm{U}(n))/\mathrm{U}(n)
\end{equation}
be the space of all $n$-dimensional irreducible unitary representations of $\Gamma$ up to unitary equivalence, and
\begin{equation*}
   \mathcal{M}_{\Gamma} : = \bigcup_{n=1}^{\infty} \mathcal{M}^n_{\Gamma}.
\end{equation*}
Let $(\mathrm{End}(\mathbb{C}^n),\langle \cdot,\cdot\rangle_{\mathrm{HS}})$ be the Hilbert space consisting of all endomorphisms of $\mathbb{C}^n$ equipped with the Hilbert-Schmidt inner product $\langle \cdot,\cdot\rangle_{\mathrm{HS}}$, that is,
\begin{equation}
    \langle A, B\rangle_{\mathrm{HS}}: = \mathrm{tr}\,(AB^*).
\end{equation}
For all $(\rho,A)  \in \mathrm{Hom}_{\mathrm{irr}}(\Gamma, \mathrm{U}(n)) \times \mathrm{End}(\mathbb{C}^n)$ and $U \in \mathrm{U}(n)$, we define 
\begin{equation}
    U \cdot (\rho, A) := ( U\rho U^*, UA U^* ),
\end{equation}
which gives a free, linear action of projective unitary group $\mathrm{PU}(n)$ on 
$\mathrm{Hom}_{\mathrm{irr}}(\Gamma, \mathrm{U}(n)) \times \mathrm{End}(\mathbb{C}^n) $. Thus,
\begin{equation} \label{eq:On}
\mathscr{O}^n := \big( \mathrm{Hom}_{\mathrm{irr}}(\Gamma, \mathrm{U}(n)) \times \mathrm{End}(\mathbb{C}^n) \big) \big/ \mathrm{PU}(n) \longrightarrow \mathcal{M}^n_{\Gamma}
\end{equation}
defines a Hermitian bundle over $\mathcal{M}^n_{\Gamma}$. Let $\Gamma(\mathscr{O}^n)$ be the space of sections of $\mathscr{O}^n$. Let $\mathscr{O}$ be the Hilbert-sheaf of $\mathcal{M}_{\Gamma}$ induced by the vector bundles $\mathscr{O}^n$, and $\Gamma(\mathscr{O}) =\bigoplus\limits_{n\in \mathbb{N}^+}\Gamma(\mathscr{O}^n)$. Finally, let
\begin{equation*}
    C_0(\Gamma) := \bigl\{ f : \Gamma \to \mathbb{C} \,:\, |\Gamma \setminus f^{-1}(0) | < \infty\bigr\}
\end{equation*}
be the space of compactly supported, complex-valued functions on $\Gamma$. 

\begin{definition}\label{defftrans}
    The Fourier transform $\mathcal{F}: C_0(\Gamma) \rightarrow \Gamma(\mathscr{O})$ is defined by
\begin{equation}\label{ftrans}
    (\mathcal{F} f) ([\rho]): = \left[ \rho,\sum_{\gamma\in \Gamma } f(\gamma^{-1}) \rho(\gamma) \right], \qquad \forall f \in C_0(\Gamma).
\end{equation}
\end{definition}

%The Fourier transform of $\psi$ is sometimes denoted by $\widehat{\psi}$.

Similarly to the construction of $\mathscr{O}^n$, we define a Hilbert bundle $\mathcal{V}^n$ over $\mathcal{M}^n_{\Gamma}$ for each $n \in \mathbb{N}^+$. Since the covering map $\pi:X\rightarrow M$ is normal, $N = \pi_*(\pi_1(X))$ is a normal subgroup of $\pi_1(M)$, and $\Gamma$ is isomorphic to $\pi_1(M)/N$. This gives an embedding map $\iota: \mathrm{Hom}_{\mathrm{irr}} (\Gamma, \mathrm{U}(n)) \rightarrow \mathrm{Hom}_{\mathrm{irr}} (\pi_1(M), \mathrm{U}(n))$, which is defined as
\begin{equation}\label{embeddinggm}
    \iota(\rho) (\gamma) : = \rho (\gamma  N), \quad \forall \gamma\in \pi_1(M).
\end{equation}
For any $[\rho] \in \mathcal{M}_{\Gamma}$, let $F^{\rho}$ denote the flat Hermitian bundle $\widetilde{M} \times_{\rho}\mathbb{C}^n$ constructed in Section \ref{subsec: twisted laplacian} (here, we use the embedding \eqref{embeddinggm} to view $\rho$ as a representation of $\pi_1(M)$). Any $[\rho] \in \mathcal{M}_{\Gamma}$ induces a natural group homomorphism $\widetilde{\rho}: \Gamma \to \mathrm{U}(\mathrm{End}(\mathbb{C}^n))$ by
\begin{equation*}
    \widetilde{\rho}(\gamma):A\mapsto \rho(\gamma)\circ A, \qquad \forall A\in \mathrm{End}(\mathbb{C}^n).
\end{equation*}
Therefore, we can define the flat Hermitian bundle $\mathrm{End}(F^{\rho}):=\widetilde{M} \times_{\widetilde{\rho}}\mathrm{End}(\mathbb{C}^n)$ as in Section \ref{subsec: twisted laplacian}.
For any 
\begin{equation*}
   (\rho,\psi)  \in \bigsqcup_{\rho \in \mathrm{Hom}_{\mathrm{irr}}(\Gamma, \mathrm{U}(n))}L^2(M;\mathrm{End}(F^{\rho})),
\end{equation*}
and $U \in \mathrm{U}(n)$, we define
\begin{equation}
    U\cdot (\rho, \psi) := ( U\rho U^*, U\psi U^* ),
\end{equation}
which gives a free, linear action of $\mathrm{PU}(n)$ on 
$\bigsqcup\limits_{\rho \in \mathrm{Hom}_{\mathrm{irr}}(\Gamma, \mathrm{U}(n))}L^2(M;\mathrm{End}(F^{\rho}))$. Thus,
\begin{equation} \label{eq:Vn}
\mathcal{V}^n := \left( \bigsqcup_{\rho \in \mathrm{Hom}_{\mathrm{irr}}(\Gamma, \mathrm{U}(n))}L^2(M;\mathrm{End}(F^{\rho})) \right) \bigg/ \mathrm{PU}(n) \longrightarrow \mathcal{M}^n_{\Gamma}
\end{equation}
defines a Hilbert bundle over $\mathcal{M}^n_{\Gamma}$. Let $\Gamma(\mathcal{V}^n)$ be the space of sections of $\mathcal{V}^n$. Finally, let $\mathcal{V}$ be the Hilbert-sheaf of $\mathcal{M}_{\Gamma}$ induced by the vector bundles $\mathcal{V}^n$, and $\Gamma(\mathcal{V}) = \bigoplus\limits_{n\in \mathbb{N}^+} \Gamma(\mathcal{V}^n)$.
\begin{definition}\label{defbloch1}
For any $x\in X$ and $[\rho]\in \mathcal{M}_\Gamma$, the  generalized Bloch transform $\mathcal{B}: C_0^{\infty}(X) \rightarrow \Gamma(\mathcal{V})$ associated with $(x,[\rho])$
is defined via
\begin{equation}\label{btrans}
    (\mathcal{B}\psi)([\rho]) (x) : = \left[ \rho, \mathcal{F}\Big( (\mathcal{B}_{\mathrm{NC}}\psi)(x) \Big)\right]=\left[ \rho, \sum_{\gamma \in \Gamma} \psi(\gamma^{-1} x)\rho(\gamma) \right],
\end{equation}
where the second equality follows from \eqref{defftrans} by taking $f(\gamma)=\psi(\gamma x)$.
\end{definition}

\begin{example}
    We review the classical Bloch transform on Euclidean space. Let $X=\widetilde{M}=\mathbb{R}^{d}$, $\Gamma=\mathbb{Z}^{d}$, and $M=\mathbb{T}^d$ be the unit torus. The group action of $\Gamma$ acting on $X$ is given by $n\cdot x = x+n$. Since $\Gamma$ is commutative, any irreducible unitary representation of $\Gamma$ is one dimensional and can be identified by a parameter $\theta \in \mathbb{T}^d$:
    \begin{equation*}
        \rho_\theta:\mathbb{Z}^d \rightarrow\mathbb{C}, \qquad n\mapsto e^{2\pi i n\cdot \theta}.
    \end{equation*}
    The non-commutative Bloch transform is
    \begin{equation*}
        (\mathcal{B}_{\mathrm{NC}} \psi)(x)=\sum_{n\in \mathbb{Z}^d} \psi(n\cdot x) \delta_n = \sum_{n\in \mathbb{Z}^d} \psi(x+n) \delta_n.
    \end{equation*}
   The Fourier transform is
\begin{equation*}
    (\mathcal{F}f)(\theta)=(\mathcal{F}f)(\rho_\theta)=\sum_{n\in \mathbb{Z}^d}f(n^{-1}) \rho_\theta(n) =   \sum_{n\in \mathbb{Z}^d}f(n) e^{-2\pi i n\cdot \theta},
\end{equation*}
where $n^{-1}$ denotes the inverse of $n$, viewed as a group element of $\Gamma$. Then, the classical Bloch transform is recovered:
\begin{equation}
    (\mathcal{B}\psi)(x,\theta)=\mathcal{F}\big((\mathcal{B}_{\mathrm{NC}}\psi)(x) \big)(\rho_\theta)=\sum_{n\in \mathbb{Z}^d}\psi(x+n)e^{-2\pi i n\cdot \theta}.
\end{equation}
It is clear that, for any $\theta \in \mathbb{T}^{d}$ and $m \in \Gamma$, 
\begin{equation*}
    (\mathcal{B}\psi)(x+m,\theta)=e^{2\pi i m\cdot \theta} (\mathcal{B}\psi)(x,\theta)=\rho_{\theta}(m)(\mathcal{B}\psi)(x,\theta).
\end{equation*}
Therefore, $(\mathcal{B}\psi)(\cdot,\theta)$ is a section of the flat Hermitian bundle $\mathrm{End}(F^{\rho_\theta})$ over $\mathbb{T}^d$ for each $\theta \in \mathbb{T}^d$.
\end{example}

%In the following, we introduce the generalized Bloch transform for any $(X,g)$ with symmetric group $\Gamma$. Let $\pi:X\rightarrow M$ be the smooth covering map. Then,

%For all $p\in [1,\infty]$, let $\ell^p(\Gamma)$ be defined in the obvious way. We have $C_0(\Gamma) \subset \ell^1(\Gamma)\subset \ell^p(\Gamma)$.

%Similarly to the Euclidean case, for any $\psi\in C_{0}^{\infty}(X)$, we define $\psi_{y}(\gamma):=\psi(\gamma^{-1}y) \in C_{0}(\Gamma)$. Then we apply the Fourier transform on $\psi_{y}(\gamma)$, we obtain that 
%$$\mathcal{F}(\psi_y)([\rho])=\left[ \rho,\sum\limits_{\gamma\in \Gamma } \psi(\gamma^{-1}y) \rho(\gamma) \right].$$ 
%We notice that
%$$\mathcal{F}(\psi_{\eta y})([\rho])=\rho(\eta)\mathcal{F}(\psi_{y}).$$

%It implies $\mathcal{F}(\psi_y)([\rho]) \in L^2(M,\mathrm{End}(F^{\rho}))$ with respect to $y$.

% \begin{remark}\label{rem1}
%     Given $f \in C_0^{\infty}(X)$ and $y\in M$, we define $f_y \in C_0(\Gamma)$ by:
%     \begin{equation}
%         f_y(\gamma) = f(\gamma\cdot y), \quad \forall \gamma \in \Gamma.
%     \end{equation}
% Then we have $\mathscr{B}(f)([\rho])(y)  = \mathscr{F}(f_y)([\rho])$.
% \end{remark}

\subsection{Mapping properties of the generalized Bloch transform for type I groups} \label{sec:mapBlochtypeI}

In this section, we further assume that $\Gamma$ is a type I group. Under this assumption, the mapping properties of the generalized Bloch transform are well understood. Before presenting the main results, we introduce some additional concepts.

Let $G$ be a topological group and $\pi$ be a unitary representation of $G$. We say that $\pi$ is a \emph{factor representation} of $G$ if the Von Neumann algebra generated by the family $\pi(G)$ has a trivial center $\mathbb{C}I$. Moreover, if $\pi$ is a direct sum of irreducible representations, $\pi$ is a factor representation if and only if all its irreducible subrepresentations are unitarily equivalent.

There are many equivalent definitions of type I groups. We adopt the following one \cite[Page 229]{folland2016course}:

\begin{definition}\label{defoftype1}
    Let $G$ be a topological group. $G$ is a type I group if every factor representation of $G$ is a direct sum of many copies of the same irreducible unitary representation.
\end{definition}
%{\color{red}[Yunlei: I simplify the def here, true? Check it]}

%\begin{definition}\label{defoftype1}
%    Let $G$ be a topological group. $G$ is a type I group if every factor representation of $G$ is the direct sum of irreducible unitary representations that are equivalent to a single irreducible unitary representation.
%\end{definition}

The celebrated Thoma's theorem characterizes type I countable discrete groups.

\begin{theorem}[Thoma's theorem, \cite{thoma1964unitare, tonti2019thoma}]\label{typeIandvirtuelly}
A countable discrete group $\Gamma$ is type I if and only if $\Gamma$ is a virtually Abelian group, \emph{i.e.}, there is an Abelian normal subgroup of $\Gamma$ with finite index. In this case, the supremum of the dimensions of the irreducible unitary representations of $\Gamma$ is finite, denoted by $d_\Gamma$.
\end{theorem}

As a consequence, $\mathcal{M}_{\Gamma } = \widehat{\Gamma}$ is the dual space of $\Gamma$. There exists a measure $\mu$ on $\widehat{\Gamma}$ such that the Plancherel formula holds \cite{folland2016course}: 
\begin{equation*}
    \sum_{\gamma \in \Gamma} f(\gamma) \overline{g(\gamma)}= \int_{\widehat{\Gamma}} \big\langle  \mathcal{F}f([\rho]) , \mathcal{F}g([\rho]) \big\rangle_{\mathrm{HS}}  \,d\mu([\rho]), \qquad \forall f,g\in C_0(\Gamma).
\end{equation*}
Using this formula, the Fourier transform $\mathcal{F}$ in Definition \ref{defftrans} extends to an isometry $\mathcal{F} :\ell^2(\Gamma) \rightarrow L^2(\widehat{\Gamma},\mu)$, \emph{i.e.},  
\begin{equation}\label{uniF}
    \| f\|_{\ell^2(\Gamma)}^2 = \int_{\widehat{\Gamma}} \|\mathcal{F}f([\rho]) \|_{\mathrm{HS}}^2 \,d\mu([\rho]) , \qquad \forall f\in \ell^2(\Gamma).
\end{equation}

Finally, we denote 
\begin{equation*}
    L^2(\mathcal{V}) = \int_{\widehat{\Gamma}}^{\oplus}L^2(M;\mathrm{End}(F^{\rho}))d\mu([\rho]).
\end{equation*}

\begin{theorem}\label{thmunitB}
    The generalized Bloch transform in Definition \ref{defbloch1} can be extended to the map
    \begin{equation*}
        \mathcal{B}:L^2(X) \rightarrow L^2(\mathcal{V})
    \end{equation*}
   with the unitary property:
\begin{equation}\label{unitB}
       \| \psi\|_{L^2(X)}^2 = \int_{\widehat{\Gamma}} \int_M \big\|(\mathcal{B}\psi)([\rho]) (x) \big\|_{\mathrm{HS}}^2  \,dxd\mu([\rho]), \qquad \forall \psi \in L^2(X).
   \end{equation}
\end{theorem}
\begin{proof}
    Let $\psi \in L^2(X)$. We have
    \begin{equation*}
        \begin{aligned}
            \| \psi \|_{L^2(X)}^2 & = \| \mathcal{B}_{\mathrm{NC}}\psi \|_{L^2(M;F^{\rho_{\pi}})}^2 =  \int_M \| (\mathcal{B}_{\mathrm{NC}}\psi) (x)\|^2_{\ell^2(\Gamma)} \,dx \\
            &= \int_M \int_{\widehat{\Gamma}}  \left\| \mathcal{F}\Big( (\mathcal{B}_{\mathrm{NC}}\psi)(x) \Big) ([\rho])\right\|_{\mathrm{HS}}^2 \,d\mu([\rho])dx \\
            &= \int_{\widehat{\Gamma}} \int_M  \left\| \mathcal{F}\Big( (\mathcal{B}_{\mathrm{NC}}\psi)(x) \Big) ([\rho])\right\|_{\mathrm{HS}}^2 \,dxd\mu([\rho])\\
            &= \int_{\widehat{\Gamma}} \int_M \big\|(\mathcal{B}\psi)([\rho]) (x) \big\|_{\mathrm{HS}}^2  \,dxd\mu([\rho]),
        \end{aligned}
    \end{equation*}
where we used the unitary property \eqref{unitaryofnonbloch} of $\mathcal{B}_{\mathrm{NC}}$ for the first equality, the unitary property \eqref{uniF} of $\mathcal{F}$ for the second line, Fubini theorem for the third line, and the definition of $\mathcal{B}$ for the last line. The proof is complete.
\end{proof}

\subsection{Examples of Abelian and type I coverings}\label{sec:typeIcover}

In this section, we give examples of Abelian and type I coverings of compact hyperbolic surfaces. 

Let $M_g$ be a compact hyperbolic surface of genus $g$. A fundamental domain for $M_g$ in the universal cover $\mathbb{H}$ is a hyperbolic $4g$-gon $D_g$, which is defined by the boundary word 
\begin{equation*}
    a_{1}b_{1}a_{1}^{-1}b_{1}^{-1}\cdots a_{g}b_{g}a_{g}^{-1}b_{g}^{-1}.
\end{equation*}
The surface $M_g$ is obtained by identifying the edges $a_j$ with $a_j^{-1}$ and $b_j$ with $b_j^{-1}$ for $j=1,\dots,g$. 

We begin with the construction of the Abelian cover. Let 
\begin{equation*}
    \gamma_1,\cdots,\gamma_g,\gamma_{g+1}, \gamma_{2g} \in \mathrm{PSL}_2(\mathbb{R})
\end{equation*}
be the corresponding side-pairing transformations such that $\gamma_j$ identifies $a_j$ with $a_j^{-1}$ and $\gamma_{g+j}$ identifies $b_j$ with $b_j^{-1}$ for $j=1,\cdots,g$. For each $i\in \Z$, starting from $\gamma_1^i(D_g)$, the $i$-th translated copy of $D_g$, we identify $\gamma_1^i(a_j)$ with $\gamma_1^i(a^{-1}_j)$ for $j=2,\cdots,g$ and $\gamma_1^i(b_j)$ with $\gamma_1^i(b_j^{-1})$ for $j=1,\dots,g$. The resulting surface is a surface of type $(g-1,2)$, denoted by $\widetilde{M}_{g-1,2}^{i}$. By identifying $\gamma_1^{i}(a_1)$ with $\gamma_1^{i+1}(a_1^{-1})$ for all $i \in \mathbb{Z}$, we obtain a $\mathbb{Z}$-cover hyperbolic surface $X$.

%To construct a $\mathbb{Z}$-cover of $M_g$, we consider a surface of type $(g-1,2)$, denoted by $\widetilde{M}_{g-1,2}^{i}$, which is obtained from $\gamma_1^{i}(D_g)$ by gluing all edges except for $a_1$ and $a_1^{-1}$, for each $i \in \mathbb{Z}$. 

Another significant example of Abelian covers is the maximal Abelian cover $\pi^{ab}: M^{ab}_g \to M_g$, whose deck transformation group is isomorphic to the first homology group $H_{1}(M_g, \mathbb{Z}) \cong \mathbb{Z}^{2g}$. For any Abelian cover $\pi: X \to M_g$, there exists an intermediate Abelian cover $\pi^A_{X}: M^{ab}_g \to X$ such that $\pi^{ab} = \pi \circ \pi^A_{X}$. To construct $M^{ab}_g$, we take copies $\{D_g^v\}_{v \in \mathbb{Z}^{2g}}$ of the fundamental domain. The edges $a_j$ of $D_g^v$ are glued to the edges $a_j^{-1}$ of $D_g^{v+e_j}$, and the edges $b_j$ of $D_g^v$ are glued to $b_j^{-1}$ of $D_g^{v+e_{g+j}}$ for $j=1,\dots,g$. This tiling results in the $\mathbb{Z}^{2g}$-cover of $M_g$.

We now consider a type I covering $\pi: X \to M$ with deck transformation group $\Gamma$. Since $\Gamma$ is a countable discrete type I group, Thoma's Theorem \ref{typeIandvirtuelly} states that $\Gamma$ is virtually Abelian; that is, there exists an Abelian normal subgroup $H \triangleleft
 \Gamma$ of finite index ($[\Gamma: H] < \infty$). According to the Galois correspondence of covering spaces, there exists a hyperbolic surface $X_H$ such that the map $\pi$ factors as $\pi = \pi_{f} \circ \pi_{H}$, where $\pi_{f}: X_H \to X$ is a finite covering and $\pi_{H}: X \to X_H$ is an $H$-covering. Consequently, the study of observability for type I coverings can be reduced to the Abelian case. However, this reduction does not directly yield the dependence of the observability constant on the group structure.

Based on this observation, we construct a non-Abelian type I cover of a genus-2 surface. Let $M_3$ be a genus-3 compact hyperbolic surface with an involution $\tau: M_3 \to M_3$ satisfying $\tau^2 = \mathrm{Id}$, such that the quotient $M_3 / \tau$ is a genus-2 surface $M_2$. We assume that a simple non-separating closed curve $\gamma$ (represented by $a_1$ in the polygon representation of $M_3$) is invariant under $\tau$ but with its orientation reversed ($\tau(\gamma) = \gamma^{-1}$). For instance, if $M_3$ is centrally symmetric around a point $O$ and $\tau$ is the rotation by $\pi$ around $O$, then $\gamma$ could be the curve passing through the central genus.

This involution $\tau$ induces a double cover $\pi_2: M_3 \to M_2$. By repeating the $\mathbb{Z}$-cover construction along $\gamma$, we obtain $\pi_{\mathbb{Z}}: X \to M_3$. We claim that the composite map $\pi := \pi_2 \circ \pi_{\mathbb{Z}}: X \to M_2$ is a \emph{$D_\infty$-covering} of $M_2$, where $D_\infty$ is the infinite dihedral group:
\begin{equation}
    D_\infty = \langle T, R : R^2=1, RTR=T^{-1} \rangle.
\end{equation}
Here, the $\mathbb{Z}$-covering $\pi_{\mathbb{Z}}$ corresponds to the translation $T$ on the homology class $[\gamma]$, while $\pi_2$ induces the reflection $R$. Thus, $X \to M_2$ is a non-Abelian type I covering. Note that any infinite type I covering of a closed hyperbolic surface is necessarily a non-compact, infinite-genus hyperbolic surface without cusps or funnels; see the schematic diagram in Figure \ref{type1cover}.

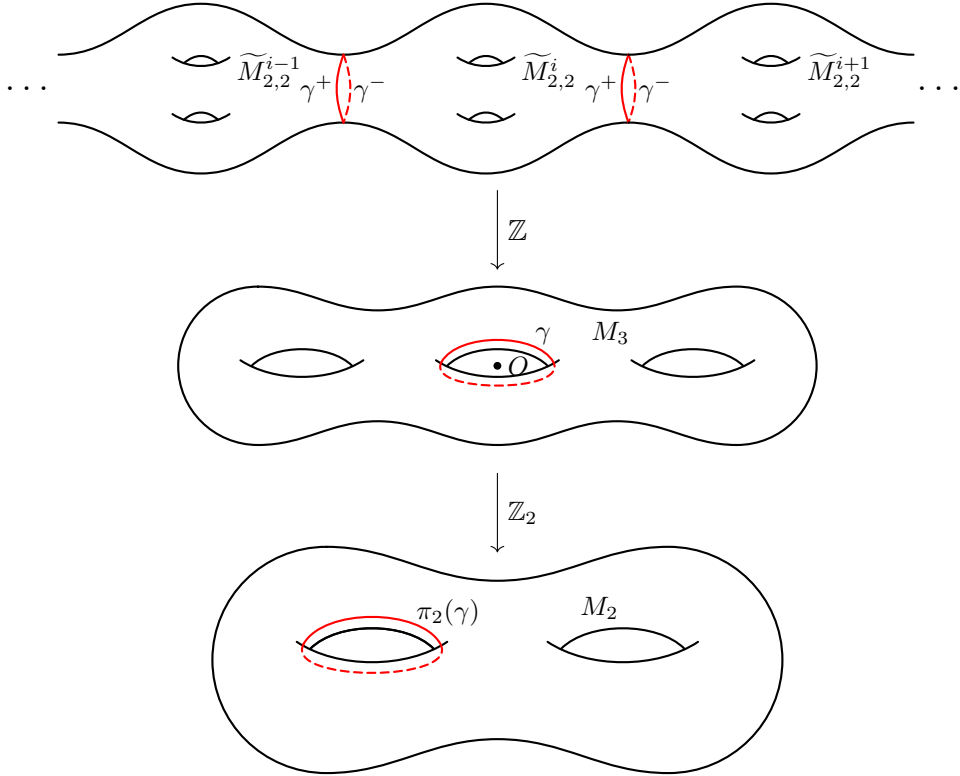
\begin{figure}[htbp!]
    \centering
\begin{tikzpicture}[line join=round, line cap=round, scale=1.5]
% --- 1. THE INFINITE SURFACE (REPEATING GENUS 2 BLOCKS) ---
    \begin{scope}[shift={(1.4,4.2)},scale=0.5]
  % Define the repeating unit for the infinite translation
  % Each unit is 5 units wide and contains 2 holes (Genus 2)
  \foreach \i in {-1, 0, 1} {
    \begin{scope}[shift={(\i*5, 0)}]
      
      % --- SMOOTH BOUNDARY (No cusps) ---
      % The 'out=0' and 'in=180' ensure the join at +/- 2.5 is perfectly horizontal
      \draw[thick] (-2.5, 0.6) to[out=0, in=180] (0, 1.5) 
                            to[out=0, in=180] (2.5, 0.6);
                            
      \draw[thick] (-2.5, -0.6) to[out=0, in=180] (0, -1.5) 
                              to[out=0, in=180] (2.5, -0.6);

      % --- CLEAN HOLES (No crossing) ---
      \foreach \h in {-0.5, 0.5} {
        \begin{scope}[shift={(0, \h)}]
          % Bottom curve (the smile)
          \draw[thick] (-0.5, 0.05) .. controls (-0.2, -0.15) and (0.2, -0.15) .. (0.5, 0.05);
          % Top curve (the frown) - set slightly inside the smile to prevent intersection
          \draw[thick] (-0.3, -0.05) .. controls (-0.1, 0.12) and (0.1, 0.12) .. (0.3, -0.05);
        \end{scope}
      }
    \end{scope}
    % draw gamma
        \ifnum\i=0\relax
            \draw[thick,red] (-2.5,0.6) to[out=-110,in=110] (-2.5,-0.6);
            \draw (-2.5,0) node[left] {\small$\gamma^+$};
            \draw (-2.5,0) node[right] {\small$\gamma^-$};
            \draw[thick,red,densely dashed] (-2.5,0.6) to[out=-70,in=70] (-2.5,-0.6);
            \draw[thick,red] (2.5,0.6) to[out=-110, in=110] (2.5,-0.6);
            \draw[thick,red,densely dashed] (2.5,0.6) to[out=-70, in=70] (2.5,-0.6);
            \draw (2.5,0) node[left] {\small$\gamma^+$};
            \draw (2.5,0) node[right] {\small$\gamma^-$};
            \draw (1.1+\i*5,0.8) node[below] {\small$ \widetilde{M}^{i}_{2,2}$};
            \else
            \ifnum \i=1\relax
            \draw (1.2+\i*5,0.8) node[below] {\small$\widetilde{M}^{i+\i}_{2,2}$};
            \else
            \ifnum \i=-1\relax
            \draw (1.2+\i*5,0.8) node[below] {\small$\widetilde{M}^{i-1}_{2,2}$};
            \fi
            \fi
        \fi
        
  }

  % --- INDICATORS & LABELS ---
  % Ellipsis to show it continues forever
  \node at (-8, 0) {\Large $\dots$};
  \node at (8, 0) {\Large $\dots$};

\end{scope}

 % The map
  \draw[->] (1.5,3.3) -- node[right] {$\mathbb{Z}$} (1.5,2.6);

  % Main outline of genus 3 surface
\begin{scope}[shift={(0.1,1.75)},scale=0.7]

  \draw[thick, smooth, tension=0.6] 
    (-1,1) to [out=0, in=180] (0.5,0.7) 
           to [out=0, in=180] (2,1) 
           to [out=0, in=180] (3.5,0.7) 
           to [out=0, in=180] (5,1)
           to [out=0, in=90]  (6,0) 
           to [out=-90, in=0] (5,-1)
           to [out=180, in=0] (3.5,-0.7) 
           to [out=180, in=0] (2,-1) 
           to [out=180, in=0] (0.5,-0.7) 
           to [out=180, in=0] (-1,-1)
           to [out=180, in=-90] (-2,0) 
           to [out=90, in=180] (-1,1);

  % Hole 1 (Left)
  \begin{scope}[shift={(-1.5,0)},scale=1.4]
    \draw[thick,shift={(0,-0.05)}] (0.2,0.1) .. controls (0.5,-0.1) and (1.0,-0.1) .. (1.3,0.1);
    \draw[thick] (0.3,0)   .. controls (0.5,0.2)  and (1.0,0.2)  .. (1.2,0);
  \end{scope}

  % Hole 2 (Middle)
  \begin{scope}[shift={(0.95,0)},scale=1.4]
    \draw[thick,shift={(0,-0.05)}] (0.2,0.1) .. controls (0.5,-0.1) and (1.0,-0.1) .. (1.3,0.1);
    % curve gamma
    % bottom
    \draw[thick,shift={(0,-0.05)}, color=red,densely dashed] (0.25,0.085) .. controls (0.1,-0.2) and (1.4,-0.2) .. (1.25,0.085);
    % above
    \draw[thick, shift={(0,-0.05)}, color=red] (0.25,0.085) .. controls (0.4,0.35) and (1.1,0.35) .. (1.25,0.085);
    \draw (1.0,0.27) node[right] {\small$\gamma$};
    \draw (1.5,0.3) node[right] {\small$ M_3$};
    
    \draw[thick] (0.3,0)   .. controls (0.5,0.2)  and (1.0,0.2)  .. (1.2,0);
    % Center
    \fill (0.75,0) circle (1pt) node[right] {\small$O$};
  \end{scope}

  % Hole 3 (Right)
  \begin{scope}[shift={(3.4,0)},scale=1.4]
    \draw[thick,shift={(0,-0.05)}] (0.2,0.1) .. controls (0.5,-0.1) and (1.0,-0.1) .. (1.3,0.1);
    \draw[thick] (0.3,0)   .. controls (0.5,0.2)  and (1.0,0.2)  .. (1.2,0);
  \end{scope}
\end{scope}

% Main outline of genus 2 surface
  \begin{scope}[shift={(0,-0.85)}]
  \draw[thick, smooth, tension=0.7] 
    (0,1) to [out=0, in=180] (1.5,0.7) to [out=0, in=180] (3,1) 
    to [out=0, in=90] (4,0) to [out=-90, in=0] (3,-1) 
    to [out=180, in=0] (1.5,-0.7) to [out=180, in=0] (0,-1) 
    to [out=180, in=-90] (-1,0) to [out=90, in=180] (0,1);

  % First hole (left)
  \begin{scope}[shift={(-0.5,0.1)},scale=1.2]
  \draw[thick,shift={(0,-0.05)}] (0.2,0.1) .. controls (0.5,-0.1) and (1.0,-0.1) .. (1.3,0.1);
  \draw[thick] (0.3,0)   .. controls (0.5,0.2)  and (1.0,0.2)  .. (1.2,0);
      % bottom
    \draw[thick,shift={(0,-0.05)}, color=red,densely dashed] (0.25,0.085) .. controls (0.1,-0.2) and (1.4,-0.2) .. (1.25,0.085);
    % above
    \draw[thick, shift={(0,-0.05)}, color=red] (0.25,0.085) .. controls (0.4,0.35) and (1.1,0.35) .. (1.25,0.085);
    \draw (1.0,0.27) node[right] {\small$\pi_2(\gamma)$};
    
    \draw[thick] (0.3,0)   .. controls (0.5,0.2)  and (1.0,0.2)  .. (1.2,0);
    \draw (2.2,0.3) node[right] {\small$ M_2$};
  \end{scope}

  % Second hole (right)
  \begin{scope}[shift={(-0.1,0.1)},scale=1.2]
  \draw[thick,shift={(0,-0.05)}] (1.7,0.1) .. controls (2.0,-0.1) and (2.5,-0.1) .. (2.8,0.1);
  \draw[thick] (1.8,0)   .. controls (2.0,0.2)  and (2.5,0.2)  .. (2.7,0);
  \end{scope}
\end{scope}
 % The map
  \draw[->] (1.5,0.8) -- node[right] {\small$\mathbb{Z}_2$} (1.5,0.1);

\end{tikzpicture}
    \caption{Schematic of a nontrivial type I covering with non-Abelian deck group. The middle map $\pi_2:M_3\to M_2$ is a $\Z_2$-cover induced by an involution $\tau$ with $\tau(\gamma)=\gamma^{-1}$. The top map $\pi_\Z:X\to M_2$ is the $\Z$-cover obtained by cutting $M_3$ along $\gamma$ and gluing copies $\widetilde{M}_{2,2}^{i}$ for all $i\in\Z$ along $\gamma^{\pm}$, producing an infinite-genus surface $X$.}
    \label{type1cover}
\end{figure}

\section{Dynamics of geodesic flow and anisotropic symbol classes}\label{sec: dynamics of geodesic flow}

In this section, we review the geodesic flow and the anisotropic symbol classes on compact hyperbolic surfaces; see \cite{dyatlovzahl2016spectral,dyatlov2018semiclassical}. 

Let $(M, g)$ be an oriented compact hyperbolic surface and $T^* M \setminus \{0\} $ consist of elements of the cotangent bundle $(x,\xi) \in T^*M$ such that $\xi \neq 0$. Denote by $S^* M=\left\{|\xi|_g=1\right\}$ the cosphere bundle, where $|\xi|_g:=\langle \xi,\xi\rangle_g^{1/2}$. Define the symbol $p \in C^{\infty}\left(T^* M \setminus \{0\} ;\mathbb{R}\right)$ by $p(x, \xi)=|\xi|_{g}$. The Hamiltonian flow of $p$ is the homogeneous geodesic flow given by
\begin{equation*}
    \varphi_t:=\exp \left(t H_p\right): T^* M \setminus \{0\} \rightarrow T^* M \setminus \{0\}.
\end{equation*}

We use an explicit frame on $T^* M \setminus \{0\}$ consisting of four vector fields
\begin{equation}
   H_p, \,U_{+}, \, U_{-},\, D \in C^{\infty}\big( T^* M \setminus \{0\} ; T( T^* M \setminus \{0\} )\big) . 
\end{equation}
Here $H_p$ is the generator of $\varphi_t$ and $D=\xi \cdot \partial_{\xi}$ is the generator of dilation. The vector fields $U_{\pm}$ are defined on $S^*M$ as stable ($U_{+}$) and unstable ($U_{-}$) horocyclic vector fields and extended homogeneously to $T^*M\setminus \{0\}$. They satisfy the commutation relations
\begin{equation}\label{commutator}
\left[H_p, U_{ \pm}\right]= \pm U_{ \pm}, \quad \left[U_{ \pm}, D\right]=\left[H_p, D\right]=0.
\end{equation}
Thus on each level set of $p$, the flow $\varphi_t$ has a flow/stable/unstable decomposition, with $U_{+}$ spanning the stable space and $U_{-}$ spanning the unstable space. We use the following notation for the weak stable/unstable spaces:
\begin{equation}\label{stablefoliation}
L_s:=\operatorname{span}\left(H_p, U_{+}\right), \ L_u:=\operatorname{span}\left(H_p, U_{-}\right) \subset T (T^* M \backslash \{0\} ).
\end{equation}
Then, $L_s$ and $ L_u$ are Lagrangian foliations; see \cite[Definition 3.1]{dyatlovzahl2016spectral}. Here we give a brief introduction of Lagrangian foliations.
\begin{definition}
Let $M$ be a manifold, $U \subset T^* M$ be an open set, and
\begin{equation*}
    L_{(x, \xi)} \subset T_{(x, \xi)} (T^* M  ), \qquad  (x, \xi) \in U,
\end{equation*}
a family of subspaces depending smoothly on $(x, \xi)$. We say that $L$ is a \emph{Lagrangian foliation} on $U$ if
\begin{itemize}
\item $L_{(x, \xi)}$ is integrable in the sense that if $X, Y$ are two vector fields on $U$ lying in $L$ at each point (we denote this by $X, Y \in C^{\infty}(U ; L)$), then the Lie bracket $[X, Y]$ lies in $C^{\infty}(U ; L)$ as well;
\item $L_{(x, \xi)}$ is a Lagrangian subspace of $T_{(x, \xi)}\left(T^* M\right)$ for each $(x, \xi) \in U$.
\end{itemize}
\end{definition} 

A basic example of Lagrangian foliation is the vertical foliation on $T^*\R^d$. Let 
\begin{equation*}
    T^*\mathbb{R}^d=\{(y_{1},\cdots,y_{d},\eta_{1},\cdots,\eta_{d})\},
\end{equation*}
and define the vertical foliation $L_0$ by
\begin{equation*}
    L_0=\operatorname{span} \big(\partial_{\eta_1}, \cdots, \partial_{\eta_d} \big).
\end{equation*}
By \cite[Lemma 3.6]{dyatlovzahl2016spectral}, every Lagrangian foliation is locally symplectomorphic to $L_0$:

\begin{lemma}\label{locallagrangian}
 Let $U\subset T^*M$ be open and $L$ be a Lagrangian foliation on $U$. For any $(x_0,\xi_0)\in U$, there exists an open neighborhood $U_0$ of $(x_0,\xi_0)$ and a symplectomorphism $\kappa: U_0 \rightarrow V_0$, such that $U_0 \subset U$, $V_0\subset T^* \mathbb{R}^d$, and
 \begin{equation*}
     (\kappa_{*}L)_{(x,\xi)}=d \kappa_{(x, \xi)} (L_{(x, \xi)})=\left(L_0\right)_{\kappa(x, \xi)}, \qquad \forall(x, \xi) \in U_0.
 \end{equation*}
The triple $(\kappa,U_0,V_0)$ is called a Lagrangian chart.
\end{lemma}
In Dyatlov, Zahl \cite{dyatlovzahl2016spectral} and Dyatlov, Jin \cite{dyatlov2018semiclassical}, the authors introduced the anisotropic symbol classes associated with Lagrangian foliations. We review the definition below, following \cite[(2.16) \& Appendix A.1]{dyatlov2018semiclassical}:

\begin{definition}\label{anisosymbolclass}
For two parameters $0 \leq \mu<1$, $0 \leq \mu^{\prime} \leq \frac{\mu}{2}$, $\mu+\mu^{\prime}<1$, an $h$-dependent symbol $a=a(x,\xi;h)$ lies in the class $S_{L, \mu, \mu^{\prime}}^{\mathrm{comp}}(U)$ if
\begin{itemize}
\item $a(x, \xi ; h)$ is smooth in $(x, \xi) \in U$, defined for $0<h \leq 1$, and supported in an $h$-independent compact subset of $U$;
\item $a$ satisfies the derivative bounds
$$
\sup _{x, \xi}\left|Y_1 \cdots Y_m Z_1 \cdots Z_k a(x, \xi ; h)\right| \leq C h^{-\mu k-\mu^{\prime} m}, \quad 0<h \leq 1
$$
for all vector fields $Y_1, \cdots, Y_m, Z_1, \cdots, Z_k$ on $U$ such that $Y_1, \cdots, Y_m$ are tangent to $L$. Here the constant $C$ depends on $Y_1, \cdots, Y_m, Z_1, \cdots, Z_k$ but does not depend on $h$.
\end{itemize}
The class $S_{L, \mu}^{\text{comp}}(U)$ is defined as
\begin{equation}
S_{L, \mu}^{\text{comp}}(U)=\bigcap_{\varepsilon>0} S_{L, \mu+\varepsilon, \varepsilon}^{\mathrm{comp}}(U).
\end{equation}
\end{definition}

The following lemma (see \cite[Lemma A.1]{dyatlov2018semiclassical}) is useful later:
\begin{lemma}
\label{multiofsymbol}
Let $C$ be an arbitrary fixed constant and assume that $a_1, \cdots, a_N \in S_{L, \mu, \mu^{\prime}}^{\mathrm{comp }}(U)$, $1 \leq N \leq C |\log h|$ are such that $\sup \left|a_j\right| \leq 1$ and each $S_{L, \mu, \mu^{\prime}}^{\mathrm {comp }}(U)$ seminorm of $a_j$ is bounded uniformly in $j$. Then for all small $\varepsilon>0$ the product $a_1 \cdots a_N$ lies in $S_{L, \mu+\varepsilon, \mu^{\prime}+\varepsilon}^{\mathrm{comp}}(U)$.
\end{lemma}

Now we return to the case where $M$ is a compact hyperbolic surface. If $a \in C_0^{\infty} (T^* M \setminus \{0\} )$ is an $h$-independent symbol, then it follows from the commutation relations \eqref{commutator} that
\begin{equation}\label{derivativeoffunction}
H_p^k U_{+}^{\ell} U_{-}^m D^n (a \circ \varphi_t )=e^{(m-\ell) t} (H_p^k U_{+}^{\ell} U_{-}^m D^n a ) \circ \varphi_t.
\end{equation}
Therefore, for $0\leq \mu<1$, we have
\begin{equation}\label{propogationofsymbol}
a\circ \varphi_{t} \in S_{L_s, \mu}^{\mathrm{comp}} (T^*M\setminus \{0\} )\text{ uniformly in }t, \qquad 0\leq t\leq \mu |\log h |.
\end{equation}
Similarly, 
\begin{equation}\label{propogationofsymbol2}
 a\circ \varphi_{-t} \in S_{L_u, \mu}^{\mathrm{comp}} (T^*M\setminus \{0\} )\text{ uniformly in }t, \qquad 0\leq t\leq \mu |\log h |.
\end{equation}

In the following, we will introduce the quantization of symbol classes that gives operators between the sections of flat Hilbert bundles with uniform upper bounds for $(\mathcal{H},\rho)$.

\section{Semiclassical analysis on flat Hilbert bundles}
\label{sec:quantization on flat bundles}
In this section, we always assume that $(M,g)$ is a $d$-dimensional compact Riemannian manifold, $\mathcal{H}$ is a separable Hilbert space with a countable orthonormal basis $\{e_i\}_{i \in I}$, and $\pi:F\rightarrow M$ is a flat Hilbert bundle with fibers isomorphic to $\mathcal{H}$.

%$\rho:\pi_1(M)\to \mathrm{U}(\mathcal{H})$ is a unitary representation, and $\pi_{\rho}: F^{\rho}\to M$ is the flat Hilbert bundle given by $\rho$, see Section \ref{defofflatbundle}. 

For any $k\in \mathbb{R}$, let $S^k(T^*M)$ be the Kohn-Nirenberg symbol class (see, \emph{e.g.},\cite[(E.1.6)]{dyatlovzworski2019resonances}):
\begin{equation}\label{standardsymbolclass}
    S^k(T^*M)=\left\{  a(x,\xi;h) \in C^\infty(T^*M):\forall \alpha,\beta \in \mathbb{N}^d,\sup _{h \in\left(0,1\right]} \sup _{\substack{(x,\xi)  \in T^*M  }}\langle\xi\rangle^{|\beta|-k} |\partial_x^\alpha \partial_{\xi}^\beta a  |<\infty \right\}.
\end{equation}
For any $a \in S^k(T^*M)$ and $\rho:\pi_1(M)\to \mathrm{U}(\mathcal{H})$, we will introduce the quantization $\Oph^\rho(a)$ as a pseudodifferential operator on the flat bundle $F^\rho$; see Definition \ref{defofscalarsemipseudo}. The standard semiclassical microlocal analysis is referred to \cite{DimassiSjostrand99,zworski2012semiclassicalanalysis,dyatlovzworski2019resonances}, and the semiclassical analysis on flat bundles and principal bundles with finite-dimensional fibers is referred to \cite{mama2024semiclassical,cekiclefeuvre2024semiclassical}.

Our key observation is Proposition \ref{uniformboundedtheorem}: $\Oph^{\rho}(a)$ is a bounded operator between Sobolev spaces, with bounds independent of $(\mathcal{H},\rho)\in \mathcal{C}$. It allows us to choose uniform constants in the inequalities involving Sobolev norms. Here, we briefly explain the reason for the independence of $(\mathcal{H},\rho)$. 

% By the choice of a parallel local trivialization coordinate chart, see Definition \ref{defoflocaltrivialization}, we can regard the quantization of symbol class as pseudodifferential operator on the Hilbert space valued functions space on each coordinate chart. 

For any open subset $U\subset M$, linear operator $A: C^{\infty}(U)\mapsto C^{\infty}(U)$, and $\phi \in \mathrm{U}(\mathcal{H})$, we define $A\otimes \phi$ as a linear operator on $\mathcal{H}$-valued function space $C^{\infty}(U;\mathcal{H})$ as
\begin{equation}\label{localformofop}
(A\otimes\phi)(u)(x):=\phi\big((Au^i)(x) e_{i}\big), \qquad \forall u=u^i e_i \in C^{\infty}(U;\mathcal{H}).
\end{equation}
If $A$ in \eqref{localformofop} is an $L^2$-bounded operator between $L^2(U)$, we have:
\begin{equation}\label{squaretensorextension}
\|A\otimes \phi\|_{L^2(U;\mathcal{H}) \rightarrow L^2(U;\mathcal{H})}=\|A\|_{L^2(U) \rightarrow L^2(U)}.
\end{equation}
All the operators in Section \ref{sec:quantization on flat bundles} on $C^{\infty}(M;F^{\rho})$ have such a form \eqref{localformofop} on parallel orthonormal local trivialization charts. In what follows, we will follow \cite[Appendix E]{dyatlovzworski2019resonances} to introduce semiclassical microlocal analysis on flat bundles and highlight the constants that are independent of $(\mathcal{H},\rho)$.

\subsection{Semiclassical Sobolev spaces on flat Hilbert bundles}\label{subsec:semisobolev}

We start by generalizing the notion of \emph{cutoff chart}, which is widely used in the quantization of symbols on manifolds (\emph{c.f.} \cite[Definition E.9]{dyatlovzworski2019resonances}), to the setting of flat Hilbert bundles.

%To define the quantization of symbol, we introduce the following \textit{cutoff parallel orthonormal local trivialization chart} similar to \cite[Definition E.9]{dyatlovzworski2019resonances} and \cite[Chapter 14]{zworski2012semiclassicalanalysis} to define semiclassical pseudodifferential operator class by gluing the local form \eqref{localformofop} on each parallel local trivialization chart.

\begin{definition}\label{defofcutofflocaltrivialization}
We say that $(\Phi,U,\varphi,\chi)$ is a \textit{cutoff normal chart} of a flat Hilbert bundle $F$ if $(\Phi,U,\varphi)$ is a parallel orthonormal local trivialization chart in the sense of Proposition \ref{defoflocaltrivialization} and $\chi\in C_{0}^{\infty}(U)$ is a cutoff function. 

%More precisely,
%\begin{itemize}
%    \item $(U,\varphi)$ is a chart on $M$, \emph{i.e.}, $U\subset M$ is open, and $\varphi:U\rightarrow \varphi(U)\subset \mathbb{R}^d$ is a diffeomorphism.
%    \item $\Phi:\pi^{-1}(U)  \subset F \rightarrow U \times \mathcal{H}$ is a local trivialization of $F$, and $\{ \Phi^{-1}(e_i)\}_{i\in I}$ is a local parallel orthonormal frame of $F$.
%    \item $\chi\in C_{0}^{\infty}(U)$ is a cutoff function.
%\end{itemize}
\end{definition}

We can view the local trivialization $\Phi:\pi^{-1}(U)  \subset F \rightarrow U \times \mathcal{H}$ as an operator $\Phi:C^\infty(U;F)\rightarrow C^\infty(U ;\mathcal{H})$: for any $f \in C^\infty(U;F)$, 
\begin{equation*}
     (\Phi f)(x) = \Phi (f(x)), \qquad \forall x\in U.
\end{equation*}
We define the operators
\begin{equation*}
    (\varphi^{-1})^*\Phi \chi :C^\infty(U;F) \rightarrow C_0^\infty(\varphi(U);\mathcal{H}), \qquad (\varphi^{-1})^* \Phi \chi f(x)= \chi(\varphi^{-1}(x)) \Phi\big( f(\varphi^{-1}(x)) \big),
\end{equation*}
and
\begin{equation*}
    \chi \Phi^{-1} \varphi^*:C^\infty(\varphi(U);\mathcal{H}) \rightarrow C_0^\infty(U;F), \qquad \chi \Phi^{-1} \varphi^* g(x)= \chi(x) \Phi^{-1}\big( g(\varphi(x)) \big).
\end{equation*}
These two operators can be naturally extended to
\begin{equation}\label{cutoffoperator}
    (\varphi^{-1})^* \Phi \chi: C^{\infty}(M;F)\to C_{0}^{\infty}(\mathbb{R}^d;\mathcal{H}), \quad \mathrm{and} \quad \chi \Phi^{-1} \varphi^*: C^{\infty}(\mathbb{R}^d;\mathcal{H})\to C^{\infty}(M;F ),
\end{equation}
respectively. 

Let $U_\alpha$ and $U_\beta$ be two overlapping open subsets of $M$, with two cutoff normal charts $(\Phi_\alpha,U_\alpha,\varphi_\alpha,\chi_\alpha)$ and $(\Phi_\beta,U_\beta,\varphi_\beta,\chi_\beta)$. Since $F$ is flat, the transition map $\Phi_{\alpha\beta}(x)=(\Phi_\alpha \circ \Phi_\beta^{-1})(x,\cdot) \in \mathrm{U}(\mathcal{H})$ is a constant unitary operator on $\mathcal{H}$: 
\begin{equation*}
    \Phi_{\alpha \beta} (x) = \Phi_{\alpha \beta} \in \mathrm{U} (\mathcal{H}), \qquad \forall x  \in U_\alpha \cap U_\beta.
\end{equation*}
The composition of the operators in \eqref{cutoffoperator} can be written as:
\begin{equation}\label{transitionmapform}
 \big[ (\varphi_\alpha^{-1})^* \Phi_\alpha \chi_\alpha \big] \circ \big[ \chi_\beta \Phi_\beta^{-1} \varphi_\beta^* \big] = \varphi_{\beta \alpha}^* \chi_{\alpha \beta} \otimes \Phi_{\alpha \beta} =  \chi_{\beta \alpha }  \varphi_{\beta \alpha}^* \otimes \Phi_{\alpha \beta}  . 
\end{equation}
where $\varphi_{\beta \alpha} = \varphi_\beta \circ \varphi_\alpha^{-1}$ denotes the transition function, and $\chi_{\alpha \beta} := (\chi_\alpha \chi_\beta) \circ \varphi_\beta^{-1}$ denotes the multiplication operator:
\begin{equation*}
    (\chi_{\alpha \beta} f)(x)= \chi_\alpha (\varphi_\beta^{-1}(x)) \chi_\beta (\varphi_\beta^{-1}(x))  f(x), \qquad \forall f\in C^{\infty}(\mathbb{R}^d).
\end{equation*}

%For any $x\in U_\alpha \cap U_\beta$, the transition map $\Phi_{\alpha\beta}(x)=(\Phi_\alpha \circ \Phi_\beta^{-1})(x,\cdot) \in \mathrm{U}(\mathcal{H})$ which is a constant unitary operator $\phi_{UU^{\prime}}(x)\equiv \phi_{UU^{\prime}} \in \mathrm{U}(\mathcal{H})$

%For two cutoff normal charts $(\Phi_\alpha,U_\alpha,\varphi_\alpha,\chi_\alpha)$ and $(\Phi_\beta,U_\beta,\varphi_\beta,\chi_\beta)$ such that $U_\alpha\cap U_\beta \not= \varnothing$. Then for any $x \in U\cap U^{\prime}$, the transition map $\phi_{UU^{\prime}}(x)=\Phi \circ (\Phi^{\prime-1})(x,\cdot) \in \mathrm{U}(\mathcal{H})$ which is a constant unitary operator $\phi_{UU^{\prime}}(x)\equiv \phi_{UU^{\prime}} \in \mathrm{U}(\mathcal{H})$ for any $x \in  U\cap U^{\prime}$. We consider the composition of two operators in \eqref{cutoffoperator}. Let $\varphi_{0}=\varphi\circ (\varphi^{\prime})^{-1}$ and $\psi_{0}=\psi\psi^{\prime}\circ (\varphi^{\prime})^{-1}$, we obtain that
%\begin{equation}
%\begin{aligned}
%((\varphi^{-1})^*\Phi \psi)(\psi^{\prime} (\Phi^{\prime})^{-1} (\varphi^{\prime})^*)&=(\varphi_{0}^{-1})^*\psi_{0}\otimes \phi_{UU^{\prime}}: C^{\infty}(\R^d,\mathcal{H})\mapsto C^{\infty}_{0}(\R^d,\mathcal{H})\\
%((\varphi^{\prime})^{-1})^*\Phi^{\prime} \psi^{\prime})(\psi\Phi^{-1} \varphi^*)&=\psi_{0}\varphi_{0}^*\otimes \phi_{UU^{\prime}}^{-1}: C^{\infty}(\R^d,\mathcal{H}) \mapsto C^{\infty}_{0}(\R^d,\mathcal{H}).
%\end{aligned}
%\end{equation}

We now introduce semiclassical Sobolev spaces on the flat Hilbert bundle $F\rightarrow M$. First, for any Radon measure $d\mu$ on $\R^d$, $s\in \mathbb{R}$, $0<h\leq 1$, and $u =   u^{i} e_{i}\in C_0^\infty(\mathbb{R}^d;\mathcal{H})$, we define the semiclassical Sobolev norm of $u$ as:
\begin{equation*}
    \|u\|^2_{H^{s}_{h}(\R^d,\mathcal{H},d\mu)} :=\sum_{i\in I}\|u^{i}\|^2_{H^{s}_{h}(\R^d,d\mu)}=\sum_{i\in I}\|\Oph(\langle \xi\rangle^s)u_{i}\|_{L^2(\R^d,d\mu)}, 
\end{equation*}
where $\langle \xi\rangle= \sqrt{1+|\xi|^2}$ denotes the Japanese bracket and $\Oph$ denotes the standard semiclassical quantization on $\mathbb{R}^d$:
\begin{equation}\label{standquant}
    (\Oph(a) ) u(x)= \frac{1}{(2\pi h)^d} \int_{\mathbb{R}^{2d}} a(x,\xi;h ) u(y) e^{i(x-y)\cdot \xi/h} \,dyd\xi.
\end{equation}

\begin{definition}\label{def:sobolev}
    Let $d\mu$ be a Radon measure on $M$ and $\mathcal{U}=\left\{\left(\Phi_{\alpha}, U_{\alpha}, \varphi_{\alpha},\chi_{\alpha}\right)\right\}_{\alpha\in \Lambda}$ be a finite cover of cutoff normal charts of $F$ such that $\sum_{\alpha\in \Lambda}\chi_{\alpha}^2=1$. For any $s \in \mathbb{R}$, $0<h\leq 1$, and $u \in C^{\infty}(M; F)$, we define the semiclassical Sobolev norm of $u$ as:
\begin{equation}\label{defofsobolevnorm}
\|u\|_{H_h^s(M, F,\mathcal{U},d\mu)}^2:=\sum_{\alpha\in \Lambda}\left\|\big((\varphi_{\alpha}^{-1})^*\Phi_{\alpha}\chi_{\alpha}\big)u\right\|_{H_h^s(\R^d,  \mathcal{H},(\varphi_{\alpha})_{*}d\mu)}^2,
\end{equation}
where $(\varphi_{\alpha})_{*}d\mu$ denotes the pushforward of the measure $d\mu$ to $\R^d$.
\end{definition}

%Recall that $M$ is equipped with a volume measure $d\mu$ induced by the Riemannian metric $g$. More generally, we allow $d\mu$ to be any smooth positive density on $M$. We take a finite cover of cutoff normal charts $\mathcal{U}=\left\{\left(\Phi_{\alpha}, U_{\alpha}, \varphi_{\alpha},\chi_{\alpha}\right)\right\}_{\alpha\in \Lambda}$ of $F$ such that $\sum_{\alpha\in \Lambda}\chi_{\alpha}^2=1$.  

\begin{remark}
    The semiclassical Sobolev norm coincides with the $L^2$ norm when $s=0$:
\begin{equation*}
\begin{aligned}
    \|u\|_{H_h^{s=0}(M, F,\mathcal{U},d\mu)}^2 
    &=\sum_{\alpha\in \Lambda} \int_{\mathbb{R}^d} |\chi_\alpha(\varphi_\alpha^{-1} (x) )|^2\big\| \Phi_\alpha \big( u( \varphi_\alpha^{-1} (x) ) \big) \big\|_{\mathcal{H}}^2 \, (\varphi_{\alpha})_{*}d\mu(x) \\
    &= \sum_{\alpha\in \Lambda} \int_{U_\alpha} |\chi_\alpha( y )|^2\big\| \Phi_\alpha \big( u( y) \big) \big\|_{\mathcal{H}}^2 \, d\mu(y) \\
    &= \int_M \|u(x)\|^2_{F_x}\,d\mu(x),
\end{aligned}
\end{equation*}
where we used $\sum_{\alpha\in \Lambda}\chi_{\alpha}^2=1$ in the last equality. 
\end{remark}

Let $F=F^\rho$ be the flat Hilbert bundle associated with the representation $\rho : \pi_1(M) \rightarrow \mathrm{U}(\mathcal{H})$ (see Section \ref{subsec: twisted laplacian}). At first glance, the semiclassical Sobolev norm defined above appears to depend on the choices of the cover $\mathcal{U}$ and the measure $d\mu$. We will show, however, that the semiclassical Sobolev norms corresponding to different choices of $(\mathcal{U},d\mu)$ are equivalent. More importantly, this equivalence holds uniformly with respect to the semiclassical parameter $h$, the Hilbert space $\mathcal{H}$ and the representation $\rho$.

%More generally, our Sobolev norms are depending on the choice of cover $\mathcal{U}$ and smooth positive density $d\mu$, but we can prove that they are uniformly equivalent with respect to the unitary representation $(\mathcal{H},\rho)$.

\begin{proposition}\label{prop:equivsobolev}
For any $s \in \mathbb{R}$, $0<h\leq 1$, and any two pairs of $(\mathcal{U},d\mu)$ and $(\mathcal{U}',d\mu')$ satisfying the assumptions of Definition \ref{def:sobolev}. There is a constant $C>0$, depending only on $d,s,\mathcal{U}, \mathcal{U}' , d\mu,d\mu'$ and in particular independent of $(h,\mathcal{H},\rho)$, such that for every $u \in C^{\infty}(M;F^{\rho})$ one has
\begin{equation*}
    \frac{1}{C}\|u\|_{H^{s}_{h}(M;F^{\rho},\mathcal{U}^{\prime},d\mu^{\prime})}\leq \|u\|_{H^{s}_{h}(M;F^{\rho},\mathcal{U},d\mu)}\leq C\|u\|_{H^{s}_{h}(M;F^{\rho},\mathcal{U}^{\prime},d\mu^{\prime})}.
\end{equation*}

%For any two finite cutoff parallel local trivialization coordinate chart covers $\mathcal{U}:=\{(\Phi_{\alpha},U_{\alpha},\varphi_{\alpha},\chi_{\alpha})\}_{\alpha \in \Lambda}$ and $\mathcal{U}^{\prime}:=\{(\Phi_{\beta}^{\prime},U_{\beta}^{\prime},\varphi_{\beta}^{\prime},\chi_{\beta}^{\prime})\}_{\beta \in \Lambda^{\prime}}$, and any two positive density $d\mu$ and $d\mu^{\prime}$, then there exists a constant $C>0$ independent of $(\mathcal{H},\rho)$ and $h$ such that 
\end{proposition}

\begin{proof}
%First, we notice that for any $A: H^{s}_{h}(\R^d,w_{1}(x)dx) \to H^{t}_{h}(\R^d,w_{2}(x))$ which is uniformly bounded for $0<h\leq1$, and any $\phi \in \mathrm{U}(\mathcal{H})$, we have $A\otimes \phi: H^{s}_{h}(\R^d,\mathcal{H},w_{1}(x)dx)\to H^{t}_{h}(\R^d,\mathcal{H},w_{2}(x)dx)$ is uniformly bounded and 
%\begin{equation}\label{sobolevtensorextension}
%\|A\otimes \phi\|_{H^{s}_{h}(\R^d,\mathcal{H},w_{1}(x)dx)\to H^{t}_{h}(\R^d,\mathcal{H},w_{2}(x)dx)}=\|A\|_{H^{s}_{h}(\R^d,w_{1}(x)dx) \to H^{t}_{h}(\R^d,w_{2}(x)dx)}.
%\end{equation}
%Now we return to our proof. By symmetry, we only need to prove one side. 

%Let $\chi_{1,\beta}^{\prime}\in C_{0}^{\infty}(U_{\beta}^{\prime})$ be $\psi_{\beta}^{\prime}\equiv 1$ on $\mathrm{supp}\,\chi_{\beta}^{\prime}$. 

Denote $\mathcal{U}=\{(\Phi_{\alpha},U_{\alpha},\varphi_{\alpha},\chi_{\alpha})\}_{\alpha \in \Lambda}$ and $\mathcal{U}^{\prime}=\{(\Phi_{\beta}^{\prime},U_{\beta}^{\prime},\varphi_{\beta}^{\prime},\chi_{\beta}^{\prime})\}_{\beta \in \Lambda^{\prime}}$. Let $u \in C^{\infty}(M;F^{\rho})$. Since $\sum_{\alpha \in \Lambda} \chi_\alpha^2 =1$ and 
\begin{equation*}
    \big[\chi_\alpha\Phi_\alpha^{-1}\varphi_\alpha^* \big] \circ \big[ (\varphi_\alpha^{-1})^*\Phi_\alpha\chi_\alpha  \big] u = \chi_\alpha^2 u, \qquad \forall \alpha \in \Lambda,
\end{equation*}
we get
\begin{equation}\label{for49}
\begin{aligned}
    ((\varphi_{\beta}')^{-1})^*\Phi_{\beta}'\chi_{\beta}'u
    &=((\varphi_{\beta}')^{-1})^*\Phi_{\beta}'\chi_{\beta}' \sum_{\alpha \in \Lambda} \chi_\alpha^2u \\
    &=((\varphi_{\beta}')^{-1})^*\Phi_{\beta}'\chi_{\beta}' \sum_{\alpha \in \Lambda}\big[\chi_\alpha\Phi_\alpha^{-1}\varphi_\alpha^* \big] \circ \big[ (\varphi_\alpha^{-1})^*\Phi_\alpha\chi_\alpha  \big] u \\
    & = \sum_{\alpha \in \Lambda} G_{\beta \alpha} \circ \big[ (\varphi_\alpha^{-1})^*\Phi_\alpha\chi_\alpha  \big] u ,
\end{aligned}
\end{equation}
where $G_{\beta \alpha}$ is the composition operator appear in \eqref{transitionmapform}:
\begin{equation*}
    G_{\beta \alpha} : = \big[((\varphi_{\beta}')^{-1})^*\Phi_{\beta}'\chi_{\beta}'  \big]\circ \big[\chi_\alpha\Phi_\alpha^{-1}\varphi_\alpha^* \big] :C^{\infty}(\R^d;\mathcal{H})\rightarrow C^{\infty}_{0}(\R^d;\mathcal{H}).
\end{equation*}
Denote $\varphi_{\alpha \beta } = \varphi_\alpha \circ (\varphi_\beta')^{-1}$, $\chi_{\beta \alpha} = (\chi_\beta' \chi_\alpha) \circ (\varphi_\beta')^{-1}$ and $\Phi_{\beta \alpha} = \Phi_\beta' \circ \Phi^{-1}_\alpha \in \mathrm{U}(\mathcal{H})$. By \eqref{transitionmapform}, we have
\begin{equation}\label{normdependence}
\begin{aligned}
    &\| G_{\beta \alpha} \|_{H^{s}_{h}(\R^d,\mathcal{H},(\varphi_{\alpha})_{*}d\mu)\to H^{s}_{h}(\R^d,\mathcal{H},(\varphi'_{\beta})_{*}d\mu')} \\
    &= \| \varphi_{\alpha \beta }^* \chi_{\beta \alpha } \otimes \Phi_{\beta \alpha } \|_{H^{s}_{h}(\R^d,\mathcal{H},(\varphi_{\alpha})_{*}d\mu)\to H^{s}_{h}(\R^d,\mathcal{H},(\varphi'_{\beta})_{*}d\mu')} \\
     &= \| \varphi_{\alpha \beta }^* \chi_{\beta \alpha }  \|_{H^{s}_{h}(\R^d,(\varphi_{\alpha})_{*}d\mu)\to H^{s}_{h}(\R^d,(\varphi'_{\beta})_{*}d\mu')} ,
\end{aligned}
\end{equation}
which is a constant depending only on $d,s,\alpha,\beta, d\mu,d\mu'$ (the independence with respect to $h$ is nontrivial, see \cite[Page 561]{dyatlovzworski2019resonances}). It follows from \eqref{for49} and \eqref{normdependence} that 
\begin{equation*}
\begin{aligned}
\|u\|_{H^{s}_{h}(M;F^{\rho},\mathcal{U}',d\mu')}^2&=\sum_{\beta \in \Lambda'}\| ((\varphi_{\beta}')^{-1})^*\Phi_{\beta}'\chi_{\beta}'u \|^2_{H^{s}_{h}(\R^d,\mathcal{H},(\varphi'_{\beta})_{*}d\mu')}\\
& \leq \sum_{\beta \in \Lambda'} \sum_{\alpha \in \Lambda} \|G_{\beta \alpha } ((\varphi_\alpha^{-1})^*\Phi_\alpha\chi_\alpha u)\|_{H^{s}_{h}(\R^d,\mathcal{H},(\varphi'_{\beta})_{*}d\mu')}  \\
& \leq C \sum_{\alpha \in \Lambda}\| (\varphi_{\alpha}^{-1})^*\Phi_{\alpha}\chi_{\alpha}u \|^2_{H^{s}_{h}(\R^d,\mathcal{H},(\varphi_{\alpha})_{*}d\mu)}  \\
& = C \|u\|_{H^{s}_{h}(M;F^{\rho},\mathcal{U},d\mu)}^2,
\end{aligned}
\end{equation*}
where $C>0$ is a constant depending only on $d,s,\mathcal{U}, \mathcal{U}' , d\mu,d\mu'$. The same argument yields the other direction. The proof is complete.
\end{proof}

In view of Proposition \ref{prop:equivsobolev}, we can safely omit $\mathcal{U}$ and $d\mu$ from the notation and  simply write $\|\cdot\|_{H^{s}_{h}(M;F^{\rho})}$ for the semiclassical Sobolev norm. The semiclassical Sobolev space $H^{s}_{h}(M;F^{\rho})$ is defined as the completion of $C^{\infty}(M;F^{\rho})$ with respect to this norm. For $h=1$, we suppress the subscript $h$. 

\begin{proposition}\label{inclusionofSobolev}
    We have the following properties for the semiclassical Sobolev spaces:
    \begin{itemize}
        \item[(a)] Fix $s\in \mathbb{R}$. For different values of $h\in (0,1]$, the spaces $H^s_h(M;F^\rho)$ are the same sets with different norms. These norms are mutually equivalent, but the equivalence is not uniform with respect to $h$.
        \item[(b)] For any $s>t$ and $0<h \leq 1$, one has the continuous embedding $H^{s}_{h}(M;F^{\rho})\subset H^{t}_{h}(M;F^{\rho})$. If $\mathcal{H}$ is finite-dimensional, the natural inclusion $\iota: H^{s}_{h}(M;F^{\rho})\rightarrow H^{t}_{h}(M;F^{\rho})$ is compact. In contrast, when $\mathcal{H}$ is infinite-dimensional, $\iota$ is not compact.
        \item[(c)] For any cutoff normal chart $(\Phi,U,\varphi,\chi)$ and $s\in \R$, there exists a constant $C>0$ independent of $(\mathcal{H},\rho,h)$ such that the operators in \eqref{cutoffoperator} are continuous:
\begin{equation}\label{sobolevnormofcutoff}
\|(\varphi^{-1})^*\Phi \chi\|_{H^{s}_h(M;F^{\rho})\to H^{s}_h(\R^d,\mathcal{H})} +\|\chi\Phi^{-1}\varphi^*\|_{ H^{s}_h(\R^d,\mathcal{H})\to H^{s}_h(M;F^{\rho})}\leq C.
\end{equation}
\end{itemize}
\end{proposition}
\begin{proof}
    The proof in the setting of flat Hilbert bundles follows the same lines as in the manifold case and is entirely routine, so we omit the details. For the corresponding argument for Sobolev spaces on manifolds, we refer to \cite[Appendix E]{dyatlovzworski2019resonances}.
    
    We give an example to show that the natural inclusion $\iota: H^{s}_{h}(M;F^{\rho})\rightarrow H^{t}_{h}(M;F^{\rho})$ is not compact when $s>t$ and $\mathcal{H}$ is infinite-dimensional. Let $(\Phi,U,\varphi,\chi)$ be a cutoff normal chart of $F^\rho$, $g \in H^{s}_{h}(M)$ with $\mathrm{supp}\,g \subset U$ and $u_i = g \otimes  \Phi^{-1}( e_i)$ for every $i \in I$. Then 
    \begin{equation*}
        \|u_i \|_{H^{s}_h(M;F^{\rho})} = \| g \|_{H^s_h(M)}, \qquad \forall i \in I,
    \end{equation*}
    and
    \begin{equation*}
        \|u_i -u_j \|^2_{H^{t}_h(M;F^{\rho})}  = 2\| g \|_{H^t_h(M)} ,\qquad \forall i,j\in I.
    \end{equation*}
    Therefore, $\{u_i\}_{i\in I}$ is bounded in $H^{s}_h(M;F^{\rho})$ but has no convergent subsequence in $H^{t}_h(M;F^{\rho})$. The inclusion $\iota: H^{s}_{h}(M;F^{\rho})\rightarrow H^{t}_{h}(M;F^{\rho})$ is not compact when $\mathcal{H}$ is infinite-dimensional.
\end{proof}

%\begin{equation}
% h^{s}\|u\|_{H^{-s}_h}\lesssim\|u\|_{H^{-s}}\lesssim\|u\|_{H^{-s}_h} \lesssim \|u\|_{H^{s}_h}\lesssim  \|u\|_{H^{s}}\lesssim h^{-s}\|u\|_{H^{s}_h}.
%\end{equation}

%By Definition \eqref{defofsobolevnorm}, \eqref{transitionmapform} and \eqref{sobolevtensorextension}, we have the uniform upper bounded of cutoff parallel local trivialization linear operators:

\subsection{Quantization of scalar symbols on flat Hilbert bundles}\label{subsec:quansymb}

In the general theory of pseudodifferential operators on vector bundles (\emph{c.f.} the semiclassical analysis on flat bundles and on principal bundles with finite-dimensional fibers in \cite{mama2024semiclassical,cekiclefeuvre2024semiclassical}), symbols are smooth sections of $\pi_0^{*}\mathrm{End}(F)$, where $\pi_0^{*}\mathrm{End}(F)$ denotes the pullback of the endomorphism bundle $\mathrm{End}(F)$ over $T^*M$ via the natural projection $\pi_{0}: T^*M \to M$. 

In this work, however, we restrict our attention to scalar symbols of the form $a(x,\xi;h)\mathrm{Id}_{\mathcal{H}}$, which are canonically identified as $a \in C^{\infty}(T^*M)$. We start by showing how pseudodifferential operators transform under a change of coordinates in Euclidean space.

%We start from the local model, notice that \cite[Proposition E.10]{dyatlovzworski2019resonances} implies that $\Psi^{\bullet}_{h}(\R^d) \otimes \mathrm{Id}_{\mathcal{H}}$ is invariant under the transition map \eqref{transitionmapform}.

%\begin{lemma}\label{changevariable}
%Let $U\subset \R^d$ be an open set, $\varphi: U \rightarrow \varphi(U) \subset \R^d$ be a diffeomorphism, $\chi \in C_0^\infty(U)$, and $V ,W\in \mathrm{U}(\mathcal{H})$. For any $k\in \mathbb{R}$ and any $a \in S^k (T^* \mathbb{R}^d )$, there exists $b \in S^k (T^* \mathbb{R}^d )$ such that
%\begin{equation*}
%    (\chi \varphi^*\otimes  V )(\mathrm{Op}_h(a)\otimes \mathrm{Id}_{\mathcal{H}})((\varphi^{-1})^*\chi\otimes W)=\mathrm{Op}_h(b)\otimes VW .
%\end{equation*}
%Here $S^k\left(T^* \mathbb{R}^d\right)$ is defined as in \eqref{standardsymbolclass}, with $M$ replaced by $\R^d$. Moreover, the symbol $b$ admits an asymptotic expansion $b \sim \sum_{j=0}^{\infty} h^j L_j(a \circ \widetilde{\varphi})$, in the sense that,
%\begin{equation*}
%    b-\sum_{j=0}^{N-1} h^j L_j(a \circ \widetilde{\varphi})  \in h^N S^{k-N}(T^*\R^d), \qquad \forall N \in \mathbb{N}.
%\end{equation*}
%Here $\widetilde{\varphi}=(\varphi,(d\varphi)^{-t}): T^*U\to T^*\varphi(U)$ denotes the natural lift of $\varphi$ to the cotangent bundle, $L_0=\chi^2$, and for $j\geq 1$, $L_j$ is an $h$-independent differential operator of order $2 j$ on $T^*U$ with coefficients compactly supported in $x$.
%\end{lemma}
%\begin{proof}
%    See \cite[Proposition E.10]{dyatlovzworski2019resonances}.
%\end{proof}

\begin{lemma}\label{changevariable}
Let $U\subset \R^d$ be an open set, $\varphi: U \rightarrow \varphi(U) \subset \R^d$ be a diffeomorphism, and $\chi \in C_0^\infty(U)$. For any $k\in \mathbb{R}$ and any $a \in S^k (T^* \mathbb{R}^d )$, there exists $b \in S^k (T^* \mathbb{R}^d )$ such that
\begin{equation*}
    \chi \varphi^*\mathrm{Op}_h(a)(\varphi^{-1})^*\chi=\mathrm{Op}_h(b) .
\end{equation*}
Here $S^k\left(T^* \mathbb{R}^d\right)$ is defined as in \eqref{standardsymbolclass}, with $M$ replaced by $\R^d$. Moreover, the symbol $b$ admits an asymptotic expansion $b \sim \sum_{j=0}^{\infty} h^j L_j(a \circ \widetilde{\varphi})$, in the sense that,
\begin{equation*}
    b-\sum_{j=0}^{N-1} h^j L_j(a \circ \widetilde{\varphi})  \in h^N S^{k-N}(T^*\R^d), \qquad \forall N \in \mathbb{N}.
\end{equation*}
Here $\widetilde{\varphi}=(\varphi,(d\varphi)^{-t}): T^*U\to T^*\varphi(U)$ denotes the natural lift of $\varphi$ to the cotangent bundle, $L_0=\chi^2$, and, for $j\geq 1$, $L_j$ is an $h$-independent differential operator of order $2 j$ on $T^*U$ with coefficients compactly supported in $x$.
\end{lemma}
\begin{proof}
    See \cite[Proposition E.10]{dyatlovzworski2019resonances}.
\end{proof}

To give a fine description of the pseudodifferential operators with scalar symbols on flat Hilbert bundles, we introduce the so-called \emph{locally fiberwise diagonal operators}. For brevity of notation, we write $A = A(\theta_1;\cdots;\theta_n)$ to indicate that $A$ denotes a family of operators indexed by the parameters $\theta_1;\cdots;\theta_n$.

\begin{definition}\label{localdiag}
We say that a $(\mathcal{H},\rho,h)$-family of operators 
\begin{equation*}
     A = A(\mathcal{H};\rho;h) : C^{\infty}(M;F^{\rho}) \to C^\infty(M;F^{\rho}) , \qquad \forall  (\mathcal{H},\rho) \in \mathcal{C},\, 0<h\leq 1,
\end{equation*}
is \emph{locally fiberwise diagonal} if, for any $(\mathcal{H},\rho) \in \mathcal{C}$, $0<h\leq 1$, and any two cutoff normal charts $(\Phi_i,U_i,\varphi_i,\chi_i)$, $i=1,2$, of $F^\rho$, with the transition map $\Phi_{12}:=\Phi_1\circ \Phi_{2}^{-1}$, there exists an operator $A_{\varphi_{12},\chi_{12}}(h): C^{\infty}(\mathbb R^d)\to C^{\infty}(\mathbb R^d)$, independent of $(\mathcal{H},\rho)$, such that
\begin{equation}\label{localdiagonal}
(\varphi^{-1}_1)^{*}\Phi_1\chi_1 A(\mathcal{H};\rho;h) \chi_2\Phi^{-1}_2\varphi^{*}_2
= A_{\varphi_{12},\chi_{12}} (h) \otimes \Phi_{12}.
\end{equation}
\end{definition}

The residual class of semiclassical pseudodifferential operators $A(h)\in h^{\infty}\Psi^{-\infty}(\R^d)$ is defined as in \cite[Definition E.11]{dyatlovzworski2019resonances}. We now introduce the corresponding notion of \emph{locally fiberwise diagonal residual class} in the setting of flat Hilbert bundles.

%it is equivalent to that for any $N\in \mathbb{N}$, there exists $C_{N}>0$ such that $\|A(h)\|_{H^{-N}(\R^d)\to H^{N}(\R^d)}=\mathcal{O}_{N}(h^{\infty})$. We define the general residual class for semiclassical pseudodifferential operators with scalar symbols by Sobolev norm. 

\begin{definition}\label{residualclass}
Let $A=A(\mathcal{H};\rho;h)$ be a family of locally fiberwise diagonal operators. We say that $A \in h^{\infty}\Psi^{-\infty,sc}(M;F^\rho)$ or $A=\mathcal{O}(h^{\infty})_{\Psi^{-\infty,sc}(M;F^\rho)}$, if
\begin{equation}\label{smoothsobolevnorm}
\|A(\mathcal{H};\rho;h)\|_{H^{-N}(M;F^{\rho})\to H^{N}(M;F^{\rho})} \leq C h^N, \qquad \forall (\mathcal{H},\rho) \in \mathcal{C},\, 0<h\leq 1,N \in \mathbb{N}.
\end{equation}
where $C=C(N)>0$ is a constant depending only on $N$. 
\end{definition}
\begin{remark}\label{remark46}
    It is easy to show that $A  \in h^{\infty}\Psi^{-\infty,sc}(M;F^\rho)$ if and only if $A_{\varphi_{12},\chi_{12}} = \Oph(a_{\varphi_{12},\chi_{12}})$ for some $a_{\varphi_{12},\chi_{12}} \in h^\infty S^{-\infty}(T^*\R^d)$ in \eqref{localdiagonal}.
\end{remark}

We are ready to define pseudodifferential operators with scalar symbols on flat Hilbert bundles. These operators are given by finite sums of local pseudodifferential operators defined in cutoff normal charts, modulo the locally fiberwise diagonal residual class.

%by gluing together pullbacks of locally fiberwise diagonal operators on $C^{\infty}(\R^d,\mathcal{H})$ using cutoff parallel local trivialization charts. 

\begin{definition}\label{defofscalarsemipseudo}
Let $k \in \mathbb{R}$. We say that a $(\mathcal{H},\rho,h)$-family of operators 
\begin{equation*}
     A = A(\mathcal{H};\rho;h) : C^{\infty}(M;F^{\rho}) \to C^\infty(M;F^{\rho}) , \qquad \forall  (\mathcal{H},\rho) \in \mathcal{C},\, 0<h\leq 1,
\end{equation*}
belongs to the class $\Psi_{h}^{k,sc}(M;F^\rho)$, if
\begin{equation}\label{defformpseudo}
A=\sum_{\alpha\in \Lambda} \chi_{\alpha }\Phi_{\alpha}^{-1} \varphi_{\alpha}^*( \Oph(a_{\alpha})\otimes \mathrm{Id}_{\mathcal{H}})(\varphi_{\alpha}^{-1})^* \Phi_{\alpha}\chi_{\alpha}  +  \mathcal{O}(h^{\infty})_{\Psi^{-\infty,sc}(M;F^\rho)}.
\end{equation}
Here $ \{ (\Phi_{\alpha}, U_{\alpha}, \varphi_{\alpha},\chi_{\alpha} ) \}_{\alpha\in \Lambda}$ is a finite family of cutoff normal charts of $F^\rho$ and $a_{\alpha}\in S^{k}(T^*\mathbb{R}^d)$. We denote by
\begin{equation*}
    \Psi^{\infty,sc}_h(M;F^{\rho}):=\bigcup\limits_{k\in \R} \Psi^{k,sc}_h(M;F^{\rho})
\end{equation*}
the set of semiclassical pseudodifferential operators with scalar symbols on flat Hilbert bundle $F^{\rho}$.
\end{definition}

Analogous to \cite[Proposition E.13]{dyatlovzworski2019resonances}, we have the following characterization for $\Psi^{k,sc}_{h}(M;F^{\rho})$:

\begin{proposition}\label{criterionofpseudo}
A family of operators $A=A(\mathcal{H};\rho;h): C^{\infty}(M;F^{\rho}) \rightarrow C^{\infty}(M;F^{\rho})$ belongs to the class $\Psi_h^{k,sc}(M;F^\rho)$ if and only if the following two conditions are satisfied:
\begin{itemize}
\item For any cutoff normal chart $(\Phi, U,\varphi,\chi)$, there exists $a_{\varphi, \chi} \in S^k(T^*\mathbb{R}^d )$, independent of $(\mathcal{H},\rho)$, such that
\begin{equation}\label{pseudooflocal}
(\varphi^{-1})^*\Phi \chi A (\mathcal{H};\rho;h) \chi \Phi^{-1} \varphi^*=\Oph(a_{\varphi,\chi})\otimes \mathrm{Id}_{\mathcal{H}}.
\end{equation}
\item For each $\psi, \psi^{\prime} \in C^{\infty}(M)$ such that $\operatorname{supp} \psi \cap \operatorname{supp} \psi^{\prime}=\varnothing$, we have 
\begin{equation}\label{disresi}
    \psi A \psi^{\prime}=\mathcal{O} (h^{\infty} )_{\Psi^{-\infty,sc}(M;F^\rho)}.
\end{equation}
\end{itemize}
\end{proposition}
\begin{proof}
    $(\Rightarrow)$ Assume that $A\in \Psi_h^{k,sc}(M;F^\rho)$. For any cutoff normal chart $(\Phi, U,\varphi,\chi)$, by \eqref{defformpseudo}, we have
    \begin{equation*}
    \begin{aligned}
        (\varphi^{-1})^*\Phi \chi A (\mathcal{H};\rho;h) \chi \Phi^{-1} \varphi^* &= \sum_{\alpha\in \Lambda} (\varphi^{-1})^*\Phi \chi \chi_{\alpha}\Phi_{\alpha}^{-1} \varphi_{\alpha}^*
        ( \Oph(a_{\alpha})\otimes \mathrm{Id}_{\mathcal{H}})
        (\varphi_{\alpha}^{-1})^* \Phi_{\alpha}\chi_{\alpha} \chi \Phi^{-1} \varphi^* \\
        & \quad \, + (\varphi^{-1})^*\Phi \chi R\chi \Phi^{-1} \varphi^* ,
    \end{aligned}
    \end{equation*}
where $ \{ (\Phi_{\alpha}, U_{\alpha}, \varphi_{\alpha},\chi_{\alpha} ) \}_{\alpha\in \Lambda}$ is a finite family of cutoff normal charts, $a_{\alpha}\in S^{k}(T^*\mathbb{R}^d)$, and $R \in h^{\infty}\Psi^{-\infty,sc}(M;F^\rho)$. It follows that $(\varphi^{-1})^*\Phi \chi R\chi \Phi^{-1} \varphi^* = \Oph(r) \otimes \mathrm{Id}_{\mathcal{H}}$ for some $r \in h^\infty S^{-\infty}(T^*\R^d)$ (see Remark \ref{remark46}). By \eqref{transitionmapform}, we have
\begin{equation*}
    (\varphi^{-1})^*\Phi \chi \chi_{\alpha}\Phi_{\alpha}^{-1} \varphi_{\alpha}^* = \chi_{0\alpha } \varphi^*_{0\alpha} \otimes \Phi_{0\alpha}, \quad
    (\varphi_{\alpha}^{-1})^* \Phi_{\alpha}\chi_{\alpha} \chi \Phi^{-1} \varphi^* = (\varphi_{0\alpha }^{-1})^* \chi_{0\alpha}  \otimes \Phi_{0\alpha}^{-1},
\end{equation*}
where $\Phi_{0\alpha} =\Phi \circ \Phi_\alpha^{-1}$, $\varphi_{0\alpha} = \varphi  \circ \varphi_\alpha^{-1}$, and $\chi_{0 \alpha } = (\chi \chi_\alpha  ) \circ \varphi_\alpha^{-1}$. Therefore,
\begin{equation*}
    (\varphi^{-1})^*\Phi \chi A (\mathcal{H};\rho;h) \chi \Phi^{-1} \varphi^*  = \left(\sum_{\alpha\in \Lambda} \chi_{0\alpha } \varphi^*_{0\alpha } \Oph(a_{\alpha})
       (\varphi_{0\alpha }^{-1})^* \chi_{0\alpha}  + \Oph(r)\right)\otimes \mathrm{Id}_{\mathcal{H}}.
\end{equation*}
We get \eqref{pseudooflocal} from Proposition \ref{changevariable}. 

For each $\psi, \psi^{\prime} \in C^{\infty}(M)$ such that $\operatorname{supp} \psi \cap \operatorname{supp} \psi^{\prime}=\varnothing$. Let $\widetilde{\chi}_\alpha \in C_0^\infty(U_\alpha) $ such that $\widetilde{\chi}_\alpha =1$ near $\mathrm{supp}\, \chi_\alpha$. Then,  
\begin{equation}\label{dissupp}
    \psi\chi_{\alpha}\Phi_{\alpha}^{-1} \varphi_{\alpha}^*=\widetilde{\chi}_{\alpha}\Phi_{\alpha}^{-1} \varphi_{\alpha}^* \psi_\alpha \quad \mathrm{and} \quad  (\varphi_{\alpha}^{-1})^* \Phi_{\alpha}\chi_{\alpha} \psi' = \psi'_\alpha (\varphi_{\alpha}^{-1})^* \Phi_{\alpha} \widetilde{\chi}_{\alpha},
\end{equation}
where $\psi_\alpha  = (\varphi_{\alpha}^{-1})^*  (\psi \chi_\alpha )$ and $\psi'_\alpha  = (\varphi_{\alpha}^{-1})^*  (\chi_\alpha \psi')$ have disjoint supports. Substituting \eqref{dissupp} into \eqref{defformpseudo}, we obtain
\begin{equation}\label{formpsiApsi}
    \psi A(\mathcal{H};\rho;h)\psi' = \sum_{\alpha\in \Lambda} \widetilde{\chi}_{\alpha}\Phi_{\alpha}^{-1} \varphi_{\alpha}^* 
        \big( (\psi_\alpha \Oph(a_{\alpha} )\psi'_\alpha) \otimes \mathrm{Id}_{\mathcal{H}} \big)
        (\varphi_{\alpha}^{-1})^* \Phi_{\alpha} \widetilde{\chi}_{\alpha} + \psi R\psi',
\end{equation}
where $R \in h^{\infty}\Psi^{-\infty,sc}(M;F^\rho)$. Since $\psi_\alpha \Oph(a_{\alpha} )\psi'_\alpha \in h^\infty \Psi^{-\infty}(T^*\R^d) $ (see \cite[Proposition E.13]{dyatlovzworski2019resonances}), using \eqref{sobolevnormofcutoff}, we conclude that $\psi A \psi'$ satisfies \eqref{smoothsobolevnorm}. To prove \eqref{disresi}, it remains to show that $\psi A \psi'$ is locally fiberwise diagonal. For any two cutoff normal charts $(\Phi_i,U_i,\varphi_i,\chi_i)$, $i=1,2$, of $F^\rho$, by \eqref{formpsiApsi}, we have
\begin{equation}\label{LFDveri}
    \begin{aligned}
       & (\varphi^{-1}_1)^{*}\Phi_1\chi_1 \psi A(\mathcal{H};\rho;h) \psi '\chi_2\Phi^{-1}_2\varphi^{*}_2 \\
&= \sum_{\alpha\in \Lambda} \Big( (\varphi^{-1}_1)^{*}\Phi_1\chi_1\widetilde{\chi}_{\alpha}\Phi_{\alpha}^{-1} \varphi_{\alpha}^* \Big)
        \Big( (\psi_\alpha \Oph(a_{\alpha} )\psi'_\alpha) \otimes \mathrm{Id}_{\mathcal{H}} \Big)
        \Big(  (\varphi_{\alpha}^{-1})^* \Phi_{\alpha} \widetilde{\chi}_{\alpha} \chi_2\Phi^{-1}_2\varphi^{*}_2 \Big)  \\
        &\quad\, + (\varphi^{-1}_1)^{*}\Phi_1\chi_1\psi R\psi' \chi_2\Phi^{-1}_2\varphi^{*}_2.
    \end{aligned}
\end{equation}
Let $\widetilde{\chi}_i \in C_0^\infty(U_i) $ such that $\widetilde{\chi}_i =1$ near $\mathrm{supp}\, \chi_i$, $i=1,2$. In the spirit of \eqref{transitionmapform} and \eqref{dissupp}, we have
\begin{equation}\label{inspr}
    \begin{aligned}
        & (\varphi^{-1}_1)^{*}\Phi_1\chi_1\widetilde{\chi}_{\alpha}\Phi_{\alpha}^{-1} \varphi_{\alpha}^* = \chi_{1\alpha} \varphi_{1\alpha}^* \otimes \Phi_{1\alpha}, &&  (\varphi_{\alpha}^{-1})^* \Phi_{\alpha} \widetilde{\chi}_{\alpha} \chi_2\Phi^{-1}_2\varphi^{*}_2 = (\varphi_{2\alpha }^{-1})^* \chi_{2\alpha } \otimes \Phi_{\alpha 2}, \\
        & (\varphi^{-1}_1)^{*}\Phi_1\chi_1\psi= \psi_1 (\varphi^{-1}_1)^{*}\Phi_1\widetilde{\chi}_1, && \psi' \chi_2\Phi^{-1}_2\varphi^{*}_2 = \widetilde{\chi}_2 \Phi^{-1}_2\varphi^{*}_2 \psi'_2,
    \end{aligned}
\end{equation}
where $\varphi_{i\alpha} = \varphi_i  \circ \varphi_\alpha^{-1}$ and $\chi_{i \alpha } = (\chi_i \chi_\alpha  ) \circ \varphi_\alpha^{-1}$, $i=1,2$, and $\psi_1  = (\varphi_1^{-1})^*  (\chi_1 \psi  )$ and $\psi'_2 = (\varphi_2^{-1})^*  ( \psi' \chi_2)$. Substituting \eqref{inspr} into \eqref{LFDveri}, we get
\begin{equation}
    \begin{aligned}
        (\varphi^{-1}_1)^{*}\Phi_1\chi_1 \psi A(\mathcal{H};\rho;h) \psi '\chi_2\Phi^{-1}_2\varphi^{*}_2 
&= \sum_{\alpha\in \Lambda}  \chi_{1\alpha} \varphi_{1\alpha}^* \psi_\alpha \Oph(a_{\alpha} )\psi'_\alpha (\varphi_{2\alpha }^{-1})^* \chi_{2\alpha }  \otimes \Phi_{12}    \\
        &\quad\, + \psi_1 (\varphi^{-1}_1)^{*}\Phi_1\widetilde{\chi}_1 R\widetilde{\chi}_2 \Phi^{-1}_2\varphi^{*}_2 \psi'_2 ,
    \end{aligned}
\end{equation}
where $\chi_{1\alpha} \varphi_{1\alpha}^* \psi_\alpha \Oph(a_{\alpha} )\psi'_\alpha (\varphi_{2\alpha }^{-1})^* \chi_{2\alpha }  \in h^\infty \Psi^{-\infty}(T^*\R^d) $, and $(\varphi^{-1}_1)^{*}\Phi_1\widetilde{\chi}_1 R\widetilde{\chi}_2 \Phi^{-1}_2\varphi^{*}_2 = \Oph(r_{12}) \otimes  \Phi_{12}$ for some $r_{12} \in h^\infty S^{-\infty}(T^*\R^d)$. We conclude that $\psi A \psi'$ is locally fiberwise diagonal.

$(\Leftarrow)$ The converse direction can be proved by following the argument of \cite[Proposition E.13]{dyatlovzworski2019resonances}. The only modification is that the cutoff charts used there are replaced by cutoff normal charts in our setting. We therefore omit the details.
\end{proof}

\begin{remark}
    From the proof of Proposition \ref{criterionofpseudo}, it can be seen that $A \in h^{\infty} \Psi^{-\infty,sc}(M;F^\rho)$ if and only if $a_{\varphi,\chi} \in h^\infty S^{-\infty}(T^*\R^d)$ in \eqref{pseudooflocal}. Therefore,
    \begin{equation}\label{charremain}
        h^{\infty} \Psi^{-\infty,sc}(M;F^\rho) = \bigcap_{N\in \N} h^N \Psi_h^{-N,sc}(M;F^\rho).
    \end{equation}
The formula \eqref{charremain} is useful in practice: to verify that an operator $A$ belongs to $h^{\infty} \Psi^{-\infty,sc}(M;F^\rho)$, it is not necessary to check that $A$ is locally fiberwise diagonal. It suffices to show that $A \in h^N \Psi_h^{-N,sc}(M;F^\rho)$ for every $N\in \N$.
\end{remark}

We now establish the quantization rules between $S^k(T^*M)$ and $\Psi^{k,sc}_h(M;F^\rho)$ for any $k \in \mathbb{R}$:
\begin{itemize}
\item the principal symbol map $\sigma^\rho_h:\Psi^{k,sc}_h(M;F^\rho) \rightarrow S^k(T^*M) /hS^{k-1}(T^*M)  $;
\item the quantization map $\Oph^\rho:S^k(T^*M) \rightarrow \Psi^{k,sc}_h(M;F^\rho)$.
\end{itemize}

\begin{definition}\label{defofprincipalsymbol}
Let $A=A(\mathcal{H};\rho;h) \in \Psi_h^{k,sc}(M;F^\rho)$. Assume that $A$ has the form of \eqref{defformpseudo}:
\begin{equation*}
    A=\sum_{\alpha\in \Lambda} \chi_{\alpha }\Phi_{\alpha}^{-1} \varphi_{\alpha}^*( \Oph(a_{\alpha})\otimes \mathrm{Id}_{\mathcal{H}})(\varphi_{\alpha}^{-1})^* \Phi_{\alpha}\chi_{\alpha}  +  \mathcal{O}(h^{\infty})_{\Psi^{-\infty,sc}(M;F^\rho)}.
\end{equation*}
The principal symbol of $A$ is the unique element in $S^k (T^* M)/h S^{k-1}(T^* M)$ defined as 
\begin{equation}
    \sigma_h^{\rho}(A)(x,\xi;h)=\sum_{\alpha \in \Lambda} \chi_{\alpha}(x)^2 a_{\alpha}^0\circ \widetilde{\varphi}_{\alpha}(x,\xi),
\end{equation}
where $a_{\alpha}^{0}=\sigma(\Oph(a_\alpha)) \in S^{k}(T^*\mathbb{R}^d)/h S^{k-1}(T^*\mathbb{R}^d)$ denotes the principal symbol of $\Oph(a_\alpha)$. If $a_0(x,\xi) \in S^{k} (T^*M)$ is independent of $h$ such that $a_{0} \in \sigma_{h}^{\rho}(A)  $, we identify $a_{0} $ as $\sigma_h^{\rho}(A)$.
\end{definition}
\begin{remark}
    For each cutoff normal chart $(\Phi, U,\varphi, \chi)$ and the corresponding symbol  $a_{\varphi, \chi}$ given in \eqref{pseudooflocal}. It follows from Lemma \ref{changevariable} that
\begin{equation*}
    \chi(x)^2 \sigma_h^{\rho}(A)(x,\xi)=a_{\varphi, \chi}\circ \widetilde{\varphi} \quad \mathrm{mod} \ hS^{k-1} (T^*M) \quad \text { on } T^* U.
\end{equation*}
\end{remark}

%In what follows, we quantize the scalar symbol $a(x,\xi;h) \in S^{k}_{1,0}(T^*M)$, see \cite[E.1.2]{dyatlovzworski2019resonances}, into a family of operators on the smooth section space $\Oph^{\rho}(a): C^{\infty}(M;F^{\rho}) \to C^{\infty}(M;F^{\rho})$ for any $(\mathcal{H},\rho)$.

%the scalar quantization of $a$ into the operator $\Oph^{\rho}(a): C^{\infty}(M;F^{\rho}) \to C^{\infty}(M;F^{\rho})$ is given by

\begin{definition}\label{diagofquantizeofscalar}
Choose a finite family of cutoff normal charts $ \{ (\Phi_{\alpha}, U_{\alpha}, \varphi_{\alpha},\chi_{\alpha} ) \}_{\alpha\in \Lambda}$ such that $\sum_{\alpha}\chi_\alpha\equiv 1$ on $M$. For each $\alpha \in \Lambda$, let $\chi'_\alpha\in C_{0}^{\infty}(U_{\alpha})$ such that $\chi'_{\alpha}\equiv 1$ near $\mathrm{supp}\,\chi_{\alpha}$. For any $k \in \mathbb{R}$ and  $a=a(x,\xi;h)\in S^k (T^*M)$, we define the quantization map as:
\begin{equation}\label{formquan}
    \Oph^{\rho}(a):=\sum_{\alpha\in \Lambda} \chi'_{\alpha} \Phi^{-1}_{\alpha} \varphi_{\alpha}^* \big( \mathrm{Op}_h\big( (\chi_{\alpha}a) \circ \widetilde{\varphi}_{\alpha}^{-1}\big)\otimes \mathrm{Id}_{\mathcal{H}} \big)(\varphi_{\alpha}^{-1})^*\Phi_{\alpha} \chi'_{\alpha},
\end{equation}
where $\widetilde{\varphi}_{\alpha}=(\varphi_{\alpha},(d\varphi_{\alpha})^{-t}): T^*U_{\alpha}\to T^*\varphi_{\alpha}(U_{\alpha})$ denotes the natural lift of $\varphi_{\alpha}$ to the cotangent bundle.
%If $a(x) \in C^{\infty}(M)$, then $\Oph^{\rho}(a)u(x)=a(x)u(x)$ for any $u \in C^{\infty}(M;F^{\rho})$.
\end{definition}
\begin{remark}
    Since $\widetilde{\varphi}_{\alpha}$ preserves the Kohn-Nirenberg symbol class $S^k(T^*M)$, the quantization map $\Oph^{\rho}$ defined in \eqref{formquan} maps $S^k(T^*M)$ into $\Psi^{k,sc}_h(M;F^\rho)$. This quantization map is non-canonical, in the sense that it depends on the choice of cutoff normal charts. Nevertheless, using Proposition \ref{criterionofpseudo}, it is easy to show that different choices of such charts lead to quantizations that differ only by terms in $h\Psi^{k-1,sc}_h(M;F^\rho)$.
\end{remark}

%We define the principal symbol as in \cite[Proposition E.14]{dyatlovzworski2019resonances}:

One key observation of this work is the \emph{uniform} boundedness of $\Oph^{\rho}(a)$ between Sobolev spaces.

\begin{proposition}\label{uniformboundedtheorem}
 For any $a(x, \xi ; h) \in S^k\left(T^* M\right)$ and any $s \in \mathbb{R}$, there exists a constant $C=C(a, s, k)>0$, independent of $( \mathcal{H},\rho) \in \mathcal{C}$ and $0<h\leq 1$, such that
\begin{equation}
\left\|\mathrm{Op}_h^\rho(a)\right\|_{H_h^s\left(M;F^{\rho}\right) \rightarrow H_h^{s-k}\left(M;F^{\rho}\right)} \leq C .
\end{equation}
\end{proposition}
\begin{proof}
By \cite[Proposition E.22]{dyatlovzworski2019resonances}, the operator $\Oph(a): H^{s}_{h}(M)\to H^{s-k}_{h}(M)$ is uniformly bounded in $h\in(0,1]$. This, combined with \eqref{sobolevnormofcutoff}, which states that the operator-norm of the cutoff operator is uniformly bounded in $(\mathcal{H},\rho)$, yields the uniform boundedness of $A: H^{s}_{h}(M;F^{\rho})\to H^{s-k}_{h}(M;F^{\rho})$ with respect to $(\mathcal{H},\rho,h)$.
\end{proof}

The following algebraic properties are also important:

\begin{proposition}\label{propofquantization}
We have the following properties:
\begin{itemize}
    \item[(a)] For any $k\in \R$ and $A \in \Psi_h^{k,sc}(M;F^\rho)$, there exists $a\in S^k\left(T^* M\right)$ such that
\begin{equation}\label{pseudoequalsquant}
A=\mathrm{Op}_h^{\rho}(a)+\mathcal{O} (h^{\infty} )_{\Psi^{-\infty,sc}(M;F^\rho)}.
\end{equation}
\item[(b)] For any $k\in \R$ and $A \in \Psi_h^{k,sc}(M;F^\rho)$, the formal adjoint $A^*$ belongs to $\Psi_h^{k,sc}(M;F^\rho)$, and
\begin{equation*}
    \sigma_h^{\rho}\left(A^*\right)=\overline{\sigma_h^{\rho}(A)}.
\end{equation*}
\item[(c)] $h^{\infty}\Psi^{-\infty,sc}(M;F^{\rho})$ is an ideal of $\Psi^{\infty,sc}_h(M;F^{\rho})$, \emph{i.e.}, for any $A \in \Psi^{\infty,sc}_h(M;F^{\rho})$ and $R\in h^{\infty}\Psi^{-\infty,sc}(M;F^{\rho})$, we have $AR,\, RA \in h^{\infty}\Psi^{-\infty,sc}(M;F^{\rho})$.
\item[(d)] For any $k,\ell\in \R$ and $A \in \Psi_h^{k,sc}(M;F^\rho)$, $B \in \Psi_h^{\ell,sc}(M;F^\rho)$, we have:
\begin{itemize}
\item[$\bullet$] $AB \in \Psi_h^{k+\ell,sc}(M;F^\rho)$ and $\sigma^\rho_h (AB)=\sigma^\rho_h (A)\sigma^\rho_h (B)$.
\item[$\bullet$]
$[A,B]=AB-BA \in h\Psi_h^{k+\ell-1,sc}(M;F^\rho)$ and $\sigma_h^{\rho} (\frac{i}{h}[A, B] )  = \{\sigma_h^{\rho}(A), \sigma_h^{\rho}(B) \}$. Here $\{\cdot,\cdot\}$ denotes the Poisson bracket on $T^*M$.
\end{itemize}
\end{itemize}
\end{proposition}

\begin{proof}
The proofs for (a) and (b) are routine; we refer to \cite[Proposition E.16-17]{dyatlovzworski2019resonances}. The only difference here is the additional requirement that operators in $h^{\infty} \Psi^{-\infty,sc}(M;F^\rho)$ be locally fiberwise diagonal. This condition can be mediated using formula \eqref{charremain}.

For (c), let $A \in \Psi^{k,sc}_h(M;F^{\rho})$ for some $k \in \R$ and $R\in h^{\infty}\Psi^{-\infty,sc}(M;F^{\rho})$. By Proposition \ref{uniformboundedtheorem} and (a), $AR$ satisfies \eqref{smoothsobolevnorm}. It remains to prove that $AR$ is locally fiberwise diagonal. For any two cutoff normal charts $(\Phi_i,U_i,\varphi_i,\chi_i)$, $i=1,2$, of $F^\rho$, we have
\begin{equation*}
        (\varphi^{-1}_1)^{*}\Phi_1\chi_1 AR \chi_2\Phi^{-1}_2\varphi^{*}_2  = (\varphi^{-1}_1)^{*}\Phi_1\chi_1 \chi A\chi^2R \chi_2\Phi^{-1}_2\varphi^{*}_2 + (\varphi^{-1}_1)^{*}\Phi_1\chi_1 A(1-\chi^2)R \chi_2\Phi^{-1}_2\varphi^{*}_2,
\end{equation*}
where $\chi \in C_0^\infty(U_1) $ such that $\chi =1$ near $\mathrm{supp}\, \chi_1$. For the first part of the RHS, since $(\varphi^{-1}_1)^{*}\Phi_1\chi_1 = \chi' (\varphi^{-1}_1)^{*}\Phi_1$ with $\chi' = (\varphi^{-1}_1)^{*} \chi_1$, and $\chi^2 = \chi \Phi_1^{-1} \varphi^*_1 (\varphi^{-1}_1)^{*}\Phi_1\chi$, we get
\begin{equation*}
        (\varphi^{-1}_1)^{*}\Phi_1\chi_1\chi A\chi^2R \chi_2\Phi^{-1}_2\varphi^{*}_2 = \chi' \Big( (\varphi^{-1}_1)^{*}\Phi_1 \chi A\chi \Phi_1^{-1} \varphi^*_1 \Big) \Big( (\varphi^{-1}_1)^{*}\Phi_1\chi R\chi_2\Phi^{-1}_2\varphi^{*}_2 \Big) . 
\end{equation*}
By definition, the above is equal to $\Oph(r) \otimes \Phi_{12}$ for some $r \in h^\infty S^{-\infty}(T^*\R^d)$. For the second part, choose a finite family of cutoff normal charts $ \{ (\Phi_{\alpha}, U_{\alpha}, \varphi_{\alpha},\chi_{\alpha} ) \}_{\alpha\in \Lambda}$ such that $\sum_{\alpha \in \Lambda} \chi_\alpha^2 =1$, we have
\begin{equation*}
        (\varphi^{-1}_1)^{*}\Phi_1\chi_1 A(1-\chi^2)R \chi_2\Phi^{-1}_2\varphi^{*}_2 = \sum_{\alpha \in \Lambda} \Big( (\varphi^{-1}_1)^{*}\Phi_1 \chi A' \chi_\alpha \Phi_\alpha^{-1} \varphi_{\alpha}^* \Big)  \Big( (\varphi^{-1}_\alpha)^{*}\Phi_\alpha \chi_\alpha  R \chi_2\Phi^{-1}_2\varphi^{*}_2  \Big),
\end{equation*}
where $A'=\chi_1 A(1-\chi^2) =\mathcal{O} (h^{\infty} )_{\Psi^{-\infty,sc}(M;F^\rho)}$ by \eqref{disresi}. By definition, the above is also equal to $\Oph(r') \otimes \Phi_{12}$ for some $r' \in h^\infty S^{-\infty}(T^*\R^d)$. We conclude that $AR $ is locally fiberwise diagonal. The case for $RA$ can be proved similarly. Item (d) is a direct consequence of (c) and Proposition \ref{criterionofpseudo}, we refer to \cite[Proposition E.17]{dyatlovzworski2019resonances}.
\end{proof}

%\begin{proposition}
%    The class of locally fiberwise diagonal operators is closed under composition.
%\end{proposition}
%\begin{proof}
 %   Let $\{ (\Phi_\alpha,U_\alpha,\varphi_\alpha,\chi_\alpha)\}_{\alpha\in \Lambda}$ be a finite cover of cutoff normal charts of $F^\rho$ with $\sum_{\alpha \in \Lambda} \chi_\alpha^2 =1$. Then, for any two locally fiberwise diagonal operators $A$ and $B$, and any two cutoff normal charts $(\Phi_i,U_i,\varphi_i,\chi_i)$, $i=1,2$, we have
  %  \begin{equation*}
   %     \begin{aligned}
    %        &(\varphi^{-1}_1)^{*}\Phi_1\chi_1 A(\mathcal{H};\rho;h)B(\mathcal{H};\rho;h) \chi_2\Phi^{-1}_2\varphi^{*}_2 \\
     %       &=\sum_{\alpha \in \Lambda} \Big((\varphi^{-1}_1)^{*}\Phi_1\chi_1 A(\mathcal{H};\rho;h) \chi_\alpha\Phi^{-1}_\alpha\varphi^{*}_\alpha \Big) \Big((\varphi^{-1}_\alpha)^{*}\Phi_\alpha\chi_\alpha B(\mathcal{H};\rho;h) \chi_2\Phi^{-1}_2 \varphi^{*}_2 \Big) \\
      %      &= \sum_{\alpha \in \Lambda} \Big( A_{\varphi_{1\alpha}, \chi_{1\alpha}}(h)\otimes \Phi_{1\alpha} \Big) \Big(  B_{\varphi_{\alpha 2}, \chi_{\alpha 2}}(h)\otimes \Phi_{\alpha 2}  \Big) \\
       %     &= \left( \sum_{\alpha \in \Lambda}  A_{\varphi_{1\alpha}, \chi_{1\alpha}}(h) B_{\varphi_{\alpha 2}, \chi_{\alpha 2}}(h) \right)\otimes \Phi_{12}  .
        %\end{aligned}
    %\end{equation*}
    %The proof is complete.
%\end{proof}

%We note that $\bigcap\limits_{k\in \mathbb{Z}}h^k\Psi^{k,sc}_{h}(M;F^\rho)=h^{\infty}\Psi^{-\infty,sc}(M;F^\rho)$.

Finally, by the same proof of the sharp Gårding inequality (see \cite[Proposition E.24]{dyatlovzworski2019resonances}) and noting that 
\begin{equation}\label{adjointofcutoff}
\begin{aligned}
\text{the adjoint of } & \  (\varphi^{-1})^*\Phi \chi: L^2(M;F)\to L^{2}(\mathbb{R}^d,\mathcal{H},\varphi_{*}d\mu) \\
\text{is} & \  \chi \Phi^{-1} \varphi^*: L^{2}(\mathbb{R}^d,\mathcal{H}, \varphi_{*}d\mu)\to L^2(M;F),
\end{aligned}
\end{equation}
we have the following version of sharp Gårding inequality:

\begin{proposition}\label{uniformsharpgarding}
    For any $a \in S^{0} (T^*M)$ with $a\geq 0$, there exists a constant $C=C(M,a)>0$, independent of $(\mathcal{H},\rho,h)$, such that 
    \begin{equation*}
        \langle \Oph^{\rho}(a)u,u\rangle_{L^2(M;F^\rho)}\geq -Ch \|u \|_{L^2(M;F^\rho)}^2.
    \end{equation*}
\end{proposition}

A direct consequence of the sharp Gårding inequality is the sharp bound on the norms of pseudodifferential operators (see \cite[Theorem 13.13]{zworski2012semiclassicalanalysis}): for any $a \in S^{0} (T^*M)$, there exists a constant $C=C(M,a)>0$, independent of $(\mathcal{H},\rho,h)$, such that
\begin{equation}\label{sharpnorm}
        \|\Oph^{\rho}(a)\|_{L^2(M;F^{\rho})\to L^2(M;F^{\rho})}\leq \sup_{(x,\xi)\in T^*M} |a(x,\xi)|+Ch.
\end{equation}

\subsection{Wavefront set, Ellipticity, and Functional calculus}

Let $\overline{T}^*M$ be the fiber-radial compactification of $T^*M$ (see \cite[Appendix E.1.3]{dyatlovzworski2019resonances}). For any $k\in \R$ and $a \in S^k(T^*M)$, let $\operatorname{ess-supp} a \subset \overline{T}^*M$ denote the essential support of $a$ (see \cite[Definition E.26]{dyatlovzworski2019resonances}).

%Since $A\in \Psi^{\bullet,sc}_{h}(M;F^{\rho})$ is quantized by symbol class $a \in S^{\bullet}(T^*M)$, we can define the wavefront set of $A$ by the essential support of $a$, $\operatorname{ess-supp}(a)\subset \overline{T}^*M$, where $\overline{T}^*M$ is the fiber-radial compactification of $T^*M$, see \cite[E.1.3 and Definition E.26]{dyatlovzworski2019resonances}.

\begin{definition}\label{wavefrontset}
Let $k\in \R$ and $A \in \Psi_h^{k,sc}(M;F^\rho)$. By \eqref{pseudoequalsquant}, there exists a symbol $a \in S^{k} (T^*M)$ such that $A=\mathrm{Op}_h^{\rho}(a)+\mathcal{O} (h^{\infty} )_{\Psi^{-\infty,sc}(M;F^\rho)}$. We define the semiclassical wavefront set $\mathrm{WF}_h(A) \subset \overline{T}^* M$ as
\begin{equation*}
    \mathrm{WF}_{h}(A)=\operatorname{ess-supp} a.
\end{equation*}
We say that $A$ is compactly microlocalized if $\mathrm{WF}_h(A)$ is a compact subset of $T^* M$. We denote
\begin{equation*}
    \Psi_h^{\text{comp},sc}(M;F^\rho) = \left\{A \in  \Psi_h^{\infty,sc}(M;F^\rho): A \textrm{ is compactly microlocalized} \right\}.
\end{equation*}
\end{definition}

\begin{definition}
    Let $k\in \R$ and $A,B \in \Psi_h^{k,sc}(M;F^\rho)$. For any open or closed set $U \subset \overline{T}^* M$, we say that 
    \begin{equation*}
        A=B+\mathcal{O}(h^{\infty})_{\Psi^{-\infty,sc}(M;F^\rho)}\quad \text{microlocally on } U,
    \end{equation*}
    in the case when $U$ is open, or
    \begin{equation*}
        A=B+\mathcal{O}(h^{\infty})_{\Psi^{-\infty,sc}(M;F^\rho)}\quad \text{microlocally near } U,
    \end{equation*}
    in the case when $U$ is closed, if $\mathrm{WF}_{h}(A-B)\cap U=\varnothing$.
\end{definition}

%We can easily show that $A$ and $B$ are microlocally equivalent near $U$ if and only if, for every $T\in \bigcup_{k\in \R}\Psi_{h}^{k,sc}(M;F^\rho)$ with $\mathrm{WF}_{h}(T)\subset U$, we have
%\begin{equation}
%(A-B)T\in h^{\infty}\Psi^{-\infty,sc}(M;F^\rho).\footnote{One may also require $T(A-B)\in h^{\infty}\Psi^{-\infty,sc}(M;F^\rho)$ for all such $T$.}
%\end{equation} 

From the microlocal partition of unity \cite[Proposition E.30]{dyatlovzworski2019resonances}, we get the following \emph{uniform} version of the microlocal partition of unity:
\begin{proposition}\label{microlocalpartitionofunity}
Assume that $K \subset \overline{T}^* M$ is compact and
$\{U_{\alpha}\}_{\alpha \in \Lambda}$ is a finite open cover of $K$. Then there exist compactly supported $X_{\alpha} \in \Psi_h^{\mathrm{comp},sc}(M;F^{\rho})$ such that $\mathrm{WF}_h (X_{\alpha} ) \subset U_{\alpha}$ and
\begin{equation*}
    \sum_{\alpha \in \Lambda} X_{\alpha}=I+\mathcal{O}(h^{\infty})_{\Psi^{-\infty,sc}(M;F^{\rho})} \quad \mathrm{ \ microlocally\ near\ } K .
\end{equation*}
\end{proposition}

\begin{definition}
    For any $k\in \R$ and $A \in \Psi^{k,sc}_{h}(M;F^\rho)$, we define the elliptic set $\operatorname{ell}_h(A) \subset \overline{T}^*M$ of $A$ as follows: we say that $(x_0,\xi_0) \in \operatorname{ell}_h(A)$ or $A$ is elliptic at $(x_0,\xi_0)$ if there exists a neighborhood $W \subset \overline{T}^*M$ of $(x_0,\xi_0)$ and a constant $c>0$ such that 
    \begin{equation*}
        |\sigma^\rho_h(A)(x,\xi)|>c\langle \xi \rangle^k,\qquad \forall (x,\xi) \in W\cap T^*M.
    \end{equation*}
    Note that the above definition is independent of the choice of the representative element in $\sigma^\rho_{h}(A) \in S^k(T^*M)/hS^{k-1}(T^*M)$.
\end{definition}

The elliptic set is important because it characterizes the microlocal invertibility of $A$ near elliptic points; we refer to \cite[Proposition E.32]{dyatlovzworski2019resonances}. The proof there is readily adapted to our setting:

\begin{proposition}\label{prop:microelliptic}
Let $k,\ell \in \R$, $A \in \Psi_h^{k,sc}(M;F^\rho)$, and $P \in \Psi_h^{\ell,sc}(M;F^\rho)$ such that $\mathrm{WF}_h(A) \subset \operatorname{ell}_h(P)$. Then there exist $Q, \,Q^{\prime} \in \Psi_h^{k-\ell,sc}(M;F^\rho)$ such that
\begin{equation}\label{ellipticparametrix}
A=PQ+\mathcal{O}(h^{\infty})_{\Psi^{-\infty,sc}(M;F^\rho)}=Q^{\prime}P+\mathcal{O}(h^{\infty})_{\Psi^{-\infty,sc}(M;F^\rho)},
\end{equation}
and $\mathrm{WF}_h(Q) \cup \mathrm{WF}_h (Q^{\prime} ) \subset \mathrm{WF}_h(A)$.
\end{proposition}

For any $k\in \R$, $P \in \Psi_{h}^{k,sc}(M;F^\rho)$ is said to be an elliptic operator if it is elliptic everywhere, \emph{i.e.} $\operatorname{ell}_h(P)=\overline{T}^*M$. It is equivalent to 
\begin{equation}
|\sigma_{h}^{\rho}(P)(x,\xi)| > c\langle \xi \rangle^{k}, \qquad \forall (x,\xi) \in T^*M, 
\end{equation}
for some $c>0$. 

\begin{proposition}\label{inverseofelliptic}
Let $k\in \R$ and $P \in \Psi_{h}^{k,sc}(M;F^\rho)$. By Proposition \ref{uniformboundedtheorem}, for any $s\in \R$,
\begin{equation*}
    P(\mathcal{H};\rho;h):H^{s+k}_h(M;F^\rho) \rightarrow H^{s}_h(M;F^\rho) 
\end{equation*}
is bounded uniformly in $(\mathcal{H},\rho,h)$. If $P$ is elliptic, then there exists $h_0\in (0,1]$, independent of $(\mathcal{H},\rho)$, such that for any $h\in (0,h_0)$, 
$P(\mathcal{H};\rho;h)$ is invertible and the inverse
 $P^{-1} \in \Psi_{h}^{-k,sc}(M;F^\rho)$.
\end{proposition}

\begin{proof}
Taking $A=I = \mathrm{Op}_h^\rho(1)$ in Proposition \ref{prop:microelliptic}, there exists $Q,\,Q^{\prime}\in \Psi^{-k,sc}_{h}(M;F^\rho)$ and $R,\,R^{\prime}\in h^{\infty}\Psi^{-\infty,sc}(M;F^\rho)$ such that 
\begin{equation*}
I=PQ+R=Q^{\prime}P+R^{\prime}.
\end{equation*}
Since $R \in h\Psi^0_h(M;F^\rho)$, by Proposition \ref{uniformboundedtheorem}, there exists $M_{s}>0$, independent of $(\mathcal{H},\rho,h)$, such that 
\begin{equation}\label{normR}
    \|R(\mathcal{H};\rho;h)\|_{H^s_h(M;F^\rho) \rightarrow H^s_h(M;F^\rho)} \leq M_{s}h.
\end{equation}
Taking $h_0 = (2M_s)^{-1}$. Then, for any $h \in (0,h_0)$, $P(\mathcal{H};\rho;h):H^{s+k}_h(M;F^\rho) \rightarrow H^{s}_h(M;F^\rho) $ is invertible, and the inverse $P^{-1}=Q(I-R)^{-1}:H^{s}_{h}(M;F^\rho)  \rightarrow H^{s+k}_{h}(M;F^{\rho})$ is bounded uniformly in $(\mathcal{H},\rho,h)$. 

We now show that $P^{-1} \in \Psi_{h}^{-k,sc}(M;F^\rho)$. We first check that $P^{-1} -Q$ is locally fiberwise diagonal. Since $R\in h^{\infty}\Psi^{-\infty,sc}(M;F^\rho)$ and $Q \in \Psi^{-k,sc}_{h}(M;F^\rho)$, by Proposition \ref{propofquantization}.(c), we have $QR^n \in h^{\infty}\Psi^{-\infty,sc}(M;F^\rho)$ for any $n\in \N^+$. In particular, $QR^n$ is locally fiberwise diagonal. For any two cutoff normal charts $(\Phi_i, U_i,\varphi_i,\chi_i)$, $i=1,2$, there exists an operator $A_{\varphi_{12},\chi_{12},n}: C^{\infty}(\mathbb R^d)\to C^{\infty}(\mathbb R^d)$, independent of $(\mathcal{H},\rho)$, such that
\begin{equation*}
    (\varphi_{1}^{-1})^*\Phi_{1}\chi_{1}QR^{n}\chi_{2} \Phi_2^{-1}\varphi_2^* = A_{\varphi_{12},\chi_{12},n} \otimes \Phi_{12}, \qquad \forall n \in \N^+.
\end{equation*}
By \eqref{sobolevnormofcutoff} and \eqref{normR}, for any $h\in (0,h_0)$ and $n\in \N^+$, we have
\begin{equation*}
    \|A_{\varphi_{12},\chi_{12},n}\|_{H^{s}_{h}(M;F^\rho)\to H^{s+k}_h(M;F^\rho)}\leq C \|QR^{n}\|_{H^{s}_{h}(M;F^\rho)\to H^{s+k}_h(M;F^\rho)}\leq C/2^n.
\end{equation*}
Thus, we have 
\begin{equation*}
    \widetilde{A}_{\varphi_{12},\chi_{12}}:=\sum_{n=1}^{\infty}A_{\varphi_{12},\chi_{12},n} \in \mathcal{L}(H^{s}_h(M;F^\rho),H^{s+k}_h(M;F^\rho)),
\end{equation*}
and
\begin{equation}\label{differentcpocharts}
(\varphi_{1}^{-1})^*\Phi_{1}\chi_{1}(P^{-1}-Q)\chi_{2} \Phi_2^{-1}\varphi_2^* =\widetilde{A}_{\varphi_{12},\chi_{12}}\otimes \Phi_{12}.
\end{equation}
Therefore, $P^{-1} -Q$ is locally fiberwise diagonal.

We now check \eqref{disresi} for $P^{-1}$. Let $\psi, \psi^{\prime} \in C^{\infty}(M)$ such that $\operatorname{supp} \psi \cap \operatorname{supp} \psi^{\prime}=\varnothing$. Since $\psi Q\psi'   \in h^{\infty}\Psi^{-\infty,sc}(M;F^\rho)$ and $P^{-1} -Q$ is locally fiberwise diagonal,
\begin{equation*}
    \psi P^{-1}\psi' = \psi Q\psi' +  \psi (P^{-1} - Q)\psi' 
\end{equation*}
is also locally fiberwise diagonal. On the other hand, since $R\in h^{\infty}\Psi^{-\infty,sc}(M;F^\rho)$ and $P^{-1}:H^{s}_{h}(M;F^\rho)  \rightarrow H^{s+k}_{h}(M;F^{\rho})$ is bounded uniformly in $(\mathcal{H},\rho,h)$, there exists a constant $C>0$ depending only on $N$ such that
\begin{equation*}
\begin{aligned}
     \| \psi P^{-1}\psi' \|_{H^{-N}(M;F^{\rho})\to H^{N}(M;F^{\rho})} & = \| \psi Q\psi' - \psi P^{-1} R \psi' \|_{H^{-N}(M;F^{\rho})\to H^{N}(M;F^{\rho})}  \\
     & \leq Ch^N
\end{aligned}
\end{equation*}
for any $N\in \N$. Therefore, $\psi P^{-1} \psi^{\prime}=\mathcal{O} (h^{\infty} )_{\Psi^{-\infty,sc}(M;F^\rho)}$. We have verified \eqref{disresi} for $P^{-1}$.

We now check \eqref{pseudooflocal} for $P^{-1}$. For any cutoff normal chart $(\Phi, U,\varphi,\chi)$, by \eqref{differentcpocharts}, there exists $A_{\varphi, \chi} \in \mathcal{L}(H^{s}_h(M;F^\rho),H^{s+k}_h(M;F^\rho))$, independent of $(\mathcal{H},\rho)$, such that
\begin{equation}\label{P-1A}
(\varphi^{-1})^*\Phi \chi P^{-1} \chi \Phi^{-1} \varphi^*=A_{\varphi,\chi}\otimes \mathrm{Id}_{\mathcal{H}}.
\end{equation}
It remains to verify that $A_{\varphi,\chi}\in \Psi^{-k}_{h}(\R^d)$. By the Beals's Theorem (see \cite[Theorem 9.12]{zworski2012semiclassicalanalysis}), it suffices to prove that
\begin{equation}\label{esti1ad}
    \big\|\operatorname{ad}_{x_{i_1}} \cdots \operatorname{ad}_{x_{i_N}} \operatorname{ad}_{hD_{x_{i_{N+1}}}} \cdots \operatorname{ad}_{hD_{x_{i_{N+N^\prime}}}} A_{\varphi,\chi}\big\|_{H_h^{s} (\mathbb{R}^d ) \rightarrow H_h^{s+k+N} (\mathbb{R}^d)}=\mathcal{O} (h^{N+N'})
\end{equation}
for all $N, N^{\prime} \in \mathbb{N}$ and all $i_k \in\{1, \cdots, d\}$, $1 \leq k \leq N+N^{\prime}$. Here $\mathrm{ad}_{A}B:=[A,B]$, $D_{x_{i}}=-i\partial_{x_i}$.

Let $\widetilde{\chi} \in C_{0}^{\infty}(\varphi(U))$ be such that $\chi \equiv 1$ near $\mathrm{supp}\,(\varphi^{-1})^*\chi$. For any linear function $\ell(x,\xi) \in S^{1}(\R^{2d})$, we denote $\ell(x,hD_{x})=\Oph(\ell) \in \Psi^{1}_{h}(\mathbb{R}^d)$. We define 
\begin{equation*}\label{linearsymbol}
L_{\ell(x,hD_{x}),\widetilde{\chi}}:=\varphi^*\Phi^{-1}\widetilde{\chi} \big(\ell(x,hD_{x})\otimes \mathrm{Id}_{\mathcal{H}} \big) \widetilde{\chi}\Phi(\varphi^{-1})^* .
\end{equation*}
It is clear that
\begin{equation}\label{proL}
    L_{x_{i},\widetilde{\chi}}\in \Psi^{0,sc}_{h}(M;F^\rho),\quad L_{hD_{x_i},\widetilde{\chi}}\in \Psi^{1,sc}_{h}(M;F^\rho), \qquad \forall i\in \{1,\cdots,d\}.
\end{equation}
%For any locally fiberwise diagonal linear operator $B:C^{\infty}(M;F^{\rho})\to C^{\infty}(M;F^{\rho})$, there exists $B_{\varphi,\chi}: C^{\infty}_0(\R^d)\to C^{\infty}(\R^d)$ such that $(\varphi^{-1})^{*}\Phi\chi B\chi\Phi^{-1}\varphi^{*}= B_{\varphi,\chi}\otimes \mathrm{Id}_{\mathcal H}$. 
We take $\chi^{\prime}\in C_{0}^{\infty}(U)$ with $\chi^{\prime}\equiv 1$ near $\mathrm{supp}\,\varphi^*\widetilde{\chi}$ such that
\begin{equation}\label{localformofad}
\begin{aligned}
 (\operatorname{ad}_{\ell(x, h D_x)}A_{\varphi,\chi} )\otimes \mathrm{Id}_{\mathcal{H}}
&=\operatorname{ad}_{\ell(x, h D_x)\otimes \mathrm{Id}_{\mathcal{H}}}\big( (\varphi^{-1})^{*}\Phi\chi P^{-1}\chi\Phi^{-1}\varphi^{*} \big)\\
&=(\varphi^{-1})^{*}\Phi\chi^{\prime} \big(\operatorname{ad}_{L_{\ell(x,hD_{x}),\widetilde{\chi}}}\chi P^{-1}\chi \big)\chi^{\prime}\Phi^{-1}\varphi^{*},
\end{aligned}
\end{equation}
where we used \eqref{P-1A} for the first equality. By iterating \eqref{localformofad} $N+N'$ times, we obtain
\begin{equation}\label{adA435}
\begin{aligned}
&\big(\operatorname{ad}_{x_{i_1}} \cdots \operatorname{ad}_{x_{i_N}} \operatorname{ad}_{hD_{x_{i_{N+1}}}} \cdots \operatorname{ad}_{hD_{x_{i_{N+N^\prime}}}} A_{\varphi,\chi}\big)\otimes\mathrm{Id}_{\mathcal{H}}\\
&\quad =(\varphi^{-1})^{*}\Phi\chi^{\prime}(\operatorname{ad}_{L_{1}}\operatorname{ad}_{L_{2}}\cdots\operatorname{ad}_{L_{N+N^{\prime}}}\chi P^{-1}\chi)\chi^{\prime}\Phi^{-1}\varphi^{*},
\end{aligned}
\end{equation}
where
\begin{equation*}L_k=\left\{
    \begin{aligned}
        & L_{x_{i_k}, \chi_k} ,&& k = 1,\cdots,N, \\
        & L_{hD_{x_{i_k}},\chi_k} && k = N+1, \cdots, N+N',
    \end{aligned}\right.
\end{equation*}
$\chi_k \in C_{0}^{\infty}(\varphi(U))$ be such that $\chi_k\equiv 1$ near $\mathrm{supp}\,\chi_{k+1}$, $\chi_{N+N^{\prime}}\equiv 1$ near $\mathrm{supp}\,(\varphi^{-1})^*\chi$, and $\chi^{\prime}\in C_{0}^{\infty}(U)$ with $\chi^{\prime}\equiv 1$ near $\mathrm{supp}\,\varphi^*\chi_1$.

By \eqref{sobolevnormofcutoff} and \eqref{adA435}, the desired estimate \eqref{esti1ad} is reduced to the following estimate:
\begin{equation}\label{commutatorestimateforinverse}
\| \operatorname{ad}_{L_{1}}\operatorname{ad}_{L_{2}}\cdots\operatorname{ad}_{L_{N+N^{\prime}}}\chi P^{-1}\chi \|_{H^{s} (M;F^\rho ) \rightarrow H_h^{s+k+N} (M;F^{\rho} )}=\mathcal{O} (h^{N+N^{\prime}} ).
\end{equation}
Since $P \in \Psi^{k,sc}_{h}(M;F^\rho)$, by Proposition \ref{propofquantization} and \eqref{proL}, we have
\begin{equation}\label{adofp}
\mathrm{ad}_{L_k}P \in \left\{
\begin{aligned}
    &h\Psi^{k-1,sc}_{h}(M;F^\rho),  && k=1,\cdots,N, \\
    &h\Psi^{k,sc}_{h}(M;F^\rho) ,&& k= N+1,\cdots,N+N'.
\end{aligned}\right.
\end{equation}
For any operators $A,B,L$, we have
\begin{equation}\label{propofad}
\mathrm{ad}_{L}P^{-1}=P^{-1}(\mathrm{ad}_{L}P)P^{-1},\quad \mathrm{ad}_{L}(AB)=(\mathrm{ad}_{L}A)B+A(\mathrm{ad}_{L}B).
\end{equation}
By Proposition \ref{uniformboundedtheorem}, \eqref{adofp} and \eqref{propofad}, for any $k_1 \in\{1,\cdots, N\}$ and $k_2\in\{ N+1,\cdots ,N+N'\}$, we have
\begin{equation*}
    \|\mathrm{ad}_{L_{k_1}}P^{-1}\|_{H^{s}_h(M;F^\rho)\to H^{s+k+1}_h(M;F^\rho)} + \|\mathrm{ad}_{L_{k_2}}P^{-1}\|_{H^{s}_h(M;F^\rho)\to H^{s+k}_h(M;F^\rho)}= \mathcal{O}(h) .
\end{equation*}
Then, \eqref{commutatorestimateforinverse} is deduced by induction and \eqref{propofad}. Therefore, $A_{\varphi,\chi}\in \Psi^{-k}_{h}(\R^d)$, and hence, $P^{-1}\in \Psi^{-k,sc}_h(M;F^\rho)$. The proof is complete.
\end{proof}

We now turn to the proof of Theorem \ref{functionalcalculus}. We begin with an outline of the argument. Recall that the semiclassical Laplacian $-h^2\Delta^{\rho}$ is diagonalized under parallel orthonormal local trivializations; see \eqref{diagofLaplacian1}. We take a finite cover of cutoff normal charts $\mathcal{U}=\{\Phi_{\alpha},U_{\alpha},\varphi_{\alpha},\chi_{\alpha}\}_{\alpha \in \Lambda}$ with $\sum_{\alpha\in\Lambda} \chi_\alpha\equiv 1$ on $M$. Let $\widetilde{\chi}_\alpha =\chi_{\alpha} \circ \varphi_\alpha^{-1}$ and $\psi_{\alpha}\in C_{0}^{\infty}(U_\alpha)$ such that $\psi_{\alpha}\equiv 1$ near $\mathrm{supp}\,\chi_\alpha$. Then, by \eqref{diagofLaplacian1}, it is easy to check that
\begin{equation*}
    -h^2\Delta^{\rho}=\sum_{\alpha\in \Lambda}\psi_{\alpha}\Phi_{\alpha}^{-1}\varphi_{\alpha}^* \big(\widetilde{\chi}_{\alpha}(-h^2\Delta_{g})\otimes \mathrm{Id}_{\mathcal{H}} \big)(\varphi_{\alpha}^{-1})^*\Phi_{\alpha}\psi_{\alpha}.
\end{equation*}
It follows that
\begin{equation}
-h^2\Delta^{\rho}\in \Psi^{2,sc}_h(M;F^\rho) \quad \text{and} \quad \sigma_{h}^{\rho}(-h^2\Delta^{\rho}) (x,\xi)=|\xi|^2_x.
\end{equation}

Recall the Helffer-Sjöstrand formula (see \cite[Theorem 8.1]{DimassiSjostrand99}): for any $f \in C_{0}^{\infty}(\mathbb{R})$ and any self-adjoint operator $P$ on a Hilbert space $\mathcal{H}$, we have
\begin{equation}\label{HSformula}
f(P)= \frac{1}{\pi } \int_{\mathbb{C}} \overline{\partial} \widetilde{f}(z)(P-z)^{-1} \,L(d z),
\end{equation}
where $L(d z)$ denotes the Lebesgue measure on $\mathbb{C}\simeq \R^2$, and $\widetilde{f} \in C_{0}^{\infty}(\mathbb{C})$ is an almost analytic continuation of $f$ such that 
\begin{itemize}
    \item $\widetilde{f} \big|_{\R}=f$;
    \item $\mathrm{supp}\,\widetilde{f}\subset \{z\in \C \ | \ \mathrm{Re} \, z \in \mathrm{supp}\,f, \ |\mathrm{Im} \, z|\leq 1 \} $;
    \item For any $N\in \N$, there exists $C=C(N)>0$ such that $|\overline{\partial}_z \widetilde{f}(z)| \leq C|\operatorname{Im}z|^N$.
\end{itemize}
The integral in \eqref{HSformula} is convergent in the sense of $\mathcal{L}(\mathcal{H})$-valued Riemann integration. 

In the Helffer–Sjöstrand formula \eqref{HSformula}, we may take $P=-h^2\Delta^{\rho}$. For $\mathrm{Im}\,z\neq 0$, $h^2\Delta^{\rho} +z\in \Psi^{2}_h(M;F^\rho)$ is elliptic, and hence 
\begin{equation*}
    (h^2\Delta^{\rho} + z)^{-1}\in \Psi^{-2}_h(M;F^\rho)
\end{equation*}
by Proposition \ref{inverseofelliptic}. However, in order
to apply the Helffer–Sjöstrand formula and conclude that $f(P) $ is a pseudodifferential operator, 
it is necessary to obtain uniform control on the growth of $(P-z)^{-1}$ and its commutator $\mathrm{ad}_{L} (P-z)^{-1}$ as $\mathrm{Im}\,z \rightarrow 0$. More precisely, we require that these operators grow at most polynomially in $|\mathrm{Im}\, z|^{-1}$. Then, repeating the argument of Proposition \ref{inverseofelliptic} yields that $f(P)$ is a pseudodifferential operator.

%\noindent \emph{Proof of Theorem \ref{functionalcalculus}.} 
%\hfill $\square$

%\hfill

In what follows, we establish a generalization of Theorem \ref{functionalcalculus} based on the idea mentioned above. Our proof follows \cite[Theorem 14.9]{zworski2012semiclassicalanalysis}.

\begin{theorem}\label{mainthmofquantized}
For any $k \in \mathbb{N}$, let $P\in \Psi^{k,sc}_h(M;F^\rho)$ be a self-adjoint operator on $L^2(M;F^{\rho})$ such that $P+i$ is elliptic in $\Psi^{k,sc}_h(M;F^\rho)$ and $\sigma_{h}^{\rho}(P)$ is independent of $h$. Then, for any $f\in C_{0}^{\infty}(\R)$, 
\begin{equation}
f(P) \in \Psi^{0,sc}_{h}(M;F^\rho) \quad \mathrm{with} \ \sigma_{h}^{\rho}(f(P))=f(\sigma_{h}^{\rho}(P)),
\end{equation}
Moreover, $f(P)$ is compactly microlocalized, and $\mathrm{WF}_{h}(f(P))\subset \{(x,\xi)\in T^*M: \sigma^\rho_{h}(P) \in \mathrm{supp}\,f\}$.
\end{theorem}
\begin{proof}
For simplicity of notations, below we write $H^s_h$ for $H^s_h(M;F^\rho)$ and $A\lesssim B$ for $A\leq CB$ for some constant $C>0$ independent of $(\mathcal{H},\rho,h,z)$. In this proof, we always assume that $\mathrm{Im} \, z \neq 0$. We divide the proof into several steps.

\emph{Step 1. We show that}
\begin{equation}\label{sobolevnormofresolvent}
\|(P-z)^{-1}\|_{H^{n}_h\to H^{n+k}_{h}}\lesssim |\mathrm{Im} \, z|^{-2^n} \langle z\rangle^{2^n}
 , \qquad \forall n\in \N.
\end{equation}

Since $P \in \Psi^{k,sc}_{h}(M;F^\rho)$ is self-adjoint on $L^2(M;F^{\rho})$, by the spectral theorem, we have 
\begin{equation}\label{spectraltheorem}
\|(P-z)^{-1}\|_{L^2\to L^2} \lesssim |\mathrm{Im} \, z|^{-1}.
\end{equation}
Since $P+i$ is elliptic in $\Psi^{k,sc}_h(M;F^\rho)$, $(P+i)^{-1}: L^2(M;F^{\rho})\to H^{k}_h(M;F^{\rho})$ is uniformly bounded in $(\mathcal{H},\rho,h)$. For any $u\in C^{\infty}(M;F^{\rho})$, applying \eqref{spectraltheorem} yields
\begin{equation*}
\|u\|_{H^{k}_h}\lesssim \|(P+i)u\|_{L^2}\lesssim \|(P-z)u\|_{L^2}+\langle z\rangle\|u\|_{L^2}\lesssim |\mathrm{Im} \, z|^{-1} \langle z\rangle \|(P-z)u\|_{L^2}.
\end{equation*}
This gives \eqref{sobolevnormofresolvent} for $n=0$. We proceed by induction. Assume that \eqref{sobolevnormofresolvent} holds for $n\in \N$. Let $D=\Oph^{\rho}(\langle \xi \rangle) \in \Psi^{1,sc}_h(M;F^\rho)$. Then,
\begin{equation*}
    D(P-z)^{-1}=(P-z)^{-1}D+(P-z)^{-1}[P,D](P-z)^{-1},
\end{equation*}
where $[P,D]\in h\Psi^k_h(M;F^\rho)$. For any $u \in H^{n+1}_{h}$, we have
\begin{equation*}
\begin{aligned}
\|(P-z)^{-1}u\|_{H^{n+k+1}_h} &\lesssim \|D(P-z)^{-1}u\|_{H^{n+k}_h}\\
& \leq\|(P-z)^{-1}Du\|_{H^{n+k}_h}+\|(P-z)^{-1}\|_{H^{n}_h\to H^{n+k}_{h}}  \|[P,D](P-z)^{-1}u\|_{H^{n}_h}\\
&\lesssim  \frac{\langle z\rangle^{2^n}}{|\mathrm{Im} \, z|^{2^n}}\|u\|_{H^{n+1}_h}+\frac{\langle z\rangle^{2^n}}{|\mathrm{Im} \, z|^{2^n}} \cdot \|[P,D]\|_{H^{n+k}_h\to H^n_h} \cdot\frac{\langle z\rangle^{2^n}}{|\mathrm{Im} \, z|^{2^n}}\|u\|_{H^n_h}\\
&\lesssim |\mathrm{Im} \, z|^{-2^{n+1}} \langle z\rangle^{2^{n+1}} \|u\|_{H^{n+1}_h}.
\end{aligned}
\end{equation*}
By induction, we obtain \eqref{sobolevnormofresolvent}.

\emph{Step 2. We show that $f(P)\in \Psi^{0,sc}_h(M;F^\rho)$.} 

We first verify \eqref{disresi} for $f(P)$. Let $\psi, \psi^{\prime} \in C^{\infty}(M)$ such that $\operatorname{supp} \psi \cap \operatorname{supp} \psi^{\prime}=\varnothing$. Let 
\begin{equation*}
    p_{0}=\sigma^{\rho}_{h}(P).
\end{equation*}
It follows that $|p_{0}(x,\xi)|\geq c\langle \xi \rangle_{x}^k-1$ for some $c>0$, and hence $P -z \in \Psi^{k,sc}_{h}(M;F^\rho)$ is elliptic. By Proposition \ref{inverseofelliptic}, $(P-z)^{-1} \in \Psi^{-k,sc}_{h}(M;F^\rho)$. In particular, $\psi (P-z)^{-1} \psi'$ is locally fiberwise diagonal. Using the Helffer-Sj\"{o}strand formula \eqref{HSformula} and the estimate \eqref{spectraltheorem}, we deduce that $f(P)\in \mathcal{L}(L^2(M;F^{\rho}))$ and $\psi f(P)\psi'$ is locally fiberwise diagonal. 

By the existence of elliptic parametrix \eqref{ellipticparametrix}, for any $T\in \mathbb{N}$, there exists $Q_{T}(z) \in \Psi^{-k,sc}_{h}(M;F^\rho)$ and $R_{T+1}(z) \in \Psi^{-T-1,sc}_h(M;F^\rho)$ such that
\begin{equation*}
    I=(P-z)Q_{T}(z)+h^{T+1}R_{T+1}(z) .
\end{equation*}
For $T=0$, $Q_{0}(z)=\Oph^{\rho}((p_{0}-z)^{-1})$. Computing the symbol of $R_{T+1}(z)$ under cutoff normal charts yields that, for any $T\in \N$ and any $s\in \R$, there exists $K=K_{s,T}>0$ such that
\begin{equation}\label{sobolevnormQR}
\|Q_{T}(z)\|_{H^{s}_{h}\to H^{s+k}_{h}}  + \|R_{T+1}(z)\|_{H_{h}^{s}\to H^{s+T+1}_{h}} \lesssim |\mathrm{Im}\, z|^{-K}.
\end{equation}
For any $N\in \N$, taking $T\geq 2N+|k|$, there exists $K=K_{N,T}>0$ such that
\begin{equation}
\begin{aligned}
&\|\psi  (P-z)^{-1}\psi'\|_{H^{-N}\to H^{N}}\\
%&\leq \|\psi  Q_{T}(z)\psi'\|_{H^{-N}\to H^{N}}+h^{T+1}\|\psi (z-P)^{-1}R_{T+1}(z)\psi'\|_{H^{-N}\to H^{N}}\\
& \lesssim \|\psi  Q_{T}(z)\psi' \|_{H^{-N}\to H^{N}}+h^{T+1}\|(P-z)^{-1}\|_{H_{h}^{-N+T+1}\to  H_{h}^{N}}\cdot\|R_{T+1}(z)\|_{H^{-N}_h\to H^{-N+T+1}_h} \\
&\lesssim  h^{T}|\mathrm{Im}\, z|^{-K }\langle z\rangle^{K },
\end{aligned}
\end{equation}
where we used \eqref{sobolevnormofresolvent} and \eqref{sobolevnormQR}. By Helffer-Sjöstrand formula \eqref{HSformula}, we obtain that there exists a constant $C=C_N>0$ such that
\begin{equation}\label{cutoffoffp}
\|\psi f(P) \psi'\|_{H^{-N}_h\to H^N_h}\leq C_N h^N.
\end{equation}
Therefore, \eqref{disresi} holds for $f(P)$.

We now verify \eqref{pseudooflocal} for $f(P)$. For any cutoff normal chart $(\Phi, U,\varphi,\chi)$, let $\ell(x,hD_x),\, \widetilde{\chi}$, $\chi'$ and $ L_k$, $k=1,\cdots,N+N'$, be defined as in the proof of Proposition \ref{inverseofelliptic}. By \eqref{propofad} and \eqref{sobolevnormofresolvent}, for any $s\in \R$, there exists $K=K_{s,k}>0$ such that
\begin{equation}\label{adofresolvent}
 \big\|\operatorname{ad}_{L_{\ell(x,hD_x),\widetilde{\chi}}}(P-z)^{-1} \big\|_{H^{s}_{h}\to H^{s+k}_{h}} \lesssim h|\mathrm{Im}\, z|^{-K}\langle z \rangle^{K}.
\end{equation}
Combining \eqref{sobolevnormofresolvent} and \eqref{adofresolvent}, we find a constant $K=K_{N,N',k}>0$ such that
\begin{equation}\label{adestiP-z}
\big\|\operatorname{ad}_{L_{1}} \cdots \operatorname{ad}_{L_{N+N'}} \chi (P-z)^{-1}\chi \big\|_{L^2\left(M;F^\rho\right) \rightarrow H_h^{k+N}\left(M;F^{\rho}\right)} \lesssim h^{N+N'}|\mathrm{Im}\, z|^{-K}\langle z\rangle^{K} .
\end{equation}

%Next we apply the semiclassical Beals Theorem to show $f(P) \in \Psi^{0,sc}_{h}(M;F^\rho)$. For any linear symbol $\ell(x,hD_x)\in \Psi^{1}_h(\R^d)$, $L_{\ell(x,hD_x),\widetilde{\psi}}\in \Psi^{1,sc}_h(M;F^\rho)$ is defined in \eqref{linearsymbol}. For any $N>0$, we take the same cutoff functions $\psi_{i} \in C_{0}^{\infty}(\varphi(U))$ for $i=1,\cdots,N$ such that $\psi_{i}\equiv 1$ on $\mathrm{\psi_{i+1}}$ and $\psi_{N+N^{\prime}}\equiv 1$ on $\mathrm{supp}((\varphi^{-1})^*\psi)$, and $\psi^{\prime}\in C_{0}^{\infty}(U)$ with $\psi^{\prime}\equiv 1$ on $\mathrm{supp}(\varphi^*\psi_1)$.
%For any $N$ linear symbols $\ell_{i}(x,hD_x) \in \Psi^1_{h}(\R^d)$, we define $$L_{i}=L_{\ell_{i}(x,hD_x),\psi_{i}}, \ \text{for any} \ i=1,\cdots,N.$$

As in Proposition \ref{inverseofelliptic}, we have
\begin{equation}\label{adfPHS}
\begin{aligned}
&\operatorname{ad}_{x_{i_1}} \cdots \operatorname{ad}_{x_{i_N}} \operatorname{ad}_{hD_{x_{i_{N+1}}}} \cdots \operatorname{ad}_{hD_{x_{i_{N+N^\prime}}}} (\varphi^{-1})^*\Phi \chi f(P)\chi\Phi^{-1}\varphi^*\\
&=(\varphi^{-1})^*\Phi \chi^{\prime} \big(\mathrm{ad}_{L_1}\cdots \mathrm{ad}_{L_{N+N'}} \chi f(P)\chi \big)\chi^{\prime}\Phi^{-1}\varphi^*\\
&= (\varphi^{-1})^*\Phi \chi^{\prime} \left(\frac{1}{\pi }\int_{\C}\overline{\partial}\widetilde{f}(z)\,\mathrm{ad}_{L_1}\cdots \mathrm{ad}_{L_{N+N'}} \chi (P-z)^{-1}\chi \,L(dz)\right)\chi^{\prime}\Phi^{-1}\varphi^*.
\end{aligned}
\end{equation}
By \eqref{adestiP-z} and \eqref{adfPHS}, we conclude that \eqref{pseudooflocal} holds for $f(P)$. Therefore, $f(P)\in \Psi^{0,sc}_{h}(M;F^\rho)$. The property $\sigma_{h}^{\rho}(f(P))=f(\sigma_{h}^{\rho}(P))$ is easily seen from the Helffer-Sj\"{o}strand formula.

\emph{Step 3. We prove that $\mathrm{WF}_{h}(f(P))\subset p_{0}^{-1}(\mathrm{supp}\,f)$ and $f(P)$ is compactly microlocalized.}  For any $(x_0,\xi_0)\notin p_{0}^{-1}(\mathrm{supp}\,f)$, there exist two open neighborhoods $U$ of $(x_0,\xi_0)$ and $V$ of $p_{0}^{-1}(\mathrm{supp}\,f)$ such that $U \cap V=\varnothing$. Let $T \in \Psi^{\mathrm{comp},sc}_{h}(M;F^\rho)$ with $\mathrm{WF}_{h}(T)\subset U$. Since 
\begin{equation*}
    \mathrm{WF}_{h}(T)\subset U\subset \mathrm{ell}_{h}(z-P), \qquad \forall z \in \mathrm{supp}\,\widetilde{f},
\end{equation*}
by the existence of elliptic parametrix \eqref{ellipticparametrix}, there exists $Q(z)\in \Psi^{\mathrm{comp},sc}_{h}(M;F^\rho)$ and $R(z)\in h^{\infty}\Psi^{-\infty,sc}(M;F^\rho)$ such that $T=(P-z)Q(z)+R(z)$ and $Q(z),\,R(z)$ are holomorphic for $z \in \mathrm{supp}\,\widetilde{f}$. Thus, we have
\begin{equation}
\begin{aligned}
f(P)T&=\frac{1}{\pi}\int_{\C}\overline{\partial}\widetilde{f}(z)(P-z)^{-1}T \, L(dz)\\
&=\frac{1}{\pi}\int_{\C}\overline{\partial}\widetilde{f}(z)Q(z)\,L(dz) + \frac{1}{\pi}\int_{\C}\overline{\partial}\widetilde{f}(z)(P-z)^{-1}R(z)\,L(dz)\\
&=\frac{1}{\pi}\int_{\C}\overline{\partial}\widetilde{f}(z)(P-z)^{-1}R(z)\, L(dz).
\end{aligned}
\end{equation}
By the almost analyticity of $\widetilde{f}$ and the estimate \eqref{sobolevnormofresolvent}, for any $z \in \mathrm{supp}\,\widetilde{f}$ and $N\in \mathbb{N}$, we have
\begin{equation*}
    \int_{\C}\|\overline{\partial}\widetilde{f}(z)(P-z)^{-1}\|_{H^N\to H^{N}}\,L(dz) \leq C_Nh^{-N},\qquad \|R(z)\|_{H^{-N}\to H^{N}}\leq C_Nh^{2N}.
\end{equation*}
We conclude that $f(P)T \in h^{\infty}\Psi^{-\infty,sc}(M;F^\rho)$. Therefore, $U\cap \mathrm{WF}_{h}(f(P))=\varnothing$ and $\mathrm{WF}_{h}(f(P))\subset p_0^{-1}(\mathrm{supp}\,f)$. From this and the fact that $|p_{0}(x,\xi)|\geq c\langle \xi \rangle_{x}^k-1$ for some $c>0$, we immediately obtain the compactness of $\mathrm{WF}_{h}(f(P))$.
\end{proof}

In Section \ref{sec: dynamics of geodesic flow}, we defined the homogeneous geodesic flow 
\begin{equation*}
    \varphi_{t}=\mathrm{exp}(tH_{p}), \qquad p=|\xi|_g.
\end{equation*}
Since $p$ is not smooth on the zero section, we remove it from consideration. To this end, we fix a cutoff function $\psi_P \in C_0^{\infty}((0, \infty) ; \mathbb{R})$ such that 
\begin{equation}\label{defofcutoff}
\psi_P(\lambda)=\sqrt{\lambda} \quad \text{for} \quad 1/16 \leq \lambda \leq 16.
\end{equation}
In the following, we assume that $\mathcal{P}\in \Psi_{h}^2(M;F^\rho)$ is self-adjoint on $L^2(M;F^{\rho})$ and satisfies 
\begin{equation*}
    \sigma_{h}^{\rho}(\mathcal{P})=p^2=|\xi|_g^2.
\end{equation*}
The operator $-h^2\Delta^{\rho}$ is a special case of such an operator $\mathcal{P}$. Since $\mathcal{P}$ satisfies the assumptions of Theorem \ref{mainthmofquantized}, we may define
\begin{equation}\label{defofcutofflaplacian}
\begin{aligned}
P:=\psi_P (\mathcal{P} ) \in \Psi^{\mathrm{comp},sc}_{h}(M;F^\rho), \quad \sigma_{h}^{\rho}(P)=|\xi|_g \ \text{on} \  \{1/4\leq |\xi|_{g}\leq 4 \}.
\end{aligned}
\end{equation}

We consider the cutoff Schrödinger propagator
\begin{equation}\label{schordingerpropagator}
U(t):=\exp  (-itP/h ): L^2(M;F^{\rho}) \rightarrow L^2(M;F^{\rho}).
\end{equation}
By stone's Theorem, $U(t)$ is unitary on $L^2(M;F^{\rho})$. For a bounded linear operator $A: L^2(M;F^{\rho}) \rightarrow L^2(M;F^{\rho})$, we define its ``Heisenberg evolution'' by
\begin{equation}\label{dyanamicsofoperator}
A(t):=U(-t) A U(t) .
\end{equation}
We study the propagator time that is independent of $h$. Applying the argument in the proof of Egorov's Theorem (see \cite[Theorem 11.1]{zworski2012semiclassicalanalysis}) and the similar argument used in the proof of Proposition \ref{inverseofelliptic}, we obtain the following \emph{uniform} version of Egorov's theorem up to finite time.

\begin{proposition}\label{uniformegorovtheorem}
Let $A=\Oph^{\rho}(a)$ with $a\in C_{0}^{\infty}(T^*M)$ such that $\mathrm{supp}\,a\subset \{1 / 4<|\xi|_g<4 \} $. Let $T>0$ and $A(t): L^2(M;F^{\rho})\to L^2(M;F^{\rho})$ be defined in \eqref{dyanamicsofoperator} for $0\leq t\leq T$. Then, we have
\begin{equation}\label{egorov}
A(t) \in \Psi^{\mathrm{comp},sc}_{h}(M;F^\rho), \quad \sigma_{h}^{\rho}(A(t))=a\circ \varphi_{t}.
\end{equation}
\end{proposition}
To prove Theorem \ref{microlocalcontrolonhilbertbundle}, we need to consider the long time propagation. For $a\in C_{0}^{\infty}(T^*M)$, by the propagation of symbol \eqref{propogationofsymbol}, $a\circ \varphi_{t} \notin S^{\mathrm{comp}}(T^*M)$ when $t=\mu|\log h|$ for $0<\mu<1$. In Dyatlov and Zahl \cite{dyatlovzahl2016spectral}, they introduce the quantization $\Oph^{L}$ of anisotropic symbol class $S^{\mathrm{comp}}_{L,\mu,\mu^{\prime}}(T^*M\setminus \{0\})$. In the following, we will introduce the quantization $\Oph^{\rho,L}(a): L^2(M;F^{\rho})\mapsto L^2(M;F^{\rho})$ for all $a \in S^{\mathrm{comp}}_{L,\mu,\mu^{\prime}}(T^*M\setminus \{0\})$. 

\subsection{Fourier integral operators on flat Hilbert bundles}\label{subsec:fio}

The class $I_h^{\mathrm{comp}}(\kappa)$ consists of compactly supported and microlocalized Fourier integral operators associated with an exact symplectomorphism $\kappa$; see, for instance, Datchev, Dyatlov \cite[\S 3.2]{datchev2013fractal} and Dyatlov, Zahl \cite[\S 2.2]{dyatlovzahl2016spectral} for details. In what follows, we extend the construction of $I_h^{\mathrm{comp}}(\kappa)$ to the setting of flat Hilbert bundles. 

For $j=1,2$, let $M_{j}$ be $d$-dimensional manifolds and let $U_j\subset T^*M_j$ be open sets. Let $\kappa: U_2 \rightarrow U_1 $ be an exact symplectomorphism. By exactness, we mean that there exists a generating function $S(x_1,\xi_2)\in C^\infty (U_/S)$ for some open $U_S\subset \R^{2d}$ such that the graph of $\kappa$ is given by 
\begin{equation*}
    \operatorname{Gr}(\kappa)= \{\xi_{1}=\partial_{x_1} S(x_1, \xi_2), \, x_{2}=\partial_{\xi_2} S(x_1, \xi_2)\ : \  (x_1, \xi_2) \in U_S \},
\end{equation*}
where $(x_j,\xi_j)$ denotes the local coordinate on $U_j$, $j=1,2$. Any $B \in I_h^{\mathrm{comp}}(\kappa)$ has the following form modulo $\mathcal{O}(h^{\infty})_{\mathcal{D}^{\prime}\left(M_2\right) \rightarrow C_0^{\infty}\left(M_1\right)}$:
\begin{equation*}
    B f(x_1)=(2 \pi h)^{-d} \int_{\mathbb{R}^{2d}} e^{\frac{i}{h}(S(x_1, \xi_2)-x_2 \cdot \xi_2)} b(x_1, \xi_2 ; h) \chi(x_2) f(x_2) \,d x_2 d \xi_2, \quad f \in \mathcal{D}^{\prime}(M_2),
\end{equation*}
where $b(x_1, \xi_2 ; h) \in C_{0}^{\infty}(U_S)$ is supported uniformly in $h$, and $\chi \in C_0^{\infty} (M_2 )$ is any function such that $\chi \equiv 1$ near $\partial_{\xi_2} S(\operatorname{supp} b)$.

For $j=1,2$, let $F^{\rho_{j}}$ be a flat $\mathcal{H}$-bundle over $M_{j}$ associated with the unitary representation $\rho_j$, and let $\pi_{\rho_j}$ be the corresponding bundle projection; see Section \ref{subsec: twisted laplacian}. Assume further that $U_{j} \subset T^*\widetilde{U}_{j}$ with a local canonical symplectic coordinate map 
\begin{equation*}
    \widetilde{\varphi}_{j}=(\varphi_{j},(d\varphi_{j})^{-t}): T^*\widetilde{U}_{j} \rightarrow T^* \varphi_j(\widetilde{U}_{j})
\end{equation*}
and a parallel orthonormal local trivialization
\begin{equation*}
    \Phi_{j}:=\Phi_{j}^{\rho_j}: \pi^{-1}_{\rho_{j}}(\widetilde{U}_{j}) \subset F^{\rho_j}\to \widetilde{U}_{j}\times \mathcal{H}.
\end{equation*}
We now define the operator class $I_{h}^{\mathrm{comp}}(\kappa,\rho_{2},\rho_{1})$, which generalizes the class $I_{h}^{\mathrm{comp}}(\kappa)$ and consists of operators acting from  $L^2(M_{2};F^{\rho_{2}}) $ to $L^2(M_{1};F^{\rho_{1}})$.

\begin{definition}\label{defoffioonbundles}
$I_{h}^{\mathrm{comp}}(\kappa,\rho_{2},\rho_{1})$ is the collection of operators $\mathcal{B}$ of the form
\begin{equation*}
    \mathcal{B}=\Phi_{1}^{-1}(B \otimes \phi)\Phi_{2},
\end{equation*}
where $B\in I_{h}^{\mathrm{comp}}( \kappa )$ and $\phi \in \mathrm{U}(\mathcal{H})$. If $\rho_{1}=\rho_{2}=\rho_{\mathrm{triv}}$, we denote $I_{h}^{\mathrm{comp}}(\kappa,\rho_{\mathrm{triv}},\rho_{\mathrm{triv}})$ by $I_{h}^{\mathrm{comp}}(\kappa,\mathcal{H})$, which consists of operators $\mathcal{B} = B \otimes \phi $ for some $B \in I_{h}^{\mathrm{comp}}(\kappa)$ and $\phi \in \mathrm{U}(\mathcal{H})$.
\end{definition}

Note that if we change the parallel orthonormal local trivialization $\Phi_{j}^{\rho_j}$ to $\Phi_j'=(\Phi_{j}^{\prime})^{\rho_j}$, then there exists another unitary operator $\phi^{\prime}\in \mathrm{U}(\mathcal{H})$ such that
$\mathcal{B}=(\Phi_{1}^{\prime})^{-1}(B \otimes \phi^{\prime})\Phi_{2}^{\prime}$. This shows that the operator class $I_{h}^{\mathrm{comp}}(\kappa,\rho_{2},\rho_{1})$ is independent of choice of the parallel orthonormal local trivialization.

We list some basic properties of the class $I^{\mathrm{comp}}_h(\kappa,\rho_2,\rho_1)$, which can be easily derived from the properties of class $I_{h}^{\mathrm{comp}}(\kappa)$ (see \cite[\S 2.2]{dyatlovzahl2016spectral}):
\begin{itemize}
    \item Since $F^{\rho_{j}}$ is a flat Hilbert bundle, for any $\mathcal{B} =\Phi_{1}^{-1}(B \otimes \phi)\Phi_{2}\in I_{h}^{\mathrm{comp}}(\kappa,\rho_{2},\rho_{1})$, we have
    \begin{equation}\label{operatornormoffio}
\| \mathcal{B}\|_{L^2(M_2;F^{\rho_2})\to L^2(M_1;F^{\rho_1})}=\|B\|_{L^2(M_2)\to L^2(M_1)}.
\end{equation}
\item For any two exact symplectomorphisms $\kappa: U_{2}\subset T^*M_2 \to U_{1}\subset T^*M_{1}$ and $\kappa^{\prime}: U_{3}\subset T^*M_3 \to U_{2}\subset T^*M_{2}$, $\kappa \circ \kappa^{\prime}: U_{3} \to U_{1}$ is an exact symplectomorphism and
\begin{equation}
    I_{h}^{\mathrm{comp}}(\kappa,\rho_{2},\rho_{1})\circ I_{h}^{\mathrm{comp}}(\kappa^{\prime},\rho_{3},\rho_{2}) \subset  I_{h}^{\mathrm{comp}}(\kappa\circ \kappa^{\prime},\rho_{3},\rho_{1}).
\end{equation} 
\item For any $\mathcal{B}=\Phi_{1}^{-1}(B\otimes \phi)\Phi_{2} \in I^{\mathrm{comp}}_h(\kappa,\rho_2,\rho_1)$ and $\widetilde{\mathcal{B}}=\Phi_{2}^{-1}(\widetilde{B}\otimes \phi^{-1})\Phi_{1} \in I^{\mathrm{comp}}_h(\kappa^{-1},\rho_1,\rho_2)$, 
\begin{equation}
    \mathcal{B}\widetilde{\mathcal{B}}\in \Psi_{h}^{\mathrm{comp},sc}(M_{1};F^{\rho_{1}}), \quad \widetilde{\mathcal{B}}\mathcal{B}\in \Psi_{h}^{\mathrm{comp},sc}(M_{2};F^{\rho_{2}}), \quad \sigma_{h}^{\rho_{2}}(\widetilde{\mathcal{B}}\mathcal{B})=\sigma^{\rho_{1}}_{h}(\mathcal{B}\widetilde{\mathcal{B}})\circ \kappa.
\end{equation}
\end{itemize}

For any $\mathcal{B}=\Phi_{1}^{-1}(B\otimes \phi)\Phi_{2} \in I_{h}^{\mathrm{comp}}(\kappa,\rho_{2},\rho_{1})$, the $L^2$-adjoint of $\mathcal{B}$ has the following form:
\begin{equation*}
    \mathcal{B}^*=\Phi_{2}^{-1}(B^*\otimes \phi^{-1})\Phi_{1},
\end{equation*}
where $B^* \in I^{\mathrm{comp}}_h(\kappa^{-1})$ is the $L^2$-adjoint of $B$ (see \cite[\S 2.2]{dyatlovzahl2016spectral}). Thus, we have:

\begin{proposition}\label{propofadjoint}
For any $\mathcal{B} \in I^{\mathrm{comp}}_h(\kappa,\rho_{2},\rho_{1})$, the $L^2$-adjoint $\mathcal{B}^*\in  I^{\mathrm{comp}}_h(\kappa^{-1},\rho_{1},\rho_{2})$, and 
 \begin{equation}
 \begin{aligned}
 \mathcal{B}\mathcal{B}^*\in \Psi_{h}^{\mathrm{comp},sc}(M_{1};F^{\rho_1}), \quad \mathcal{B}^*\mathcal{B}\in \Psi_{h}^{\mathrm{comp},sc}(M_{2};F^{\rho_2}), \quad \sigma_{h}^{\rho_{2}}(\mathcal{B}^*\mathcal{B})=\sigma^{\rho_{1}}_{h}(\mathcal{B}\mathcal{B}^*)\circ \kappa.
 \end{aligned}
 \end{equation}
\end{proposition}

Furthermore, we can quantize $\kappa$ on flat Hilbert bundles microlocally. As in \cite[(2.13)]{dyatlovzahl2016spectral}, we define the quantization of $\kappa$ as follows:

\begin{definition}\label{quantizesymponflatbundle}
Let $V\subset U_{2}$ be a compact set. We say that $\mathcal{B}, \,\widetilde{\mathcal{B}}$ quantize $\kappa$ near $\kappa(V) \times V$ if there exist $\mathcal{B}\in I^{\mathrm{comp}}_h(\kappa,\rho_{2},\rho_{1})$ and $\widetilde{\mathcal{B}}\in I^{\mathrm{comp}}_h(\kappa,\rho_{2},\rho_{1})$ such that
\begin{equation}
\begin{array}{ll}
\widetilde{\mathcal{B}}\mathcal{B}=I+\mathcal{O} (h^{\infty} )_{\Psi^{-\infty,sc}(M_2)}& \text { microlocally near } V,\\
\mathcal{B}\widetilde{\mathcal{B}}=I+\mathcal{O} (h^{\infty} )_{\Psi^{-\infty,sc}(M_1)} & \text { microlocally near } \kappa(V).
\end{array}
\end{equation}
\end{definition}

By the discussion below \cite[(2.13)]{dyatlovzahl2016spectral}, for any point $(x_2,\xi_2)\in U_2$, the operators $\mathcal{B},\,\widetilde{\mathcal{B}}$ that quantize $\kappa$ near $(\kappa(x_2,\xi_2),x_2,\xi_2)$ exist.

Before discussing the quantization of anisotropic symbol class $S^{\mathrm{comp}}_{L,\mu,\mu^{\prime}}(U)$ for $U\subset T^*M$ (see Definition \ref{anisosymbolclass}), we consider the transformation of $\Oph(a)\otimes \mathrm{Id}_{\mathcal{H}}$ under the action of Fourier integral operators for $a \in S^{\mathrm{comp}}_{L_0,\mu,\mu^{\prime}}(T^*\R^{d})$. Here $\Oph(a)$ is still defined as the standard quantization \eqref{standquant}. By \cite[Appendix A.2]{dyatlov2018semiclassical}, $\Oph(a): L^2(\R^d)\mapsto L^2(\R^d)$ is uniformly bounded in $h$. By \cite[Lemma 3.10]{dyatlovzahl2016spectral}, we have:

\begin{proposition}\label{conjugateoffio}
Assume that $\kappa: U_2 \rightarrow U_1$, where $U_1,\,U_2\subset T^* \mathbb{R}^d$ are open, is an exact symplectomorphism such that $\kappa_*\left(L_0\right)=L_0$. Let $B \in I^{\mathrm{comp}}_h(\kappa)$, $B^{\prime} \in I^{\mathrm{comp}}_h (\kappa^{-1} )$, and $\phi\in \mathrm{U}(\mathcal{H})$. Then for any $a \in S_{L_0, \mu, \mu^{\prime}}^{\mathrm{comp}} (T^* \mathbb{R}^d)$, there exists $b \in S_{L_0, \mu, \mu^{\prime}}^{\mathrm{comp}} (T^* \mathbb{R}^d)$ such that
\begin{equation}
\begin{aligned}
&(B^{\prime}\otimes \phi^{-1}) (\mathrm{Op}_h(a)\otimes \mathrm{Id}_{\mathcal{H}})(B\otimes \phi)  =\mathrm{Op}_h(b)\otimes \mathrm{Id}_{\mathcal{H}}+\mathcal{O} (h^{\infty} )_{L^2(M;F^\rho) \rightarrow L^2(M;F^\rho)} ,\\
&b =(a \circ \kappa) \sigma_h (B^{\prime} B )+\mathcal{O} (h^{1-\mu-\mu^{\prime}} )_{S_{L_0, \mu, \mu^{\prime}}^{\mathrm{comp}} (T^* \mathbb{R}^d )} ,\\
&\operatorname{supp} b  \subset \kappa^{-1}(\operatorname{supp} a).
\end{aligned}
\end{equation}
\end{proposition}

\subsection{Anisotropic calculi and long-time propagation}\label{subsec: aniso}

In this section, we define a general (non-canonical) quantization procedure $\mathrm{Op}_h^{\rho, L}(a): L^2(M;F^{\rho})\to L^2(M;F^{\rho})$ for any $a=a(x,\xi;h)\in S^{\mathrm{comp}}_{L,\mu,\mu^{\prime}}(T^*M\setminus \{0\})$, following the framework of \cite[\S 3.3]{dyatlovzahl2016spectral} and \cite[Appendix A.4]{dyatlov2018semiclassical}. 

In what follows, we fix an open subset $U'\subset T^*M$ and always assume that $L$ is a Lagrangian foliation on $U'$.

%and $U$ is an open subset of $U^{\prime}$ with compact closure $\overline{U}$. 

\begin{definition}
Let $U$ be an open subset of $U'$ such that $U\subset T^*\widetilde{U}$ for an open set $\widetilde{U} \subset M$. We call $(\kappa,U,\Phi,\widetilde{U},\varphi,\mathcal{B},\widetilde{\mathcal{B}})$ a \textit{Lagrangian normal chart} for $(L,\rho)$ if 
\begin{itemize}
    \item $(\kappa,U,\kappa(U))$ is a Lagrangian chart for $L$;
    \item $(\Phi,\widetilde{U},\varphi)$ is a parallel orthonormal local trivialization chart of $F^\rho$;
    \item $\mathcal{B}\in I^{\mathrm{comp}}_h(\kappa,\rho,\rho_{\mathrm{triv}})$ and $\widetilde{\mathcal{B}}\in I^{\mathrm{comp}}_h(\kappa^{-1},\rho_{\mathrm{triv}},\rho)$.
\end{itemize}
\end{definition}

By Lemma \ref{locallagrangian}, Proposition \ref{propofadjoint} and the paragraph following \cite[(2.12)]{dyatlovzahl2016spectral}, for each $(x_0, \xi_0) \in U'$, there exists a Lagrangian normal chart $(\kappa, U, \Phi,\widetilde{U},\varphi,\mathcal{B}, \mathcal{B}^*)$ such that $\sigma_h^{\rho}(\mathcal{B}^* \mathcal{B})(x_0, \xi_0)>0$.

\begin{definition}\label{defofquantizationofaniso}
For any open set $U\subset U'$, let 
\begin{equation*}
    \big\{ (\kappa_{\alpha}, U_{\alpha}. \Phi_{\alpha}, \widetilde{U}_{\alpha}, \varphi_{\alpha},\mathcal{B}_{\alpha},\mathcal{B}_{\alpha}^*) \big\}_{\alpha\in \Lambda}
\end{equation*}
be a collection of Lagrangian normal charts for $(L,\rho)$ such that $\{U_\alpha \}_{\alpha \in \Lambda}$ is a locally finite cover of $U$ and
\begin{equation*}
    \sum\limits_{\alpha \in \Lambda}\sigma^{\rho}_{h} (\mathcal{B}_{\alpha}^{*} \mathcal{B}_{\alpha} ) \equiv 1 \qquad \text{on} \ U.
\end{equation*}
For each $\alpha \in \Lambda$, let $\chi_\alpha\in C_{0}^{\infty}(U_{\alpha})$ be such that $\chi_{\alpha}\equiv 1$ near $\mathrm{supp}\,\sigma^\rho_h(\mathcal{B}_{\alpha}^{*} \mathcal{B}_{\alpha})$. For any $a\in S^{\mathrm{comp}}_{L,\mu,\mu^{\prime}}(U)$, we define $\Oph^{\rho,L}(a): L^{2}(M;F^{\rho}) \to L^{2}(M;F^{\rho})$ as: 
\begin{equation}
\Oph^{\rho,L}(a):=\sum_{\alpha\in \Lambda} \mathcal{B}_{\alpha}^*  \big(\mathrm{Op}_h(a_\alpha) \otimes \mathrm{Id}_{\mathcal{H}} \big)\mathcal{B}_{\alpha},
\end{equation}
where $a_\alpha = (\chi_{\alpha} a) \circ \kappa_{\alpha}^{-1} \in S_{L_0, \mu, \mu' }^{\mathrm{comp}}(T^* \mathbb{R}^d) $.
\end{definition}

%\begin{remark}
%The definition also holds for general open subset $U$, we replace the finite cover by locally finite Lagrangian normal charts. In our application, we always assume that $\overline{U}\subset U^{\prime}\subset T^*M$ is compact subset.
%\end{remark}

The following properties of $\mathrm{Op}_h^{\rho,L}$ are uniform versions of properties (1)–(7) in \cite[Appendix A.4]{dyatlov2018semiclassical}; see also \cite[Lemma 3.12 \& Lemma 3.14]{dyatlovzahl2016spectral}:

\begin{proposition}\label{propofanisoquantization}
Let $U$ be an open subset of $U'$. In the following, the constants in $\mathcal{O}$ are independent of $(\mathcal{H},\rho)$: 
\begin{itemize}
\item[(1)] for each $a \in S_{L, \mu, \mu^{\prime}}^{\mathrm{comp}}(U)$, $\mathrm{Op}_h^{\rho, L}(a): L^2(M;F^{\rho}) \rightarrow L^2(M;F^{\rho})$ is uniformly bounded in $(\mathcal{H},\rho,h)$;
\item[(2)] if $a=a(x,\xi)\in C_{0}^{\infty}(U)$ is independent of $h$, then $\Oph^{\rho,L}(a) \in \Psi_{h}^{\mathrm{comp},sc}(M;F^\rho)$;
\item[(3)] for each $a,b \in S_{L, \mu, \mu^{\prime}}^{\mathrm{comp}}(U)$, there exists $a \#_L b \in S_{L, \mu, \mu^{\prime}}^{\mathrm{comp}}(U)$ such that
\begin{equation}
\begin{aligned}
&\mathrm{Op}^{\rho,L}_h(a) \mathrm{Op}^{\rho,L}_h(b) =\mathrm{Op}^{\rho,L}_h(a \#_L b)+\mathcal{O}(h^{\infty})_{L^2(M;F^\rho) \rightarrow L^2(M;F^\rho)} ,\\
&a \#_L b  =a b+\mathcal{O}(h^{1-\mu-\mu^{\prime}})_{S_{L, \mu, \mu^{\prime}}^{\mathrm{comp}}(U)}, \\
&\operatorname{supp}\, (a \#_L b) \subset \operatorname{supp} a \cap \operatorname{supp} b;
\end{aligned}
\end{equation}
\item[(4)] for each $a \in S_{L, \mu, \mu^{\prime}}^{\mathrm{comp}}(U)$ there exists $a_L^* \in S_{L, \mu, \mu^{\prime}}^{\mathrm{comp}}(U) $ such that
\begin{equation}
\begin{aligned}
&\mathrm{Op}^{\rho,L}_h(a)^*  =\mathrm{Op}^{\rho,L}_h(a_L^*)+\mathcal{O}(h^{\infty})_{L^2(M;F^\rho) \rightarrow L^2(M;F^\rho)}, \\
&a_L^*  =\overline{a}+\mathcal{O}(h^{1-\mu-\mu^{\prime}})_{S_{L, \mu, \mu^{\prime}}^{\operatorname{comp}}(U)}, \\
&\operatorname{supp} a_L^* \subset \operatorname{supp} a.
\end{aligned}
\end{equation}
\end{itemize}
\end{proposition}

Furthermore, by Definition \ref{defofquantizationofaniso} and \cite[Lemma A.2]{dyatlov2018semiclassical}, we also have the following uniform version of sharp Gårding inequality for $\Oph^{\rho,L}$:

\begin{proposition}\label{sharpgardingforanisoquant}
For any $a\in S^{\mathrm{comp}}_{L,\mu,\mu^{\prime}}(U)$ with $a\geq 0$, there exists a constant $C>0$ independent of $(\mathcal{H},\rho,h)$ such that  
\begin{equation*}
    \langle \Oph^{\rho,L}(a)u,u\rangle_{L^2(M;F^\rho)} \geq-Ch^{1-\mu-\mu^{\prime}}\|u\|_{L^2(M;F^\rho)}^2.
\end{equation*}
\end{proposition}

As in \eqref{sharpnorm}, a direct consequence of the sharp Gårding inequality is the sharp operator bound:
\begin{equation}\label{sharpnormaniso}
     \|\Oph^{\rho,L}(a) \|_{L^2(M;F^\rho) \to L^2(M;F^\rho) }\leq \sup\limits_{(x,\xi)\in T^*U}|a|+\mathcal{O}(h^{1-\mu-\mu^{\prime}}).
\end{equation}

The following proposition generalizes Proposition \ref{conjugateoffio} to the setting of flat Hilbert bundles:

\begin{proposition}\label{mappingoffio}
Assume that $M_1, M_2$ are manifolds of the same dimension, $(\Phi_j,\widetilde{U}_j,\chi_j)$ are parallel orthonormal local trivialization charts of $\pi_{\rho_j}:F^{\rho_j} \rightarrow M_j$, $U_j^{\prime} \subset T^* M_j$ are open sets, $L_j$ are Lagrangian foliations on $U_j^{\prime}$, and $U_{j}\subset U_{j}^{\prime}$ are open subsets satisfying $U_{j}\subset T^*\widetilde{U}_j$. Let $\kappa: U_2 \rightarrow U_1$ be an exact symplectomorphism mapping $L_2$ to $L_1$. Then for any open subset $U_{1,0} \subset U_{1}^{\prime}$ with compact closure in $U_{1}^{\prime}$, $a_{1} \in S^{\mathrm{comp}}_{L_1,\mu,\mu^{\prime}}(U_{1,0})$, $\mathcal{B}\in I^{\mathrm{comp}}_{h}(\kappa,\rho_2,\rho_{1})$ and $\widetilde{\mathcal{B}} \in I^{\mathrm{comp}}_{h}(\kappa^{-1},\rho_{1},\rho_{2})$ such that $\widetilde{\mathcal{B}}\mathcal{B}\in \Psi^{\mathrm{comp}}_{h}(U_2)\otimes \mathrm{Id}_{\mathcal{H}}$, there exists $a_2 \in S_{L_2, \mu, \mu^{\prime}}^{\mathrm{comp}} \big(\kappa^{-1}(U_{1,0}\cap U_1) \big)$ such that
\begin{equation}
\begin{aligned}
&\widetilde{\mathcal{B}}\,\mathrm{Op}_h^{\rho_1,L_1}(a_1)\,\mathcal{B}  =\mathrm{Op}_h^{\rho_2,L_2}(a_2)+\mathcal{O}(h^{\infty})_{L^2(M_2;F^{\rho_2}) \rightarrow L^2(M_2;F^{\rho_2})} ,\\
&a_2 =(a_1 \circ \kappa ) \sigma_h^{\rho_2}(\widetilde{\mathcal{B}}\mathcal{B})+\mathcal{O}(h^{1-\mu})_{S_{L_0, \mu, \mu^{\prime}}^{\mathrm{comp}}(U_2)} ,\\
&\operatorname{supp} a_2  \subset \kappa^{-1}(\operatorname{supp} a_1).
\end{aligned}
\end{equation}
\end{proposition}

\begin{proposition}\label{multiplicationofoperators}
Let $a_1, \cdots, a_N \in S_{L, \mu, \mu^{\prime}}^{\mathrm{comp}}(U)$ be as in Lemma \ref{multiofsymbol} and assume that $A_1, \cdots, A_N$ are operators on $L^2(M;F^{\rho})$ such that 
\begin{equation*}
    A_j=\mathrm{Op}_h^{\rho,L}(a_j)+\mathcal{O} \big(h^{(1-\mu-\mu^{\prime})-} \big)_{L^2(M;F^\rho) \rightarrow L^2(M;F^\rho)} ,
\end{equation*}
where the constants in $\mathcal{O}(\bullet)$ are independent of $j$ and $(\mathcal{H},\rho)$. Then,
\begin{equation*}
    A_1 \cdots A_N=\mathrm{Op}_h^{\rho,L} (a_1 \cdots a_N )+\mathcal{O} \big(h^{(1-\mu-\mu^{\prime})-} \big)_{L^2(M;F^\rho) \rightarrow L^2(M;F^\rho)}.
\end{equation*}
\end{proposition}

\begin{proof}
We have
\begin{equation}
\begin{gathered}
A_1 \cdots A_N-\mathrm{Op}_h^{\rho,L} (a_1 \cdots a_N )=\sum_{j=1}^N B_j A_{j+1} \cdots A_N, \quad \text{where} \\
B_j:= \begin{cases}A_1-\mathrm{Op}_h^{\rho,L}\left(a_1\right), & j=1 ; \\
\mathrm{Op}_h^{\rho,L} (a_1 \cdots a_{j-1} ) A_j-\mathrm{Op}_h^{\rho,L} (a_1 \cdots a_j ), & 2 \leq j \leq N .\end{cases}
\end{gathered}
\end{equation}
Here $\mathrm{Op}_h^{\rho,L}\left(a_1 \cdots a_{j-1}\right)$ is well-defined by Lemma \ref{multiofsymbol}, $a_1 \cdots a_{j-1} \in S_{L, \mu+\varepsilon, \mu^{\prime}+\varepsilon}^{\mathrm{comp}}(U)$ uniformly in $j$ for any small $\varepsilon>0$.

Since $\sup  |a_j | \leq 1$, by \eqref{sharpnormaniso}, we have, for some $C$ independent of $(\mathcal{H},\rho,h,j)$,
\begin{equation*}
 \|A_j \|_{L^2(M;F^\rho) \rightarrow L^2(M;F^\rho)} \leq 1+C h^{1-\mu-\mu^{\prime}}   .
\end{equation*}
Since $N=\mathcal{O}(|\log h|)$, we have, uniformly in $j$,
\begin{equation*}
     \|A_{j+1} \cdots A_N \|_{L^2(M;F^\rho) \rightarrow L^2(M;F^\rho)} \leq C .
\end{equation*}
By properties (1) and (3) in Proposition \ref{propofanisoquantization}, we conclude
\begin{equation*}
    \|B_j \|_{L^2(M;F^\rho) \rightarrow L^2(M;F^\rho)}=\mathcal{O} \big(h^{(1-\mu-\mu^{\prime})-} \big)_{L^2(M;F^\rho) \rightarrow L^2(M;F^\rho)} \quad \text{uniformly in} \ j.
\end{equation*}
The proof is complete.
\end{proof}
Let $P \in \Psi_h^{\mathrm{comp},sc}(M;F^\rho)$ be self-adjoint with principal symbol $p=\sigma_h(P) \in C_0^{\infty}\left(T^* M ; \mathbb{R}\right)$. We assume that
\begin{equation}\label{conditionoffoliation}
L_{(x, \xi)} \subset \operatorname{ker} d p(x, \xi) \text { for all }(x, \xi) \in U^{\prime} ;
\end{equation}
this is equivalent to the Hamiltonian vector field $H_p$ lying inside $L$. Let $\varphi_{t}=\exp(tH_{p})$ be the Hamiltonian flow generated by $H_{p}$. The operator $e^{-i t P / h}: L^2(M;F^\rho) \rightarrow L^2(M;F^{\rho})$ is unitary. We start with the following fixed-time statement similar to \cite[Lemma 3.17]{dyatlovzahl2016spectral}:

\begin{proposition}\label{fixtimeegorov}
Let $a \in S_{L, \mu, \mu^{\prime}}^{\mathrm{comp}}(U)$, $A=\Oph^{\rho,L}(a)$ and fix an $h$-independent constant $T \geq 0$ such that $\varphi_{-t}(\operatorname{supp} a) \subset U^{\prime}$ for all $t \in[0, T]$. Then
$$
A(t)=\Oph^{\rho,L}(a \circ \varphi_{t})+\mathcal{O}(h^{1-\mu-\mu^{\prime}})_{L^2(M;F^\rho) \rightarrow L^2(M;F^\rho)} \qquad \mathrm{for\ } 0 \leq t \leq T.
$$
\end{proposition}

\begin{proof}
We first claim that for each $b \in S_{L, \mu, \mu^{\prime}}^{\mathrm{comp}}(U)$,
\begin{equation}\label{commutatorestimate}
\left[P, \Oph^{\rho,L}(b)\right]=-i h \Oph^{\rho,L} (H_p b)+\mathcal{O} \big(h^{2-\mu-\mu^{\prime}}\big)_{L^2(M;F^\rho) \rightarrow L^2(M;F^\rho)} .
\end{equation}
Using the microlocal partition of unity on $\overline{U}$ (see Proposition \ref{microlocalpartitionofunity} and \eqref{pseudoequalsquant}), we take a finite family of $X_{\alpha}=\Oph^{\rho}(\chi_{\alpha})+\mathcal{O}(h^{\infty})_{\Psi^{-\infty,sc}_h(M;F^\rho)}$ such that $\mathrm{\chi_{\alpha}}\in C_{0}^{\infty}(U_\alpha)$ and
\begin{equation*}
    \sum_{\alpha}X_{\alpha}=I+\mathcal{O}(h^{\infty})_{\Psi^{-\infty,sc}_h(M;F^\rho)}, \quad \text{microlocally near} \ U,
\end{equation*}
where $(\kappa_{\alpha},U_{\alpha},\kappa_{\alpha}(U_{\alpha}))$ are Lagrangian charts for $L$. Note that there exist $\mathcal{B}_{\alpha} \in I^{\mathrm{comp}}_h(\kappa_{\alpha},\rho,\rho_{\mathrm{triv}})$ and $\widetilde{\mathcal{B}}_{\alpha} \in I^{\mathrm{comp}}_h(\kappa_{\alpha}^{-1},\rho_{\mathrm{triv}},\rho)$ quantize $\kappa_{\alpha}(U_{\alpha})\times U_{\alpha}$; see Definition \ref{quantizesymponflatbundle}. Thus, there exists $b_{\alpha}\in S^{\mathrm{comp}}_{L,\mu,\mu^{\prime}}(U_{\alpha}\cap U)$ such that 
\begin{equation*}
    X_{\alpha}\Oph^{\rho,L}(b)=\Oph^{\rho,L}(b_{\alpha})+\mathcal{O}(h^{\infty})_{L^2(M;F^\rho) \rightarrow L^2(M;F^\rho)}.
\end{equation*} 
Thus, we only need to prove \eqref{commutatorestimate} for $b_{\alpha}$. Since $B_{\alpha}$ and $\widetilde{B}_{\alpha}$ quantize $\kappa(U_{\alpha})\times U_{\alpha}$ and $\mathrm{supp}\,b_{\alpha}\subset U_{\alpha}$, we have
\begin{equation}
\begin{aligned}
&{\left[P, \mathrm{Op}_h^{\rho,L}(b_{\alpha})\right] } \\
& =\widetilde{\mathcal{B}}_{\alpha} \mathcal{B}_{\alpha}\Big(P \widetilde{\mathcal{B}}_{\alpha} \mathcal{B}_{\alpha} \Oph^{\rho,L}(b_{\alpha})-\Oph^{\rho,L}(b_{\alpha}) \widetilde{\mathcal{B}}_{\alpha} \mathcal{B}_{\alpha} P\Big) \widetilde{\mathcal{B}}_{\alpha} \mathcal{B}_{\alpha}+\mathcal{O}(h^{\infty})_{L^2(M;F^\rho) \rightarrow L^2(M;F^\rho)} \\
& =\widetilde{\mathcal{B}}_{\alpha}\Big[\mathcal{B}_{\alpha} P \widetilde{\mathcal{B}}_{\alpha} ,\, \mathcal{B}_{\alpha} \Oph^{\rho,L}(b_{\alpha}) \widetilde{\mathcal{B}}_{\alpha}\Big] \mathcal{B}_{\alpha}+\mathcal{O}(h^{\infty})_{L^2(M;F^\rho) \rightarrow L^2(M;F^\rho)}.
\end{aligned}
\end{equation}
In the following, we omit the index $\alpha$. By Proposition \ref{mappingoffio}, we have
\begin{equation}
\begin{aligned}
\mathcal{B} \mathrm{Op}_h^{\rho,L}(b) \widetilde{\mathcal{B}}& =\mathrm{Op}_h(\widetilde{b})\otimes \mathrm{Id}_{\mathcal{H}}+\mathcal{O}(h^{\infty})_{L^2(M;F^\rho) \rightarrow L^2(M;F^\rho)} \quad \text { for some } \widetilde{b} \in S_{L_0, \mu, \mu^{\prime}}^{\mathrm{comp}}(T^* \mathbb{R}^d) ,\\
\widetilde{b}&=b \circ \kappa^{-1}+\mathcal{O}\big(h^{1-\mu-\mu^{\prime}} \big)_{S_{L_0, \mu, \mu^{\prime}}^{\mathrm{comp}}(T^* \mathbb{R}^d)}, \qquad \operatorname{supp} \widetilde{b} \subset \kappa(\operatorname{supp} b).
\end{aligned}
\end{equation}
Next, $\mathcal{B} P \widetilde{\mathcal{B}} \in \Psi^{\mathrm{comp}}_h(\mathbb{R}^d)\otimes \mathrm{Id}_{\mathcal{H}}$, $\sigma_h(\mathcal{B} P \widetilde{\mathcal{B}})=(p \circ \kappa^{-1}) \sigma_h(\mathcal{B} \widetilde{\mathcal{B}})$ is equal to $p \circ \kappa^{-1}$ near $\operatorname{supp} \widetilde{b}$. By \eqref{conditionoffoliation}, we then have $\partial_\eta \sigma_h(\mathcal{B} P \widetilde{\mathcal{B}})=0$ near $\operatorname{supp} \widetilde{b}$. By \cite[(A.13)]{dyatlov2018semiclassical}, we have
\begin{equation}\label{commuatorformula}
\begin{aligned}
{\left[P, \Oph^{\rho,L}(b)\right] } & =\widetilde{\mathcal{B}}\left[\mathcal{B} P \widetilde{\mathcal{B}}, \mathrm{Op}_h(\widetilde{b})\right] \mathcal{B}+\mathcal{O}(h^{\infty})_{L^2(M;F^\rho) \rightarrow L^2(M;F^\rho)} \\
& =-i h \widetilde{\mathcal{B}} \,\mathrm{Op}_h \big(\{p \circ \kappa^{-1}, b \circ \kappa^{-1}\} \big) \mathcal{B}+\mathcal{O} \big(h^{2-\mu-\mu^{\prime}}\big)_{L^2(M;F^\rho) \rightarrow L^2(M;F^\rho)}.
\end{aligned}
\end{equation}
Since
\begin{equation*}
     \{p \circ \kappa^{-1}, b \circ \kappa^{-1}\}=(H_p b) \circ \kappa^{-1} \in h^{-\mu^{\prime}} S_{L_0, \mu, \mu^{\prime}}^{\mathrm{comp}}(T^* \mathbb{R}^d) ,
\end{equation*}
the second line of \eqref{commuatorformula} is equal to 
\begin{equation*}
    -i h \mathrm{Op}_h^{\rho,L}\left(H_p b\right)+\mathcal{O} \big(h^{2-\mu-\mu^{\prime}}\big)_{L^2(M;F^\rho) \rightarrow L^2(M;F^\rho)}.
\end{equation*}
This completes the proof of \eqref{commutatorestimate}.

Now, we put $a_t:=a \circ \varphi_{t}$, $t \in[0, T]$. By \eqref{conditionoffoliation}, the map $e^{t H_p}$ preserves the foliation $L$ on $\operatorname{supp} a$, so $a_t \in S_{L, \mu, \mu^{\prime}}^{\mathrm{comp}}(U)$. Since $\partial_t a_t=H_p a_t$, by \eqref{commutatorestimate}, we have
\begin{equation}
\begin{aligned}
i h \partial_t\left(e^{-i t P / h} \Oph^{\rho,L}\left(a_t\right) e^{i t P / h}\right) & =e^{-i t P / h}\left(i h \Oph^{\rho,L}\left(\partial_t a_t\right)+\left[P, \Oph^{\rho,L}\left(a_t\right)\right]\right) e^{i t P / h} \\
& =\mathcal{O} \big(h^{2-\mu-\mu^{\prime}}\big)_{L^2(M;F^\rho) \rightarrow L^2(M;F^\rho)}, \quad 0 \leq t \leq T.
\end{aligned}
\end{equation}
Integrating this from $0$ to $t$, we conclude the proof.
\end{proof}

We now restrict ourselves to the case when $M$ is a compact hyperbolic surface, $U^{\prime}=T^* M \setminus \{0\}$, $L \in\left\{L_u, L_s\right\}$ with $L_u, L_s$ defined in \eqref{stablefoliation}, and $U=\left\{(x,\xi): 1/4<|\xi|_g<4\right\}$. Let $\varphi_t$ be the homogeneous geodesic flow, $P \in \Psi^{\mathrm{comp},sc}_h(M;F^\rho)$ be defined in \eqref{defofcutofflaplacian}, and $U(t)=e^{-i t P / h}$ as in \eqref{schordingerpropagator}. The following Egorov's theorem is a long-time version of Proposition \ref{uniformegorovtheorem} for times up to $\mu|\log h|$.

\begin{theorem}\label{longtimeegorovtheorem}
Assume that $a=a(x,\xi) \in C_0^{\infty}\left(\left\{1 / 4<|\xi|_g<4\right\}\right)$ is $h$-independent. Then we have, uniformly in $t \in[0, \mu |\log h|]$,
\begin{equation}\label{longtimeweakegorovtheorem}
\begin{aligned}
& A(t)=\mathrm{Op}_h^{\rho,L_s} (a \circ \varphi_t )+\mathcal{O} (h^{1-\mu} |\log h| )_{L^2(M;F^\rho) \rightarrow L^2(M;F^\rho)} ,\\
& A(-t)=\mathrm{Op}_h^{\rho,L_u} (a \circ \varphi_{-t} )+\mathcal{O} (h^{1-\mu} |\log h| )_{L^2(M;F^\rho) \rightarrow L^2(M;F^\rho)}  .
\end{aligned}
\end{equation}
\end{theorem}

\begin{proof}
We prove first line of \eqref{longtimeweakegorovtheorem}, with second line proved similarly (replacing $P$ by $-P$). By property (2) in Proposition \ref{propofanisoquantization}, we may replace $\mathrm{Op}_h^{\rho}(a)$ by $\mathrm{Op}_h^{\rho,L_{s}}(a)$ with an $\mathcal{O}(h)_{L^2(M;F^\rho) \rightarrow L^2(M;F^\rho)}$ error.

We write $t=N s$ where $0 \leq s \leq 2$ and $N \in \mathbb{N}, N \leq |\log h|$. Then
\begin{equation}
\begin{aligned}
&A(t)-\mathrm{Op}_h^{\rho,L_s}\left(a \circ \varphi_t\right) \\
&=\sum_{j=0}^{N-1} U(-j s)\left( U(-s) \mathrm{Op}_h^{\rho,L_s} \big(a \circ \varphi_{(N-1-j) s} \big) U(s)-\mathrm{Op}_h^{\rho,L_s}  \big(a \circ \varphi_{(N-j) s}\big)  \right) U(j s) .
\end{aligned}
\end{equation}
Since $U(j s)$ is unitary, it suffices to prove that uniformly in $j=0, \cdots, N-1$,
\begin{equation}\label{uniformupperbound}
U(-s) \mathrm{Op}_h^{\rho,L_s} \big(a \circ \varphi_{(N-1-j) s} \big) U(s)-\mathrm{Op}_h^{\rho,L_s}  \big(a \circ \varphi_{(N-j) s}\big)=\mathcal{O}(h^{1-\mu})_{L^2(M;F^\rho) \rightarrow L^2(M;F^\rho)}.
\end{equation}
Now \eqref{uniformupperbound} follows from Proposition \ref{fixtimeegorov} applied to $a \circ \varphi_{(N-1-j) s} \in S_{L_{s}, \mu,0}^{\mathrm{comp}}(T^* M \setminus \{0\})$. Here $\varphi_t=\exp \left(t H_{\sigma_h(P)}\right)$ on $\{1 / 4<|\xi|_g<4\}$ with $\sigma_h(P)=|\xi|_g$ on $U$.
\end{proof}

\section{Uniform semiclassical control on flat Hilbert bundles}\label{sec:uniform semiclassical control}

In this section, we prove the uniform semiclassical control estimate:

\begin{theorem}\label{generalizedmicrolocalcontrolonhilbertbundle}
Let $M$ be a compact hyperbolic surface and $a \in C_{0}^{\infty}(T^* M)$ such that $ a |_{S^* M} \not \equiv 0$. There exist constants $C=C(M,a)>0$ and $h_{0}  = h_0(M,a)>0$ such that for any $(\mathcal{H},\rho)\in \mathcal{C}$, any $0<h<h_{0}$, any self-adjoint $\mathcal{P}=\mathcal{P}(h)\in \Psi^{2,sc}_{h}(M;F^\rho)$ with $\sigma_{h}^{\rho}(\mathcal{P})=|\xi|_g^2$ and any $u \in H^2(M;F^{\rho})$,
\begin{equation}\label{generalizedmicrolocalcontrolonhilbertbundleq}
\|u\|_{L^2(M;F^{\rho})} \leq C\left( \|\Oph^{\rho}(a) u\|_{L^2(M;F^{\rho})}+\frac{|\log h|}{h}\left\|(\mathcal{P} -I\right) u\|_{L^2(M;F^{\rho})} \right).
\end{equation}
\end{theorem}

Our proof follows the argument of Dyatlov and Jin \cite[Theorem 2]{dyatlov2018semiclassical}. We emphasize that their proof involves only the $L^2$-norm estimate in semiclassical microlocal analysis. In section \ref{sec:quantization on flat bundles}, we establish the semiclassical calculus for scalar symbols with operator-norm bounds uniform over all flat $\mathcal{H}$-bundles. Here we only list the main steps and highlight that the constants are independent of $(\mathcal{H},\rho)$. The strategy of the proof is as follows.

We decompose the phase space $T^* M$ into several distinct regions:
\begin{itemize}
\item Outside of $S^* M$, elliptic estimates imply that $u$ can be controlled by $(\mathcal{P}-1) u$. 
\item On $S^*M$, we further divide into two parts: the region controlled by $a$, and the region not controlled by $a$:
\begin{itemize}
     \item the part controlled by $a$ is precisely the union of all geodesic flow orbits passing through $\{a \neq 0\}$. In this region, by propagation estimates and Egorov's theorem, $u$ can be controlled by $\mathrm{Op}_h^{\rho}(a) u$ and $(\mathcal{P}-1)u$;
\item the uncontrolled part is given by
\begin{equation*}
    K:=S^* M \setminus \bigcup_{t \in \mathbb{R}} \varphi_t(\{a \neq 0\}).
\end{equation*}
On $K$, which is ``fractal'' in both stable and unstable directions, we apply the fractal uncertainty principle, see \cite{dyatlov2018semiclassical}, to obtain control.
\end{itemize}
\end{itemize}

\subsection{Partitions and key estimates}
We take $U^{\prime}=T^*M\setminus \{0\}$ and $U=\left\{1/4<|\xi|_g<4\right\}$ in this section. Fix conical open sets
\begin{equation*}
    \mathcal{U}_1, \,\mathcal{U}_2 \subset T^* M \setminus \{0\}, \quad \mathcal{U}_1, \,\mathcal{U}_2 \neq \varnothing, \quad \overline{\mathcal{U}_1} \cap \overline{\mathcal{U}_2}=\varnothing, \quad \overline{\mathcal{U}_2} \cap S^* M \subset \{a \not= 0\}.
\end{equation*}

We start with a rough microlocal partition of unity $I=A_0+A_1+A_2$ constructed as follows. 
\begin{itemize}
    \item Let $\psi_0 \in C^{\infty}(\mathbb{R};[0,1])$ satisfy
\begin{equation}\label{cutofffunctioninproof}
\operatorname{supp} \psi_0 \cap [1/4,4]=\varnothing, \quad \operatorname{supp}\left(1-\psi_0\right) \subset(1 / 16,16).
\end{equation}
Then,
\begin{equation}
    A_{0}:=\psi_{0}(\mathcal{P}) \in \Psi^{0,sc}_h(M;F^\rho)
\end{equation}
is microlocalized in the region away from $S^* M$. Since $1-\psi_0\in C_{0}^{\infty}(\R)$, by Theorem \ref{mainthmofquantized}, $I-A_0\in \Psi^{\mathrm{comp},sc}_h(M;F^\rho)$ and there exists $b \in C_{0}^{\infty}(U)$ such that
\begin{equation}
I-A_0=\Oph^{\rho}(b)+\mathcal{O}(h^\infty)_{\Psi^{-\infty,sc}(M;F^\rho)},
\end{equation}
where $b\equiv 1-\psi_{0}(|\xi|^2_g)$ mod $hS^{-1}(T^*M)$.
\item Let 
\begin{equation}
    \Omega_1:=(U \setminus \overline{\mathcal{U}_1}) \cap \big(\{a \not= 0\} \cup (T^* M \setminus S^* M ) \big).
\end{equation}
Roughly speaking, $\Omega_{1}$ represents the region ``controlled'' by $a$. Let $\Omega_2:=U \setminus \overline{\mathcal{U}_2}$. It is clear that $\mathrm{WF}_h(I-A_0)\subset U \subset\Omega_1 \cup \Omega_2$. For $i=1,2$, choosing $\chi_{i} \in C_{0}^{\infty}(\Omega_i;[0,1])$ such that $\chi_{1}+\chi_{2} \equiv 1$ near $\mathrm{supp}\,b$, we define
\begin{equation}
\begin{aligned}
    &A_{1}:=\Oph^{\rho}(\chi_{1}b) \in \Psi^{\mathrm{comp},sc}_h(M;F^\rho), \\
    &A_{2}:=I-A_{0}-A_1=\Oph^{\rho}(\chi_{2}b)+\mathcal{O}(h^\infty)_{\Psi^{-\infty,sc}(M;F^\rho)}.
\end{aligned}
\end{equation}
%It is clear that $\mathrm{WF}_h(A_i)\subset \Omega_{i}$. Let $a_{i}(x,\xi)=\chi_{i}\big(1-\psi_0(|\xi|^2_g)\big)=\chi_i b+\mathcal{O}(h)_{S_{1,0}^{\mathrm{comp}}(T^*M)}$ which are independent of $h$. 
\end{itemize}
In summary, we have constructed a microlocal partition of unity $I=A_0+A_1+A_2$ such that:
\begin{equation}\label{microlocalpartition}
\begin{aligned}
&A_{0}=\psi_0 (\mathcal{P}) \in \Psi^{0,sc}_h(M;F^\rho), \qquad A_{1}, \, A_{2} \in \Psi^{\mathrm{comp},sc}_h(M;F^\rho);\\
&a_{i}:=\sigma^{\rho}_{h}(A_{i})=\chi_{i}\big(1-\psi_0(|\xi|^2_g)\big)\in C_{0}^{\infty}(\Omega_{i};[0,1]),\\
&\mathrm{WF}_h(A_i)\subset \Omega_{i} , \qquad \ \ i=1,2. 
\end{aligned}
\end{equation}

We next dynamically refine the partition $A_j$. For each $n \in \mathbb{N}^+$, define the set of words of length $n$,
$$
\mathcal{W}(n):=\{1,2\}^n=\big\{\mathbf{w}=w_0 \cdots w_{n-1} :w_0, \cdots, w_{n-1} \in\{1,2\}\big\} .
$$
For each word $\mathbf{w}=w_0 \cdots w_{n-1} \in \mathcal{W}(n)$, we define the operator and the symbol by
\begin{equation}\label{defoflongword}
A_{\mathbf{w}}=A_{w_{n-1}}(n-1) A_{w_{n-2}}(n-2) \cdots A_{w_1}(1) A_{w_0}(0), \qquad a_{\mathbf{w}}=\prod_{j=0}^{n-1}\left(a_{w_j} \circ \varphi_j\right).
\end{equation}
For a subset $\mathcal{E} \subset \mathcal{W}(n)$, we define the operator $A_{\mathcal{E}}$ and the symbol $a_{\mathcal{E}}$ by
\begin{equation*}
    A_{\mathcal{E}}:=\sum_{\mathbf{w} \in \mathcal{E}} A_{\mathbf{w}}, \qquad a_{\mathcal{E}}:=\sum_{\mathbf{w} \in \mathcal{E}} a_{\mathbf{w}} .
\end{equation*}

Since $A_1+A_2=I-A_0=I-\psi_0 (\mathcal{P})$ and $P=\psi_P(\mathcal{P})$ (see \eqref{defofcutoff} and \eqref{defofcutofflaplacian} for the definitions of $\psi_P$ and $\mathcal{P}$) are both functions of $\mathcal{P}$, they commute with each other. Consequently, $A_1+A_2$ commutes with $U(t)$. It follows that
\begin{equation*}
    A_{\mathcal{W}(n)}=\left(A_1+A_2\right)^n.
\end{equation*}
This operator is equal to the identity microlocally near $S^*M$, implying the following lemma (see \cite[Lemma 3.1]{dyatlov2018semiclassical}):
\begin{lemma}\label{awayfromcosphere}
For any $n\in \N$ and any $u \in H^2(M;F^{\rho})$,
\begin{equation}
\|u-(A_1+A_2)^n u\|_{L^2(M;F^\rho)} \leq C\|(\mathcal{P}-I) u\|_{L^2(M;F^\rho)} ,
\end{equation}
where $C>0$ is a constant independent of $u,n,\mathcal{P}$ and $(\mathcal{H},\rho,h)$.
\end{lemma}
Take $\mu \in (0,1)$ very close to $1$, to be chosen later (in Proposition \ref{uncontrolwordpart}), and put
\begin{equation*}
    N_0:=\left\lceil\frac{\mu}{4} |\log h|\right\rceil \in \mathbb{N}, \qquad N_1:=4 N_0 \approx \mu |\log h|.
\end{equation*}
By Proposition \ref{uniformegorovtheorem} and \ref{multiplicationofoperators}, we have:

\begin{lemma}\label{quantizationoflongwordsymbol}
For each $\mathbf{w} \in \mathcal{W}\left(N_0\right)$,
\begin{equation}\label{refinedwordquant}
a_{\mathbf{w}} \in S_{L_s, \mu / 4}^{\mathrm{comp}}(U), \qquad A_{\mathbf{w}}=\mathrm{Op}_h^{\rho, L_s}(a_{\mathbf{w}})+\mathcal{O}(h^{3/4})_{L^2(M;F^\rho) \rightarrow L^2(M;F^\rho)} ;
\end{equation}
if instead $\mathbf{w} \in \mathcal{W}(N_1)$, then
\begin{equation}\label{longerwordquant}
a_{\mathbf{w}} \in S_{L_s, \mu}^{\mathrm{comp}}(U), \qquad A_{\mathbf{w}}=\mathrm{Op}_h^{\rho, L_s}(a_{\mathbf{w}})+\mathcal{O}(h^{1-\mu-})_{L^2(M;F^\rho) \rightarrow L^2(M;F^\rho)} .
\end{equation}
Here the constant in $\mathcal{O}(\bullet)$ is independent of $\mathbf{w}$ and $(\mathcal{H},\rho,h)$.
\end{lemma}
Now, define the density function
\begin{equation}\label{defofdensityoffunc}
F: \mathcal{W}\left(N_0\right) \rightarrow[0,1], \qquad F(w_0 \cdots w_{N_0-1})=\frac{\# \big\{j \in \{0, \cdots, N_0-1\} : w_j=1\big\}}{N_0}.
\end{equation}
Fix small $\alpha \in(0,1)$ to be chosen later (in \eqref{choicealpha}) and define
\begin{equation*}
    \mathcal{Z}:=\{F \geq \alpha\} \subset \mathcal{W}(N_0) .
\end{equation*}
Writing words in $\mathcal{W}(2 N_1)$ as concatenations $\mathbf{w}^{(1)} \cdots \mathbf{w}^{(8)}$, where $\mathbf{w}^{(1)}, \cdots, \mathbf{w}^{(8)} \in \mathcal{W}(N_0)$, we define the partition
\begin{equation}\label{partitionofregions}
\begin{aligned}
\mathcal{W}(2 N_1) & =\mathcal{X} \sqcup \mathcal{Y}, \\
\mathcal{X} & :=\left\{\mathbf{w}^{(1)} \cdots \mathbf{w}^{(8)} : \mathbf{w}^{(\ell)} \notin \mathcal{Z} \text { for all } \ell\right\} ,\\
\mathcal{Y} & :=\left\{\mathbf{w}^{(1)} \cdots \mathbf{w}^{(8)} : \text {there exists } \ell \text { such that } \mathbf{w}^{(\ell)} \in \mathcal{Z}\right\}.
\end{aligned}
\end{equation}
By \cite[Lemma 3.3]{dyatlov2018semiclassical}, the number of elements in $\mathcal{X}$ is bounded by
\begin{equation}\label{numberofuncontrolwords}
\# \mathcal{X} \leq C(\alpha)h^{-4 \sqrt{\alpha}} .
\end{equation}
Now we state two key estimates. The first one estimates the mass of an approximate eigenfunction on the controlled region $\mathcal{Y}$; see \cite[Proposition 3.4]{dyatlov2018semiclassical}:

\begin{proposition}\label{controlwordpart}
For any $\alpha \in (0,1)$ and any $u \in H^2(M;F^{\rho})$,
\begin{equation*}
    \|A_{\mathcal{Y}} u\|_{L^2(M;F^\rho) } \leq \frac{C}{\alpha}\|\mathrm{Op}_h^{\rho}(a) u\|_{L^2(M;F^\rho) }+\frac{C |\log h|}{\alpha h}\|(\mathcal{P}-I) u\|_{L^2(M;F^\rho) }+Ch^{1/8}\|u\|_{L^2(M;F^\rho) } .
\end{equation*}
Here the constant $C>0$ is independent of $u,\alpha,\mathcal{P}$ and $(\mathcal{H},\rho,h)$.
\end{proposition}

The second estimate, based on a fractal uncertainty principle, gives a norm-bound on the operator corresponding to every single word of length $2 N_1 \approx 2 \mu |\log h|$; see \cite[Proposition 3.5]{dyatlov2018semiclassical}:

\begin{proposition}\label{uncontrolwordpart}
There exist $\beta>0$, $\mu \in(0,1)$ depending only on $M, \mathcal{U}_1$ and $ \mathcal{U}_2$ such that
\begin{equation*}
    \sup _{\mathbf{w} \in \mathcal{W}(2 N_1)}\|A_{\mathbf{w}}\|_{L^2(M;F^\rho) \rightarrow L^2(M;F^\rho)} \leq C h^\beta.
\end{equation*}
\end{proposition}

Propositions \ref{controlwordpart} and \ref{uncontrolwordpart} will be proved in the following two sections. Assuming that these results hold at this moment, we can prove Theorem \ref{generalizedmicrolocalcontrolonhilbertbundle}:

\vspace{.1cm}

\begin{proof}[Proof of Theorem \ref{generalizedmicrolocalcontrolonhilbertbundle}] Take $\beta, \mu$ from Proposition \ref{uncontrolwordpart}; we may assume that $\beta<1 / 8$. Since $A_{\mathcal{X}}+A_{\mathcal{Y}}=A_{\mathcal{W}(2 N_1)}=(A_1+A_2)^{2 N_1}$, we have, for any $u \in H^2(M;F^{\rho})$,
\begin{equation*}
    \|u\|_{L^2(M;F^{\rho})} \leq\|A_{\mathcal{X}} u\|_{L^2(M;F^{\rho})}+\|A_{\mathcal{Y}} u\|_{L^2(M;F^{\rho})}+ \big\|u-(A_1+A_2)^{2 N_1} u \big\|_{L^2(M;F^{\rho})} .
\end{equation*}
Combining \eqref{numberofuncontrolwords} with Proposition \ref{uncontrolwordpart} and using the triangle inequality, we have
\begin{equation*}
    \|A_{\mathcal{X}} u\|_{L^2(M;F^{\rho})}=\mathcal{O}\big(h^{\beta-4 \sqrt{\alpha}}\big)\|u\|_{L^2(M;F^{\rho})} .
\end{equation*}
Combining this with Proposition \ref{controlwordpart} and Lemma \ref{awayfromcosphere} we obtain
\begin{equation}\label{controleq}
\begin{aligned}
    &\|u\|_{L^2(M;F^{\rho})} \\
    &\leq \frac{C}{\alpha}\|\Oph^{\rho}(a) u\|_{L^2(M;F^{\rho})}+\frac{C|\log h|}{\alpha h}\|(\mathcal{P} -I) u\|_{L^2(M;F^{\rho})}+Ch^{\beta-4 \sqrt{\alpha}} \|u\|_{L^2(M;F^{\rho})}.
\end{aligned}
\end{equation}
The proof is complete by choosing 
\begin{equation}\label{choicealpha}
    \alpha:=\frac{\beta^2}{64}, \quad \mathrm{so\ that\ }\beta-4 \sqrt{\alpha}=\frac{\beta}{2},
\end{equation}
and taking $h$ small enough to remove the last term on the right-hand side of \eqref{controleq}.
\end{proof}

\subsection{The controlled region: propagation estimates}
In this section, we prove Proposition \ref{controlwordpart}. Recall that 
\begin{equation*}
    \mathrm{supp}\,a_1\cap S^*M \subset \{a\not=0\},
\end{equation*}
there exist $b,q\in C_{0}^{\infty}(T^*M)$ such that $a_1=ab+q(p^2-1)$. Since $\sigma_{h}^{\rho}(A_1)=a_1$ and $\sigma_{h}^{\rho}(\mathcal{P} -I)=p^2-1$, we have
\begin{equation*}
    A_{1}=\Oph^{\rho}(b)\Oph^{\rho}(a)+\Oph^{\rho}(q)(\mathcal{P} -I)+\mathcal{O}(h)_{L^2(M;F^\rho) \rightarrow L^2(M;F^\rho)}.
\end{equation*}
By Proposition \ref{uniformboundedtheorem}, we have the elliptic estimate: there is a constant $C>0$ independent of $u,\mathcal{P}$ and $(\mathcal{H},\rho,h)$ such that, for any $u\in H^2(M;F^{\rho})$,
\begin{equation}\label{ellipticestimate}
 \|A_1 u \|_{L^2(M;F^\rho)} \leq C \|\mathrm{Op}_h^{\rho}(a) u \|_{L^2(M;F^\rho)} + C \| (\mathcal{P} -I ) u \|_{L^2(M;F^\rho)}+C h\|u\|_{L^2(M;F^\rho)}. 
\end{equation}

Following \cite[Lemma 4.2]{dyatlov2018semiclassical}, we have the uniform version of the propagation estimate:

\begin{lemma}
    Assume that $A=A(\mathcal{H};\rho;h): L^2(M;F^{\rho}) \rightarrow L^2(M;F^{\rho})$ is uniformly bounded: there is a universal constant $\Lambda>0$ such that
    \begin{equation*}
        \|A(\mathcal{H};\rho;h)\|_{L^2(M;F^\rho) \rightarrow L^2(M;F^\rho)} < \Lambda, \qquad \forall \ (\mathcal{H},\rho)\in\mathcal{C},\,0<h\leq 1.
    \end{equation*}
Then, for any $t \in \mathbb{R}$ and $u \in H^2(M;F^{\rho})$,
\begin{equation}\label{propagationestimate}
\|A(t) u\|_{L^2(M;F^\rho)} \leq\|A u\|_{L^2(M;F^\rho)}+\frac{C\Lambda|t|}{h}\|(\mathcal{P}-I) u\|_{L^2(M;F^\rho)}.
\end{equation}
Here the constant $C>0$ is independent of $u,t,\mathcal{P}$ and $(\mathcal{H},\rho,h)$.
\end{lemma}

More generally, one can consider operators obtained by assigning a coefficient to each word. For a function $c: \mathcal{W}(N_0) \rightarrow \mathbb{C}$, define the operator $A_c$ and the symbol $a_c$ by
\begin{equation*}
    A_c:=\sum_{\mathbf{w} \in \mathcal{W}(N_0)} c(\mathbf{w}) A_{\mathbf{w}}, \qquad a_c:=\sum_{\mathbf{w} \in \mathcal{W}(N_0)} c(\mathbf{w}) a_{\mathbf{w}} .
\end{equation*}
It follows from \cite[Lemma 4.4]{dyatlov2018semiclassical} and Proposition \ref{quantizationoflongwordsymbol} that $A_c$ is a pseudodifferential operator modulo a small remainder:

\begin{lemma}\label{pseudoofweighted}
Assume that $\sup |c| \leq 1$. Then,
\begin{equation*}
    a_c \in S_{L_s,1/2,1/4}^{\mathrm{comp}}(U), \qquad A_c=\mathrm{Op}_h^{\rho, L_s}(a_c)+\mathcal{O}(h^{1 / 2})_{L^2(M;F^\rho) \rightarrow L^2(M;F^\rho)} .
\end{equation*}
The $S_{L_s, 1 / 2,1 / 4}^{\mathrm{comp}}$ seminorms of $a_c$ and the constant in $\mathcal{O}(h^{1 / 2})$ are independent of $c$.
\end{lemma}

\begin{proof}
We first notice that $\# \mathcal{W}(N_0)=2^{N_0}=\mathcal{O}(h^{-1/4})$. By \cite[Lemma 4.4]{dyatlov2018semiclassical}, $a_c \in S_{L_s,1/2,1/4}^{\mathrm{comp}}(U) $. By Lemma \ref{quantizationoflongwordsymbol}, we have
\begin{equation*}
\begin{aligned}
    A_c&=\sum_{\mathbf{w} \in \mathcal{W}(N_0)} c(\mathbf{w})\left(\mathrm{Op}_h^{\rho,L_s}(a_{\mathbf{w}})+\mathcal{O}(h^{3 / 4})_{L^2(M;F^\rho) \rightarrow L^2(M;F^\rho)}\right) \\
    &=\mathrm{Op}_h^{L_s}(a_c)+\mathcal{O}(h^{1 / 2})_{L^2(M;F^\rho) \rightarrow L^2(M;F^\rho)}.
\end{aligned}
\end{equation*}
The proof is complete.
\end{proof}

By the same argument in \cite[Lemma 4.5]{dyatlov2018semiclassical}, Proposition \ref{sharpgardingforanisoquant} and \ref{pseudoofweighted}, we have the following almost monotonicity property.

Assume $c, d: \mathcal{W}(N_0) \rightarrow \mathbb{R}$ and $|c(\mathbf{w})| \leq d(\mathbf{w}) \leq 1$ for all $\mathbf{w} \in \mathcal{W}(N_0)$. Then, for all $u \in L^2(M;F^{\rho})$, 
\begin{equation}\label{monotonicity}
 \|A_c u \|_{L^2(M;F^\rho)} \leq \|A_d u \|_{L^2(M;F^\rho)}+C h^{1 / 8}\|u\|_{L^2(M;F^\rho)} ,
\end{equation}
where the constant $C$ is independent of $u,c, d$ and $(\mathcal{H},\rho,h)$.

Since $0 \leq \alpha \mathbf{1}_{\mathcal{Z}} \leq F \leq 1$, we combine the elliptic estimate \eqref{ellipticestimate}, the propagation estimate \eqref{propagationestimate}, almost monotonicity \eqref{monotonicity}, and follow the proof of \cite[Lemma 4.6]{dyatlov2018semiclassical} to obtain:

\begin{lemma}\label{estimateofcontrolpart}
For any $u \in H^2(M;F^{\rho})$, 
\begin{equation}\label{estimateofonecontrolledpart}
\|A_{\mathcal{Z}} u\|_{L^2(M;F^\rho)} \leq \frac{C}{\alpha}\|\mathrm{Op}_h^{\rho}(a) u\|_{L^2(M;F^\rho)}+\frac{C|\log h|}{\alpha h}\|(\mathcal{P}-I) u\|_{L^2(M;F^\rho)}+ Ch^{1/8} \|u\|_{L^2(M;F^\rho)},
\end{equation}
where the constant $C>0$ is independent of $u,\alpha,\mathcal{P}$ and $(\mathcal{H},\rho,h)$.
\end{lemma}

\noindent \textit{Proof of Proposition \ref{controlwordpart}.} Recall the partition \eqref{partitionofregions}, we have:
\begin{equation*}
    \mathcal{Y}=\bigsqcup_{\ell=1}^8 \mathcal{Y}_{\ell}, \qquad \mathcal{Y}_{\ell}:=\left\{\mathbf{w}^{(1)} \cdots \mathbf{w}^{(8)} : \mathbf{w}^{(\ell)} \in \mathcal{Z}, \  \mathbf{w}^{(\ell+1)}, \cdots, \mathbf{w}^{(8)} \in \mathcal{W}(N_0) \setminus \mathcal{Z}\right\}.
\end{equation*}
Then, $A_{\mathcal{Y}}=\sum_{\ell=1}^8 A_{\mathcal{Y}_{\ell}}$. Let $\mathcal{Q}:=\mathcal{W}\left(N_0\right) \setminus \mathcal{Z}$, we have the following factorization:
\begin{equation*}
    A_{\mathcal{Y}_{\ell}}=A_{\mathcal{Q}}(7 N_0) \cdots A_{\mathcal{Q}}(\ell N_0) A_{\mathcal{Z}}((\ell-1) N_0)(A_1+A_2)^{(\ell-1) N_0}.
\end{equation*}
By Proposition \ref{propofanisoquantization}, we have 
\begin{equation*}
    \|A_{\mathcal{Q}}\|_{L^2(M;F^\rho) \rightarrow L^2(M;F^\rho)} + \|A_{\mathcal{Z}}\|_{L^2(M;F^\rho) \rightarrow L^2(M;F^\rho)} \leq C .
\end{equation*}
By estimating $(A_1+A_2)^{(\ell-1) N_0} u-u$ using Lemma \ref{awayfromcosphere}, we obtain
\begin{equation}\label{estimateofy}
\|A_{\mathcal{Y}}u\|_{L^2(M;F^\rho)} \leq C \sum_{\ell=1}^8 \|A_{\mathcal{Z}}((\ell-1) N_0) u\|_{L^2(M;F^\rho)}+C\|(\mathcal{P} -I) u\|_{L^2(M;F^\rho)}.
\end{equation}
We use the propagation estimate \eqref{propagationestimate} to obtain:
\begin{equation}\label{estimateofz}
\|A_{\mathcal{Z}}((\ell-1) N_0) u \|_{L^2(M;F^\rho)} \leq \|A_{\mathcal{Z}} u \|_{L^2(M;F^\rho)}+\frac{C |\log h|}{h}\|(\mathcal{P}-I) u\|_{L^2(M;F^\rho)}.
\end{equation}
By using Lemma \ref{estimateofcontrolpart} to bound $\|A_{\mathcal{Z}} u\|_{L^2(M;F^\rho)}$ and by combining \eqref{estimateofy} with \eqref{estimateofz}, we obtain Proposition \ref{controlwordpart}. \hfill $\square$

\subsection{The uncontrolled region: a fractal uncertainty principle} 

In this section, we prove Proposition \ref{uncontrolwordpart}. Take an arbitrary word $\mathbf{w} \in \mathcal{W}(2 N_1)$ and write it as a concatenation of two words in $\mathcal{W}(N_1)$: $
\mathbf{w}=\mathbf{w}_{+} \mathbf{w}_{-}$, where $\mathbf{w}_{ \pm} \in \mathcal{W}(N_1)$. Define the operators $\mathcal{A}_{+}:=A_{\mathbf{w}_{+}}(-N_1)$ and $ \mathcal{A}_{-}:=A_{\mathbf{w}_{-}}$. Then,
\begin{equation*}
    A_{\mathbf{w}}=U(-N_1) \mathcal{A}_{-} \mathcal{A}_{+} U(N_1) .
\end{equation*}
We relabel the letters in the words $\mathbf{w}_{ \pm}$ as
$\mathbf{w}_{+}=w_{N_1}^{+} \cdots w_1^{+}$ and $\mathbf{w}_{-}=w_0^{-} \cdots w_{N_1-1}^{-}$, and we define the symbols $a_{\pm}$ by
\begin{equation*}
    a_{+}=\prod_{j=1}^{N_1}\left(a_{w_j^{+}} \circ \varphi_{-j}\right), \qquad a_{-}=\prod_{j=0}^{N_1-1}\left(a_{w_j^{-}} \circ \varphi_j\right) .
\end{equation*}
Recall from \eqref{defoflongword} that
\begin{equation*}
    \begin{aligned}
        & \mathcal{A}_{-}=A_{w_{N_1-1}^{-}}(N_1-1) A_{w_{N_1-2}^{-}}(N_1-2) \cdots A_{w_1^{-}}(1) A_{w_0^{-}}(0) ,\\
& \mathcal{A}_{+}=A_{w_1^{+}}(-1) A_{w_2^{+}}(-2) \cdots A_{w_{N_1-1}^{+}}(1-N_1) A_{w_{N_1}^{+}}(-N_1).
\end{aligned}
\end{equation*}

\begin{lemma}
Similar with Lemma \ref{quantizationoflongwordsymbol}, the symbols $a_{\pm}$ and the operators $\mathcal{A}_{\pm}$ satisfy
\begin{equation}
\begin{aligned}
    &a_{+} \in S_{L_u, \mu}^{\mathrm{comp}}(U), \qquad \qquad a_{-} \in S_{L_s, \mu}^{\mathrm{comp}}(U) , \\
&\mathcal{A}_{+}=\mathrm{Op}_h^{\rho, L_u} (a_{+} )+\mathcal{O}\big(h^{(1-\mu)-}\big)_{L^2(M;F^\rho) \rightarrow L^2(M;F^\rho)}, \\
&\mathcal{A}_{-}=\mathrm{Op}_h^{\rho, L_s}(a_{-})+\mathcal{O}\big(h^{(1-\mu)-}\big)_{L^2(M;F^\rho) \rightarrow L^2(M;F^\rho)} .
\end{aligned}
\end{equation}
\end{lemma}

Therefore, it suffices to prove the estimate
\begin{equation}
\left\|\mathrm{Op}_h^{\rho, L_s}(a_-) \mathrm{Op}_h^{\rho, L_u}(a_+)\right\|_{L^2(M;F^\rho) \rightarrow L^2(M;F^\rho)} \leq C h^\beta.
\end{equation}
We notice that $\mathrm{supp}\,a_{\pm}\subset \{1/4<|\xi|<4\}$. Applying the microlocal partition of unity on $\overline{U}=\{1/4\leq |\xi|\leq 4\}$ and the similar argument in Proposition \ref{fixtimeegorov}, we only need to show that, for any open subset $V\subset U$ satisfying the condition in \cite[(5.10)]{dyatlov2018semiclassical} and for any $q \in C_{0}^{\infty}(V)$, 
\begin{equation*}
    \left\|\mathrm{Op}_h^{\rho, L_s}(a_-) \Oph^{\rho}(q)\mathrm{Op}_h^{\rho, L_u}(a_+)\right\|_{L^2(M;F^\rho) \rightarrow L^2(M;F^\rho)} \leq C h^\beta.
\end{equation*}
We lift $V \subset U\subset U^{\prime}=T^*M\setminus \{0\}$ to a subset of $T^* \mathbb{H}^2 \setminus \{0\}$ and use $\kappa^{ \pm}$ (see \cite{dyatlovzahl2016spectral}) to define the symplectomorphisms onto their images
\begin{equation*}
    \kappa_0^{ \pm}: V \rightarrow T^*(\mathbb{R}^{+} \times \mathbb{S}^1).
\end{equation*}
Note that we can ensure that $\kappa_0^{ \pm}(V)$ is contained within a compact subset of $T^*(\mathbb{R}^{+} \times \mathbb{S}^1)$, which depends only on $M$.

We define the operators $\mathcal{B}_{ \pm} \in I^{\mathrm{comp}}_h(\kappa_0^{ \pm},\rho,\rho_{\mathrm{triv}})$ and $\widetilde{\mathcal{B}}_{ \pm} \in I^{\mathrm{comp}}_h(\kappa_0^{ \pm},\rho_{\mathrm{triv}},\rho)$ that quantize $\kappa_{0}^{\pm}$ near $\mathrm{supp}\,q$ such that 
\begin{equation}
\begin{aligned}
&\widetilde{\mathcal{B}}_{ \pm}\mathcal{B}_{ \pm}=I+\mathcal{O}(h^{\infty})_{L^2(M;F^{\rho})\to L^2(M;F^{\rho})} \quad \text{microlocally near} \ \mathrm{supp}\,q ,\\
&\mathcal{B}_{ \pm}\widetilde{\mathcal{B}}_{ \pm}=I+\mathcal{O}(h^{\infty})_{L^2(\mathbb{R}^{+} \times \mathbb{S}^1;\mathcal{H})\to L^2(\mathbb{R}^{+} \times \mathbb{S}^1;\mathcal{H})} \quad \text{microlocally near} \ \kappa_0^{ \pm}(\mathrm{supp}\,q).
\end{aligned}
\end{equation}

Consider the following operators on $L^2(\mathbb{R}^{+} \times \mathbb{S}^1;\mathcal{H})$:
\begin{equation}
\begin{aligned}
A_{-}:=\mathcal{B}_{-} \mathrm{Op}_h^{\rho, L_s}(a_-)\widetilde{\mathcal{B}}_{-}, \quad A_{+}:=\mathcal{B}_{+}Q\,\mathrm{Op}_h^{\rho, L_u}(a_+)\widetilde{\mathcal{B}}_{+}, \quad \mathcal{B} :=\mathcal{B}_{-}\widetilde{\mathcal{B}}_{+},
\end{aligned}
\end{equation}
where $Q=\Oph^{\rho}(q)$. Then we have
\begin{equation*}
    \mathrm{Op}_h^{\rho,L_s}(a_-) Q \,\mathrm{Op}_h^{\rho, L_u}(a_+)=\widetilde{\mathcal{B}}_{-} A_{-} \mathcal{B} A_{+} \mathcal{B}_{+}+\mathcal{O}(h^{\infty})_{L^2(M;F^{\rho})\rightarrow L^2(M;F^{\rho})}.
\end{equation*}
By Proposition \ref{mappingoffio}, there exist $\widetilde{a}_{ \pm} \in S_{L_0, \rho}^{\mathrm{comp}}\left(T^*(\mathbb{R}^{+} \times \mathbb{S}^1)\right)$, where $\widetilde{a}_{ \pm}$ are the same as \cite[(5.11)]{dyatlov2018semiclassical}, such that
\begin{equation*}
    A_{ \pm}=\mathrm{Op}_h(\widetilde{a}_{ \pm})\otimes \mathrm{Id}_{\mathcal{H}}+\mathcal{O}(h^{\infty})_{L^2(\mathbb{R}^{+} \times \mathbb{S}^1;\mathcal{H})\to L^2(\mathbb{R}^{+} \times \mathbb{S}^1;\mathcal{H})}, \quad \operatorname{supp} \widetilde{a}_{ \pm} \subset \kappa_0^{ \pm}(V \cap \operatorname{supp} a_{ \pm}).
\end{equation*}
Therefore, it suffices to prove that
\begin{equation*}
    \big\|\big(\mathrm{Op}_h(\widetilde{a}_{-})\otimes \mathrm{Id}_{\mathcal{H}} \big) \mathcal{B} \big(\mathrm{Op}_h(\widetilde{a}_{+})\otimes \mathrm{Id}_{\mathcal{H}}\big)\big\|_{L^2(\mathbb{R}^{+} \times \mathbb{S}^1;\mathcal{H}) \rightarrow L^2(\mathbb{R}^{+} \times \mathbb{S}^1;\mathcal{H})} \leq C h^\beta.
\end{equation*}
Since $\mathcal{B}\in  I^{\mathrm{comp}}_h\left(\kappa^{-}_{0} \circ(\kappa^{+}_{0})^{-1},\mathcal{H}\right)$, by the Definition \ref{defoffioonbundles}, there exist
\begin{equation*}
    B\in I^{\mathrm{comp}}_h\left(\kappa^{-}_{0} \circ(\kappa^{+}_{0})^{-1} \right) \quad \text{and} \quad \phi \in \mathrm{U}(\mathcal{H})
\end{equation*}
such that $\mathcal{B}=B\otimes \phi$. By \eqref{squaretensorextension}, we only need to show 
\begin{equation*}
    \|\mathrm{Op}_h(\widetilde{a}_{-}) B\mathrm{Op}_h(\widetilde{a}_{+})\|_{L^2(\mathbb{R}^{+} \times \mathbb{S}^1) \rightarrow L^2(\mathbb{R}^{+} \times \mathbb{S}^1)} \leq C h^\beta,
\end{equation*}
which is proved in \cite[(5.12)]{dyatlov2018semiclassical} by a fractal uncertainty principle; see also Bourgain and Dyatlov \cite{2018bourgaindyatlov}. Hence, the proof of Proposition \ref{uncontrolwordpart} is complete.

\subsection{Applications}\label{subsec:application}

In this section, we present several applications of Theorem \ref{microlocalcontrolonhilbertbundle}. 
First, we derive a uniform semiclassical control estimate on Riemannian covering spaces, \emph{i.e.}, Theorem \ref{highfrequencycontrol}. 
We then explain how Theorem \ref{microlocalcontrolonhilbertbundle} yields uniform delocalization consequences in spectral theory and quantum chaos, formulated in terms of semiclassical defect measures of eigensections on flat Hilbert bundles. 
Finally, we briefly explain how the same idea extends to the mixed semiclassical/Berezin–Toeplitz quantization framework, leading to the full support of the corresponding augmented semiclassical measures.

\subsubsection{Semiclassical control on covering spaces}

By Theorem \ref{microlocalcontrolonhilbertbundle}, we have the following corollary.

\begin{theorem}\label{controlonflathilbertbundle}
Let $M$ be a compact hyperbolic surface and $0\not \equiv a \in C^{\infty}(M)$. There exist constants $C=C(M,a)>0$ and $h_{0}  = h_0(M,a)>0$ such that for any $(\mathcal{H},\rho)\in \mathcal{C}$, any $0<h<h_{0}$ and any $u \in H^2(M;F^{\rho})$,
\begin{equation}\label{controlonflathilbertbundleq}
\|u\|_{L^2(M;F^{\rho})} \leq C\left( \|au\|_{L^2(M;F^{\rho})}+\frac{|\log h|}{h}\left\|(-h^2 \Delta^{\rho}-I ) u\right\|_{L^2(M;F^{\rho})} \right).
\end{equation}
\end{theorem}

\begin{proof}
For any $a \in C^\infty(M)$, we take the cutoff function $\chi(x,\xi)=1-\psi_{0}(|\xi|^2_x)\in C_{0}^{\infty}(T^*M)$ in \eqref{cutofffunctioninproof}. Then, for any $u\in L^2(M;F^\rho)$, we have 
    \begin{equation}\label{estimate1}
        \left\| \Oph^{\rho}(a \chi )u\right\|_{L^2(M;F^\rho)}\leq \|a u\|_{L^2(M;F^\rho)}+\left\| \Oph^{\rho}\big(a(1- \chi)\big)u\right\|_{L^2(M;F^\rho)}.
    \end{equation}
Let $b(x,\xi)=a(x)\psi_0(|\xi|_x^2)(|\xi|_x^2-1)^{-1} \in S^{-2} (T^*M)$. By Proposition \ref{uniformboundedtheorem} and \ref{propofquantization}, we have
    \begin{equation}\label{estimate2}
        \left\| \Oph^{\rho}\big(a(1- \chi)\big)u\right\|_{L^2(M;F^\rho)}\leq C\left(\|(-h^2\Delta^{\rho}-I)u\|_{L^2(M;F^\rho)}+h\|u\|_{L^2(M;F^\rho)} \right).
    \end{equation}
Combining Theorem \ref{microlocalcontrolonhilbertbundle} with the two estimates \eqref{estimate1} and \eqref{estimate2}, we get Theorem \ref{controlonflathilbertbundleq}.
\end{proof}

\noindent \textit{Proof of Theorem \ref{highfrequencycontrol}.} Recall that $\pi:X\to M$ is an arbitrary Riemannian cover. It induces a unitary representation $\rho:\Gamma_{M}\to \mathrm{U}(\ell^2(V))$, where $V=\Gamma_{X}\backslash \Gamma_{M}$, and a flat $\ell^2(V)$-bundle $F^{\rho}$ over $M$. By the properties of non-commutative Bloch transform \eqref{nonblochtwisted} and \eqref{unitaryofnonbloch}, $\mathcal{B}_{\mathrm{NC}}: L^2(X)\mapsto L^2(M;F^{\rho})$ is an isometry and $\mathcal{B}_{\mathrm{NC}}: H^2(X)\mapsto H^2(M;F^{\rho})$ is an invertible bounded operator. 
For any $u \in H^2(X)$ and any open subset $\Omega \subset M$, we take any $0\not \equiv a\in C_{0}^{\infty}(\Omega)$, then by Theorem \ref{controlonflathilbertbundle},
\begin{equation}
\begin{aligned}
\|u\|_{L^2(X)}&=\|\mathcal{B}_{\mathrm{NC}}u\|_{L^2(M;F^{\rho})}\\
&\leq C\left(\|a\mathcal{B}_{\mathrm{NC}}u\|_{L^2(M;F^{\rho})}+\frac{|\log h|}{h}\|(-h^2\Delta^{\rho}-I)\mathcal{B}_{\mathrm{NC}}u\|_{L^2(M;F^{\rho})}\right)\\
&\leq C\left(\|(\pi^{*}a)u\|_{L^2(X)}+\frac{|\log h|}{h}\|(-h^2\Delta_g-I)u\|_{L^2(X)}\right)\\
&\leq C\left(\|u\|_{L^2(\pi^{-1}(\Omega))}+\frac{|\log h|}{h}\|(-h^2\Delta_g-I)u\|_{L^2(X)}\right).
\end{aligned}
\end{equation}
The proof is complete. \hfill $\square$

\subsubsection{Control of eigensections on flat Hilbert bundles}

Corollary \ref{controlofeigensection} provides a uniform lower bound for the $L^2$-norm of eigensections on any open subset $\Omega\subset M$. Next, we investigate how Theorem \ref{microlocalcontrolonhilbertbundle} can be applied in spectral theory and quantum chaos. Dyatlov and Jin \cite[Theorem 1]{dyatlov2018semiclassical} established that any semiclassical defect measure has full support on the cosphere bundle $S^*M$. Motivated by this, we will define the semiclassical defect measure associated with eigensections and demonstrate a uniform delocalization result for any Laplacian corresponding to a unitary representation.

We define the microlocal lift of the $L^2$-normalized eigensection $u_{j,\rho}$ as the Wigner measure
\begin{equation}\label{wigner}
W_{j,\rho}: C_0^{\infty}(T^*M) \to \mathbb{R}, \quad 
W_{j,\rho}(a) := \langle \mathrm{Op}^{\rho}_{h_{j,\rho}}(a)u_{j,\rho}, u_{j,\rho} \rangle_{L^2(M;F^{\rho})}, \quad h_{j,\rho} = \lambda_{j,\rho}^{-1/2}.
\end{equation}
A semiclassical defect measure $\mu$ is a weak-$*$ limit of the sequence $W_{j,\rho}$ as $\lambda_{j,\rho} \to +\infty$. Using the existence of the elliptic parametrix \eqref{ellipticparametrix}, Egorov’s Theorem \ref{uniformegorovtheorem}, and the sharp Gårding inequality Proposition \ref{uniformsharpgarding}, one concludes that $\mu$ is supported on $S^*M$, is invariant under the geodesic flow $\varphi_t$, and is a probability measure. Furthermore, Theorem \ref{microlocalcontrolonhilbertbundle} provides a uniformly fully supported property for $\mu$:

\begin{theorem}\label{semiclassicaldefectmeasure}
For any open subset $\Omega\subset S^*M$, there exists a constant $C(M,\Omega)>0$ independent of $(\mathcal{H},\rho)$ such that for any semiclassical defect measure $\mu$ given by $W_{j,\rho}$ as $\lambda_{j,\rho}\to \infty$, we have $\mu(\Omega)\ge C(M,\Omega)$. Thus, $\mathrm{supp}\,\mu=S^*M$.
\end{theorem}
\begin{proof}
For any open subset $\Omega \subset S^*M$, we take $a \in C_0^{\infty}(T^* M;[0,1])$ such that $a|_{S^* M} \not \equiv 0$ and $\operatorname{supp} a \cap S^* M \subset \Omega$. Let $u_{j,\rho}$ and $h_{j,\rho}$ be as in \eqref{wigner}. Theorem \ref{microlocalcontrolonhilbertbundle} implies that 
\begin{equation*}
    \left\|\mathrm{Op}_{h_{j,\rho}}^{\rho}(a) u_{j,\rho}\right\|_{L^2(M;F^{\rho})} \geq C^{-1}, \quad \forall\, h_{j,\rho}<h_0.
\end{equation*}
However, if $W_{j,\rho}$ weak-$*$ converges to some measure $\mu$, then
\begin{equation*}
    \lim_{h_{j,\rho}\to 0}\left\|\mathrm{Op}_{h_{j,\rho}}^{\rho}(a) u_{j,\rho}\right\|_{L^2(M;F^{\rho})}^2 =\int_{T^* M}|a|^2 d \mu \leq \mu(\Omega).
\end{equation*}
It follows that $\mu(\Omega)\geq C^{-1}$, where $C$ is the constant in Theorem \ref{microlocalcontrolonhilbertbundle}.
\end{proof}

\subsubsection{Application in mixed quantization}
The Berezin-Toeplitz quantization provides another quantization procedure on the positive line bundle $L$ with Chern connection $\nabla^{L}$ over a compact complex manifold $N$. It implies that $N$ is a Kähler manifold. 

For $p \in \mathbb{N}$, let $L^p=L^{\otimes p}$ be the $p$-th tensor power of $L$, and let $H^{(0,0)}(N, L^p)$ be the space of holomorphic sections of $L^p$ on $N$, which is a finite-dimensional space. Let 
\begin{equation*}
    P_p: L^2(N, L^p) \rightarrow H^{(0,0)}(N, L^p)
\end{equation*}
be the associated orthogonal projection. The Berezin-Toeplitz quantization of $\mathcal{H} \in C^{\infty}(N)$ is given by 
\begin{equation*}
    T_{\mathcal{H}, p}=P_{p}\mathcal{H}P_{p} \in \operatorname{End}\left(L^2(N, L^{p})\right).
\end{equation*} 
Recently, two independent teams, Ma, Ma \cite{mama2024semiclassical} and Cekić, Lefeuvre \cite{cekiclefeuvre2024semiclassical}, developed a mixed quantization that combines semiclassical quantization with Berezin-Toeplitz quantization. In the article by the second team, this is referred to as Borel-Weil calculus.

Let $\mathrm{U}$ be the group of isomorphisms on $L$ that preserve the Chern connection $\nabla^L$, with each element restricting to a holomorphic isomorphism of $N$. Let $\rho: \pi_1(X) \rightarrow \mathrm{U}$ be a finite-dimensional unitary representation. We set 
\begin{equation*}
    F_p:=\widetilde{M} \times_{\rho^{\otimes p}} H^{(0,0)}(N, L^p)
\end{equation*}
to be a sequence of flat Hilbert bundles over $M$. Let $\Delta^{F_p}$ be the Laplacian acting on $C^{\infty}(M; F_p)$. Let the eigenvalues of $\Delta^{F_p}$ be $\lambda_{j,p}=\lambda_{j,\rho^{\otimes p}}$ and the associated eigensections be $u_{j,p}=u_{j,\rho^{\otimes p}}$.

Let $q: \mathcal{N}=\widetilde{M}\times_{\rho}N\to M$ be the fiber bundle with fiber $N$. The pullback bundle $q^*(T^*M)$ is referred to as the total phase space. For any $\mathcal{A}\in C_{0}^{\infty}(q^*(T^*M))$, the mixed quantization $\Oph(\mathcal{A})$ is given by 
\begin{equation*}
    \Oph^{\rho^{\otimes p}}(T_{\mathcal{A},p}): L^2(M;F_p)\to L^2(M;F_p);
\end{equation*}
see \cite{mama2024semiclassical, cekiclefeuvre2024semiclassical}. Furthermore, if the range of $\rho:\pi_{1}(M)\to \mathrm{U}$ is dense in $\mathrm{U}$, the uniform quantum ergodicity holds for $\{u_{j,p}\}$. See \cite[Theorem 1.5 and 1.6]{mama2024semiclassical} and \cite[Section 5.3]{cekiclefeuvre2024semiclassical}.

By \cite[Proposition 5.2]{mama2024semiclassical}, as $\lambda_{j,p}\to \infty$, there exists an augmented semiclassical measure $\nu_{q^*(T^*M)}$ satisfying: there is a series of eigensections $\{u_{j_{\ell},p_{\ell}}\}_{\ell \in \mathbb{N}}$ such that $\lim_{\ell \rightarrow \infty} \lambda_{p_{\ell}, j_{\ell}}=\infty$ and, for any $\mathcal{A} \in C^{\infty}_{0}(q^*(T^* M))$,
\begin{equation*}
    \lim _{\ell \rightarrow \infty}\left\langle\mathrm{Op}_{h_{ j_{\ell},p_{\ell}}}(\mathcal{A}) u_{j_{\ell},p_{\ell}}, u_{j_{\ell},p_{\ell}}\right\rangle_{L^2\left(M, F_{p_{\ell}}\right)}=\int_{q^*\left(T^* M\right)} \mathcal{A} d \nu.
\end{equation*}
Furthermore, $\nu$ is a probability measure on $q^*(T^*M)$ with $\mathrm{supp}\,\nu\subset q^*(S^*M)$, and is invariant under the horizontal geodesic flow on $ q^*(S^*M)$.

In the first version of Ma and Ma's work \cite{mama2024semiclassical}, they showed that if the finite-dimensional case of Theorem \ref{microlocalcontrolonhilbertbundle} holds, the augmented semiclassical measure $\nu_{q^*\left(T^*M\right)}$ is fully supported on the total phase space $q^*(S^*M)$. Now we give the complete proof.

\begin{proposition}
We consider two cases as follows.
\begin{itemize}
\item For any fixed $p\in \mathbb{N}$, the augmented semiclassical measure $\nu_{p}$ is given by $u_{j,p}$ as $j\to \infty$. Then $\mathrm{supp}\,\nu_p=q^*(S^*M)$;
\item If the range of $\rho:\pi_{1}(M)\to \mathrm{U}$ is dense in $\mathrm{U}$, then for any augmented semiclassical measure $\nu$ given by $u_{j,p}$ as $\lambda_{j,p}\to \infty$, we have $\mathrm{supp}\,\nu=q^*(S^*M)$.
\end{itemize}
\end{proposition}
\begin{proof}
For any open subset $\Omega\subset q^*(S^*M)$, we take $\mathcal{A} \in C_{0}^{\infty}(q^*(T^*M);[0,1])$ such that $\mathcal{A}|_{\Omega}\equiv 1$. Let 
\begin{equation*}
    \begin{aligned}
        \widetilde{a}_p(x,\xi) &:=\min_{s \in H^{(0,0)}(N,L^p), \ \|s\|_{L^2}=1}\int_{N}\mathcal{A}(x,\xi,n)\|s(x,n)\|_{L^2}^2 \,d\mathrm{Vol}_{N}(n) \\
        &\ =\min_{s \in H^{(0,0)}(N,L^p), \ \|s\|_{L^2}=1}\langle T_{\mathcal{A},p}s,s\rangle.
    \end{aligned}
\end{equation*}
Then $\widetilde{a}_p\in [0,1]$ is continuous on $T^*M$. For any $(x,\xi)\in q(\Omega)$, there exists a non-empty open subset $\Omega_{n}\subset N$ such that $(x,\xi)\times \Omega_{n}\subset \Omega$. Thus, for any $s \in H^{0,0}(N,L^{p})$ with $\|s\|_{L^2(N,L^p)}=1$, we have $\|s\|_{L^2(\Omega_n,L^p)}>0$ by the holomorphic of $s$. Since $\mathrm{dim} \ H^{0,0}(N,L^{p})=p^{\mathrm{dim} N}<\infty$, the $L^2$-unit holomorphic section forms a compact subset, thus $\widetilde{a}_p(x,\xi)>0$. Then we take $a_p \in C_{0}^{\infty}(T^*M)$ such that $0\leq a_p\leq \widetilde{a}_p\leq 1$ and $a_p>0$ on $q(\Omega)$. It implies $T_{\mathcal{A},p}\geq a_p(x,\xi)\mathrm{Id}_{H^{0,0}(N,L^p)}$. We have
\begin{equation}
\nu_p(\Omega)\geq\lim_{\ell\to \infty}\langle \mathrm{Op}_{h_{j_\ell,p}}(\mathcal{A})u_{j_\ell,p},u_{j_\ell,p}\rangle\geq\lim_{\ell\to \infty}\langle \mathrm{Op}_{h_{j_\ell,p}}^{\rho^{\otimes p}}(a_p)u_{j_\ell,p},u_{j_\ell,p}\rangle\geq C^{-1}_p.
\end{equation}
Therefore, we have proved the first part.

By Theorem \ref{semiclassicaldefectmeasure}, for any open set $U\subset S^*M$, we have $\nu(q^{-1}(U))\geq C(M,U)>0$.
Let 
\begin{equation*}
    \widetilde{\pi}:S^* \widetilde{M} \times N\to q^*(S^*M)
\end{equation*}
be the natural projection. For $(\widetilde{x}, \widetilde{\xi}, z) \in S^* \widetilde{M} \times N$, we denote by $[(\widetilde{x}, \widetilde{\xi}, z)]$ its image under $\widetilde{\pi}$. For any open set $\mathcal{U} \subseteq q^*(S^*M)$, $\widetilde{\pi}^{-1}(\mathcal{U})$ is a $\pi_1(M)$-invariant open set of $S^* \widetilde{M} \times N$, so it contains an open ball $B\big((\widetilde{x}_0, \widetilde{\xi}_0), \varepsilon \big) \times B(z_0, 2 \varepsilon)$ for some $\varepsilon>0$ with $(\widetilde{x}_0, \widetilde{\xi}_0, z_0) \in S^* \widetilde{X} \times N$. 

Choose a finite set $\{z_i\}_{i=1}^k \subset N$ such that $N=\bigcup_{i=1}^k B(z_i, \varepsilon)$. Now we choose a sequence $(t_i)_{i=1}^k$ inductively: first, we set $\widetilde{U}_0=B\big((\widetilde{x}_0, \widetilde{\xi}_0), \varepsilon\big)$; at $i$-th step for $1 \leq i \leq k$, since the horizontal geodesic flow $(\widetilde{\varphi}_t)_{t \in \mathbb{R}}$ on $q^*(S^*M)$ is ergodic with respect to the Liouville measure, we can find $t_i \in \mathbb{R}$ such that
\begin{equation*}
    \widetilde{\varphi}_{t_i} \cdot\left(\left[B\big((\widetilde{x}_0, \widetilde{\xi}_0), \varepsilon\big) \times B\left(z_0, \varepsilon\right)\right]\right) \cap\left[\widetilde{U}_{i-1} \times B\left(z_i, \varepsilon\right)\right] \neq \varnothing,
\end{equation*}
equivalently, there is $\gamma_i \in \pi_1(M)$ such that
\begin{equation*}
    \gamma_i \cdot \widetilde{\varphi}_{t_i}\left(B\big((\widetilde{x}_0, \widetilde{\xi}_0), \varepsilon\big)\right) \cap \widetilde{U}_{i-1} \neq \varnothing, \qquad \gamma_i \cdot B\left(z_0, \varepsilon\right) \cap B\left(z_i, \varepsilon\right) \neq \varnothing,
\end{equation*}
and we put $\widetilde{U}_i=\gamma_i \cdot \widetilde{\varphi}_{t_i}\left(B\big((\widetilde{x}_0, \widetilde{\xi}_0), \varepsilon\big)\right) \cap \widetilde{U}_{i-1}$. Since $\gamma_i$ acts on $N$ isomorphically,  we have $B\left(z_i, \varepsilon\right) \subseteq \gamma_i \cdot B\left(z_0, 2 \varepsilon\right)$, then it follows that
\begin{equation*}
    \widetilde{U}_k \times N \subseteq \bigcup_{i=1}^k \gamma_i \cdot \widetilde{\varphi}_{t_i}\left(B\big((\widetilde{x}_0, \widetilde{\xi}_0), \varepsilon\big)\right) \times \gamma_i \cdot B\left(z_0, 2 \varepsilon\right) .
\end{equation*}
We denote by $\big[\widetilde{U}_k\big]$ the image of $\widetilde{U}_k$ under $\widetilde{\pi}$, then we have
\begin{equation*}
    q^{-1}\big(\big[\widetilde{U}_k\big]\big) \subset \bigcup_{i=1}^k\widetilde{\varphi}_{t_i}(\mathcal{U}).
\end{equation*}
Since $\nu$ is $(\widetilde{\varphi}_t)_{t \in \mathbb{R}}$-invariant, we have
\begin{equation*}
    k \nu(\mathcal{U})=\sum_{i=1}^k \nu\big(\widetilde{\varphi}_{t_i}(\mathcal{U}) \big) \geq \nu\left(q^{-1}\big(\big[\widetilde{U}_k\big]\big) \right)>0,
\end{equation*}
from which we conclude the proof.
\end{proof}

\section{Proof of observability}\label{sec: proof of observability}

In this section, we derive the uniform observability of Schrödinger equations on flat Hilbert bundles from Theorem \ref{controlonflathilbertbundle}. Our method closely follows that of the standard Schrödinger setting in \cite{jin2018control}, with an addition of uniformity in the resulting bounds.

For convenience, we recall the definitions of the semiclassical Fourier transform and its adjoint on $\R$:
\begin{equation*}
    \mathcal{F}_h\varphi(\tau):=\int_\R e^{-it\tau /h}\varphi(t)\,d t, \quad \mathcal{F}_h^*\varphi(\tau):=\int_\R e^{it\tau /h}\varphi(t)\,d t
\end{equation*}
for any $\varphi\in L^1(\R)\cap L^2(\R)$. They can be extended to any $\varphi \in L^2(\R)$ and satisfy the Parseval's identity
\begin{equation}\label{parseval}
    \| \mathcal{F}_h \varphi \|_{L^2(\R)}=\| \mathcal{F}_h^* \varphi \|_{L^2(\R)}=(2\pi h )^{1/2}\|\varphi\|_{L^2(\R)}.
\end{equation}
Let $\mathcal{V}$ be a Hilbert space. If $\varphi \in L^1(\mathbb{R};\mathcal{V})\cap  L^2(\mathbb{R};\mathcal{V})$, the tensor extension $\mathcal{F}_h\otimes\mathrm{Id}_{\mathcal{V}}$ gives the semiclassical Fourier transform on $\mathcal{V}$-valued functions. Let $\mathcal{V}:=L^2(M;F^{\rho})$, we define the operators 
\begin{equation*}
    \mathcal{F}_{\rho,h}:=\mathcal{F}_h\otimes\mathrm{Id}_{\mathcal{V}}, \qquad \mathcal{F}_{\rho,h}^{*}:=\mathcal{F}_h^*\otimes \mathrm{Id}_{\mathcal{V}}.
\end{equation*}
In this section, we always assume that $M$ is a compact connected hyperbolic surface.
% {\color{red}Illustration of some symbols in this setting: $F^\rho$, $Op_h(a),\cdots$}

\subsection{Semiclassical observability}
First, we establish a semiclassical version of observability.
\begin{proposition}\label{prp-semi-obs}
    Let $\chi \in C_0^\infty ( (1/2,2))$, $0\not \equiv\psi \in C_0^\infty (\R;[0,1])$ and $\Omega \subset M$ be an nonempty open subset. There exist constants $C=C(M,\Omega,\chi,\psi)>0$, $h_0= h_0(M,\Omega,\chi,\psi)>0$ such that, for any $(\mathcal{H},\rho) \in \mathcal{C}$, any $0<h<h_0$,  and any $u\in L^2(M;F^\rho)$,
    \begin{equation}
        \| \chi\bigl(-h^2\Delta^\rho\bigr) u  \|^2_{L^2(M;F^\rho)} \le C\int_\R \| \psi(t) e^{it\Delta^\rho} \chi\bigl( -h^2\Delta^\rho\bigr) u\|^2_{L^2(\Omega; F^\rho)}\,dt.\label{eqn-semi-obs}
    \end{equation}
\end{proposition}
\begin{proof}
    For any single pair $(\mathcal{H},\rho)$, one can obtain the above result directly from \cite[Theorem 4]{burq2004}. Here we follow the proof in \cite[Proposition 2.1]{jin2018control} to give a complete proof and keep the control constant uniform in $(\mathcal{H},\rho)$. 
    
    Define $v(t)=e^{ith\Delta^\rho }\chi(-h^2\Delta^\rho)u_0$ and $w(t)=\psi(ht)v(t)$. It is clear that $v$ solves the semiclassical Schrödinger equation 
    \begin{equation*}
        (ih\partial_t +h^2\Delta^\rho) v=0.
    \end{equation*}
    This gives 
    \begin{equation*}
        (ih\partial_t+h^2\Delta^\rho)w=ih^2\psi'(ht)v(t).
    \end{equation*}
    Then we take the adjoint semiclassical Fourier transform $\mathcal{F}_{\rho,h}^*$ in $t$ which gives
    \begin{equation*}
        (-h^2\Delta^\rho-\tau)\mathcal{F}_{\rho,h}^*w(\tau)=-ih^2\mathcal{F}_{\rho,h}^*\bigl( \psi'(ht)v(t) \bigr)(\tau).
    \end{equation*}
    By a simple rescaling of Theorem \ref{controlonflathilbertbundle}, we have, uniformly in $\tau\in [1/2,2]$, for any $0<h<h_0 $ and any $u\in H^2(M;F^\rho)$,
    \begin{equation}\label{rescale-semiclassical-control}
        \|u\|_{L^2(M;F^\rho)} \leq C\left(\|u\|_{L^2(\Omega;F^{\rho})}+ \frac{|\log h|}{h}\left\|\left(-h^2 \Delta^{\rho}-\tau\mathrm{Id}\right) u\right\|_{L^2(M;F^{\rho})}\right).   
    \end{equation}
    For any $\tau \in [1/2,2]$, applying \eqref{rescale-semiclassical-control} to $u=\mathcal{F}_{\rho,h}^* w(\tau)$ and $\mathrm{supp}\, \eta \subset \Omega$, we obtain
    \begin{equation}
        \|\mathcal{F}_{\rho,h}^*w(\tau)\|_{L^2(M;F^\rho)}\le  C\|\mathcal{F}_{\rho,h}^* w(\tau) \|_{L^2(\Omega;F^\rho)}+Ch|\log h|\| \mathcal{F}_{\rho,h}^*\bigl( \psi'(ht)v(t) \bigr)(\tau) \|_{L^2(M;F^\rho)}.\label{eqn-tau-in}
    \end{equation}
    For $\tau \not\in [1/2,2]$, by definition, we have
    \begin{equation*}
        \mathcal{F}_{\rho, h}^*w(\tau) =\int_\R e^{-it(-h^2\Delta^\rho-\tau)/h}\psi(ht)\chi(-h^2\Delta) u_0\,d t.
    \end{equation*}
    We rewrite it as
    \begin{equation}
            \mathcal{F}_{\rho,h}^*w(\tau)= \int_\R (h^2\Delta^\rho+\tau)^{-N}(hD_t)^N e^{-it(-h^2\Delta-\tau)/h} \psi(ht)\chi(-h^2\Delta)u_0\,d t.\label{integration-by-part}
    \end{equation}
    Recall that $\chi \in C_0^\infty ((1/2,2))$ and therefore we have
    \begin{equation}
        \| (h^2\Delta^\rho+\tau)^{-N}\chi(-h^2\Delta^\rho) u_0\|_{L^2(M;F^\rho)}\le C_{N}\langle \tau \rangle^{-N} \|\chi(-h^2\Delta^\rho)u_0 \|_{L^2(M;F^\rho)},\label{eqn-tau-notin}
    \end{equation}
    where $C_{N}$ is a positive constant depending on $N$ but not on the choice of $(\mathcal{H},\rho)$. Substituting \eqref{eqn-tau-notin} into \eqref{integration-by-part}, we obtain
    \begin{equation}
        \|\mathcal{F}_{\rho,h}^* w(\tau)\|_{L^2(M;F^\rho)}=\mathcal{O}((h\langle \tau\rangle^{-1})^\infty)\| \chi(-h^2\Delta^\rho)u_0 \|_{L^2(M;F^\rho)}.\label{eqn-tau-notin-2}
    \end{equation}
    Combining \eqref{eqn-tau-in} with \eqref{eqn-tau-notin-2}, we obtain
    \begin{align}
        \| \mathcal{F}_{\rho,h}^*w(\tau) \|^2_{(\R_\tau,L^2(M;F^\rho))}\le&\, C\|\mathcal{F}_{\rho,h}^* w(\tau)\|^2_{L^2(\R_\tau,L^2(\Omega,F^\rho))}\label{last-two-term}\\ 
        &\,+Ch^2|\log h|^2\| \mathcal{F}_{\rho, h}^*\bigl( \psi'(ht)v(t) \bigr)(\tau)   \|^2_{L^2(\R_\tau, L^2(M;F^\rho))}\notag\\ 
        &\,+ \mathcal{O}(h^\infty)\| \chi(-h^2\Delta^\rho)u_0 \|^2_{L^2(M;F^\rho)}.\notag
    \end{align}
    By the Parseval identity \eqref{parseval}, we have
    \begin{equation*}
        \begin{aligned}
            \| w \|^2_{(\R_t,L^2(M;F^\rho))}\le&\, C\|w\|^2_{L^2(\R_t,L^2(\Omega,F^\rho))}\\
            &\,+Ch^2|\log h|^2\|  \psi'(ht)v(t)   \|^2_{L^2(\R_t, L^2(M;F^\rho))}\\
            &\,+ \mathcal{O}(h^\infty)\| \chi(-h^2\Delta^\rho)u_0 \|^2_{L^2(M;F^\rho)}.
        \end{aligned}
    \end{equation*}
    By the definition of $v$ and $w$, we have
    \begin{equation*}
        \begin{aligned}
            \|w\|^2_{(\R_t,L^2(M;F^\rho))}=&\,\int_\R \psi(ht)^2\|e^{it\Delta^\rho}\chi(-h^2\Delta^\rho)u_0\|^2_{L^2(M;F^\rho)}\,d t\\
            =&\, \left(\int_\R \psi(ht)^2\,d t \right)\|\chi(-h^2\Delta^\rho)u_0\|^2_{L^2(M;F^\rho)}\,d t\\
            =&\, h^{-1}\|\psi\|^2_{L^2(\R)}\|\chi(-h^2\Delta^\rho) u_0\|^2_{L^2(M;F^\rho)},
        \end{aligned}
    \end{equation*}
    \begin{equation*}
        \begin{aligned}
            \| w\|^2_{(\R_t,L^2(\Omega,F^\rho))} = &\,\int_\R \psi(ht)^2\| \chi(-h^2\Delta^\rho) u_0 \|^2_{L^2(\Omega, F^\rho)}\,d t\\
            =&\, h^{-1}\int_\R \| \psi(t) e^{it\Delta^\rho} \chi(-h^2\Delta^\rho) u_0 \|_{L^2(\Omega)}^2\,d t,
        \end{aligned}
    \end{equation*}
    and
    \begin{equation*}
        \begin{aligned}
            \| \psi'(ht) v(t) \|^2_{L^2(\R_t,L^2(M;F^\rho))}=&\, \int_\R |\psi'(ht)|^2\| e^{ith\Delta^\rho} \chi(-h^2\Delta^\rho) u_0\|^2_{L^2(M;F^\rho)}\,d t \\
            =&\, h^{-1}\|\psi'\|^2_{L^2(\R)} \| \chi(-h^2\Delta^\rho) u_0 \|^2_{L^2(M;F^\rho)}.
        \end{aligned}
    \end{equation*}
    As long as $h$ is small enough and $\psi\not\equiv 0$, the last two terms on the right-hand side of \eqref{last-two-term} can be absorbed into the left-hand side and the proof is completed.
\end{proof}

\subsection{Observability with an error term} We prove Theorem \ref{thm-uniform-obs} with an error term:
\begin{proposition}\label{prp-obs-error}
Let $u \in L^2(M;F^{\rho})$, $\Omega$ is a non-empty open subset of $M$, $T>0$, then there exists a constant $C_T=C(M, \Omega, T)>0$ such that 
\begin{equation}
\|u\|_{L^2(M;F^{\rho})} \leq C_T\left(\int_{0}^{T}\|\mathrm{e}^{it \Delta^{\rho}}u\|^2_{L^2(\Omega,F^{\rho})}\,d t+ \|u\|^2_{H^{-4}(M;F^\rho)}\right).\label{eqn-obs-error}
\end{equation}

\end{proposition}
\begin{proof}
    Again this follows directly from  \cite[Theorem 7]{burq2004}. We present the argument in our concrete setting, which is essentially the same as the one given in \cite[Proposition 2.2]{jin2018control}. We use a dyadic decomposition  
    \begin{equation*}
        1=\varphi_0(r)^2+\sum_{k=1}^\infty \varphi_k(r)^2,
    \end{equation*}
    where 
    \begin{equation*}
        \varphi_0 \in C_0^\infty \bigl( (-2,2); [0,1] \bigr),\quad \varphi_k(r)=\varphi(2^{-k}|r|),\quad \varphi \in C_0^\infty \bigl((1/2,2);[0,1]\bigr).
    \end{equation*}
    Then we have
    \begin{equation}
        \|u\|^2_{L^2(M;F^\rho)}=\sum_{k=0}^\infty \|\varphi_k(-\Delta^\rho) u\|^2_{L^2(M;F^\rho)}.\label{eqn-dyadic-decomp}
    \end{equation}
    We choose an integer $K$ large enough so that $2^{-K}<h_0^2$. Then for any integer $k\ge K$, by \eqref{eqn-semi-obs}, we obtain
    \begin{equation}
        \|\varphi_k(-\Delta^\rho) u\|^2_{L^2(M;F^\rho)}\le C \int_\R \|\psi(t) e^{it\Delta^\rho}\varphi_k(-\Delta^\rho )u\|^2_{L^2(\Omega,F^\rho)}\,d t\label{eqn-dyadic-obs}
    \end{equation}
    uniformly in $k$ for a chosen nonvanishing function $\psi\in C_0^\infty ((0,T);[0,1])$. Taking \eqref{eqn-dyadic-obs} into \eqref{eqn-dyadic-decomp}, we obtain
    \begin{equation}
        \|u\|^2_{L^2(M;F^\rho)}\le \sum_{k=0}^{K-1} \|\varphi_k(-\Delta^\rho)u\|^2_{L^2(M;F^\rho)}+\sum_{k=K}^\infty C \int_\R \|\psi(t) e^{it\Delta^\rho}\varphi_k(-\Delta^\rho )u\|^2_{L^2(\Omega,F^\rho)}\,d t.\label{eqn-low-high-terms}
    \end{equation}
    Since
    \begin{equation*}
        \|u\|^2_{H^{-4}(M;F^\rho)}\sim \|(-\Delta^\rho+I)^{-2}u\|^2_{L^2(M;F^\rho)}\sim \sum_{k=0}^\infty 2^{-4k}\|\varphi_k(-\Delta^\rho) u\|^2_{L^2(M;F^\rho)},
    \end{equation*}
    The first sum in the right-hand side of \eqref{eqn-low-high-terms} can be bounded from above by $H^{-4}$ norm of $u$, that is,
    \begin{equation}
        \sum_{k=0}^{K-1} \|\varphi_k(-\Delta^\rho)u\|^2_{L^2(M;F^\rho)} \le C \|u\|^2_{H^{-4}(M;F^\rho)}.\label{eqn-low-estimate}
    \end{equation}
    To estimate the second sum in the right-hand side of \eqref{eqn-low-high-terms}, we use the Schrödinger equation $(D_t-\Delta^\rho)e^{it\Delta^\rho}=0$
    to change the frequency localization in space $\varphi_k(-\Delta^\rho)$ to frequency localization in time $\varphi_k(D_t)$. Precisely speaking, we have
    \begin{equation*}
        e^{it\Delta^\rho}\varphi_k(-\Delta^\rho) u =\varphi_k(-\Delta^\rho) e^{it \Delta^\rho} u=\varphi_k(-D_t)e^{it\Delta^\rho} u= \varphi_k(D_t)e^{it\Delta^\rho} u,
    \end{equation*}
    where we used the fact that all $\varphi_k$ are even. We introduce another cutoff function $\widetilde{\psi}\in C_0^\infty \bigl((0,T);[0,1]\bigr)$ such that $\widetilde{\psi}=1$ on a neighborhood of $\mathrm{supp}\, \psi$. Then we have
    \begin{equation*}
        \psi(t)\varphi_k(D_t)= \psi(t)\varphi_k(D_t)\widetilde{\psi}(t)+R_k(t,D_t)
    \end{equation*}
    where $R_k(t,D_t)=\psi(t) [\widetilde{\psi}(t),\varphi(2^{-k}D_t)]$ with its symbol $R_k(t,\tau)$ satisfying
    \begin{equation}
        \partial^\alpha R_k(t,\tau)=\mathcal{O} \bigl( 2^{-kN}\langle t\rangle^{-N}\langle \tau\rangle^{-N} \bigr),\quad \forall N.\label{eqn-symbol-estimate}
    \end{equation}
    Then we have
    \begin{equation}
        \|\psi(t)e^{it\Delta^\rho }\varphi_k(-\Delta^\rho)u\|^2_{L^2(\Omega,F^\rho)} \le \|\varphi_k(D_t)\widetilde{\psi}(t)e^{it\Delta^\rho }u\|^2_{L^2(\Omega,F^\rho)}+ \| R_k(t,D_t)e^{it\Delta^\rho} u\|^2_{L^2(\Omega,F^\rho)}.\label{eqn-high-change}
    \end{equation}
    Taking \eqref{eqn-low-estimate} and all \eqref{eqn-high-change} into \eqref{eqn-low-high-terms}, we obtain
    \begin{equation}
        \begin{aligned}
            \|u\|^2_{L^2(M;F^\rho )}\le &\, C\sum_{k=K}^\infty \int_\R \|\varphi_k(D_t) \widetilde{\psi}(t)e^{it\Delta^\rho} u\|^2_{L^2(\Omega,F^\rho)}\,d t\\
            &\, +C \sum_{k=K}^\infty \int_\R \|R_k(t,D_t) e^{it\Delta^\rho} u\|^2_{L^2(\Omega)} \,d t + C\|u\|^2_{H^{-4}(M;F^{\rho})}.
        \end{aligned}\label{eqn-low-high-2}
    \end{equation}
    The first sum is bounded by
    \begin{equation}
        \sum_{k=K}^\infty \int_\R \|\varphi_k(D_t) \widetilde{\psi}(t)e^{it\Delta^\rho} u\|^2_{L^2(\Omega,F^\rho)}\,d t\le  \int_\R \| \widetilde{\psi}(t)e^{it\Delta^\rho} u\|^2_{L^2(\Omega,F^\rho)}\,d t.\label{eqn-first-bound}
    \end{equation}
    To estimate the second sum, we write
    \begin{equation*}
        \begin{aligned}
            R_k(t,D_t) e^{it\Delta^\rho}u&=R_k(t,D_t)(-D_t+1)^2 e^{it\Delta^\rho}(-\Delta^\rho+I)^2 u\\
            &=\widetilde{R}_k(t,D_t)\langle t\rangle^{-2} e^{it\Delta^\rho} (-\Delta^\rho+I)^{-2} u
        \end{aligned}
    \end{equation*}
    where the symbol of $\widetilde{R}_k(t,D_t)= R_k(t,D_t)(-D_t+1)^2\langle t\rangle^2$ also satisfies \eqref{eqn-symbol-estimate} and thus $\widetilde{R}_k(t,D_t)=\mathcal{O}(2^{-k}):L^2(R,F^\rho)\to L^2(R,F^\rho)$. Therefore the second sum of \eqref{eqn-low-high-2} is bounded by 
    \begin{align}
        \sum_{k=K}^\infty \int_\R \|R_k(t,D_t) e^{it\Delta^\rho} u\|^2_{L^2(\Omega)} \,d t &\le C\sum_{k=K}^\infty 2^{-2k} \|\langle t\rangle^{-2} e^{it\Delta^\rho}(-\Delta^\rho+I)^{-2} u\|^2_{L^2(\R\times M;F^{\rho})}\notag\\
        &\le C \| (-\Delta^\rho+I)^{-2}u\|^2_{L^2(M;F^\rho)}=C\|u\|^2_{H^{-4}(M;F^\rho)}.\label{eqn-second-bound}
    \end{align}
    Taking \eqref{eqn-first-bound} and \eqref{eqn-second-bound} into \eqref{eqn-low-high-2} finishes the proof of \eqref{eqn-obs-error}.
\end{proof}
If we apply $\mathcal{H}=\ell^2(\Gamma)$, we obtain the observability with error term for any covering space of $M$.

\begin{theorem}
Assume that $\Omega \subset M$ is a non-empty open subset and $T>0$, then there exist constants $C>0$ and $h_0>0$ depending only on $\Omega$, $M$ and $T$, independent of $\Gamma$, such that for all $h \in\left(0, h_0\right)$, all $u \in L^2(X)$ we have the estimate
$$
\|u\|_{L^2(X)} \leq C\left(\int_{0}^{T}\|e^{it\Delta_g}u\|_{L^2(\pi^{-1}(\Omega))}dt+\|u\|_{H^{-4}(X)}\right) .
$$
\end{theorem}

However, when $\mathcal{H}$ is infinite-dimensional, $\Delta^{\rho}$ is not a Fredholm operator on $L^2(M;F^{\rho})$. We cannot apply the uniqueness-compactness argument to remove the remainder term. Finally, we assume that $\Gamma$ is type I group. Then, the dimension of the irreducible representation of $\Gamma$ is finite, \emph{i.e.}, $d_\Gamma<\infty$. By the generalized Bloch Theorem, we can fix a Hilbert space $\mathcal{H}=\mathrm{End}(\mathbb{C}^{d_\Gamma})$.

\subsection{Removing the error term}
The last step is to remove the error term and therefore complete the proof of Theorem \ref{thm-uniform-obs}. We use the  uniqueness-compactness argument of Bardos, Lebeau and Rauch \cite{bardo1992}. 

%\begin{proposition}
 %   Fix $\mathcal{H}=\mathrm{End}(\mathbb{C}^{d_\Gamma})$, for any unitary representation $\rho: \Gamma \to \mathrm{U}(\mathcal{H})$, any non-empty open subset $\Omega\subset M$, and any $T>0$, there exists a $C=C(M,\Omega, T, d_\Gamma)$ such that for any $u \in L^2(M;F^{\rho})$,
  %  \begin{equation}
   % \|u\|_{L^2(M)}^2\leq C\int_{0}^{T}\|e^{it\Delta^{\rho}}u\|_{L^2(\Omega)}^2dt.
    %\end{equation}
%\end{proposition}

For any unitary representation $\rho$ and any $T>0$, we define a linear subspace
 \begin{equation*}
     N_T^\rho=\lbrace u\in L^2(M;F^\rho):e^{it\Delta^\rho}u\equiv 0 \text{ on } (0,T)\times \Omega\rbrace.
 \end{equation*}

 \begin{lemma}
     For any unitary representation $\rho$ and any $T>0$, we have $N^\rho_T=\lbrace 0\rbrace$.
 \end{lemma}
 \begin{proof}
     Given any fixed $\rho$ and $T>0$, let $\varepsilon>0$ and $u\in N_T^\rho$. We have
     \begin{equation*}
         v_\epsilon =\frac{e^{it\Delta^\rho} u-u}{\varepsilon} \in N_{T-\delta}^\rho, \quad \forall \varepsilon\le \delta.
     \end{equation*}
     Since $u\in L^2(M;F^\rho)$, $v_\varepsilon$ is uniformly bounded in $H^{-4}(M;F^\rho)$ for $\varepsilon<T/2$,
     \begin{equation*}
         \|v_\varepsilon\|^2_{H^{-4}(M;F^\rho)}\le C ,\quad \forall \varepsilon\le \delta,
     \end{equation*}
     where $C>0$ is a positive constant. Then \eqref{eqn-obs-error} gives
     \begin{equation*}
         \|v_\varepsilon\|_{L^2(M;F^\rho)}^2\le C_{T/2} \|v_{\varepsilon}\|^2_{H^{-4}(M;F^\rho)}\le C_{T/2}C,\quad \forall \varepsilon\le \delta.
     \end{equation*}  
     By definition of $e^{it\Delta^\rho}$, we obtain $v_\varepsilon \to i\Delta^\rho u$ in $ L^2(M;F^\rho)$ for $\varepsilon\to 0$ and hence $u\in H^2(M;F^\rho)$. This gives $\Delta^\rho u\in N_{T-\delta}^\rho$ for any $\delta>0$ and hence $\Delta^\rho u \in N_T^\rho$. Therefore, $N_T^\rho$ is stable under $\Delta^\rho$. Again by Proposition \ref{prp-obs-error}, the $H^{-4}(M;F^\rho)$ norm is equivalent to the $L^2(M;F^\rho)$ norm on $N_T^\rho$, so the unit ball in $N_T^\rho$ is compact and thus $N_T^\rho$ is of finite dimension. If it is not $\lbrace 0\rbrace$, then it must contain some eigenfunction $\varphi$, but this would imply that $\varphi\equiv 0$ on $\Omega$, which violates Theorem \ref{controlonflathilbertbundle}. Hence we finish the proof.  
 \end{proof}
We notice that this step is the only place we need the Fredholm property of $\Delta^{\rho}$, because we need the compactness of the embedding $H^{-4} \hookrightarrow L^2$. 
\begin{proof}[Proof of Theorem~\ref{thm-uniform-obs}]
 To remove the error term, we proceed by contradiction. Suppose that there exists a sequence $\lbrace u_n \rbrace_{n\in\mathbb{N}^+}$ in $L^2(M;F^{\rho_n})$ such that
 \begin{equation}
     \|u_n\|_{L^2(M;F^{\rho_n})}=1 \quad\text{ and }\quad \int_0^T \|e^{it\Delta^{\rho_n}}u_n\|^2_{L^2(\Omega;F^{\rho_n})} \,d t \le \frac{1}{n}.\label{eqn-assump-error}
 \end{equation}
 Then there exists a subsequence $\lbrace u_{n_k}\rbrace_{k\in\mathbb{N}^+}$ and a sequence of maps $\lbrace \mathfrak{U}_{\rho_{n_k}} \rbrace_{k\in \mathbb{N}^+}$ defined in Section~\ref{subsec: twisted laplacian}, such that
 \begin{equation*}
      \mathfrak{U}_{\rho_{n_k}}u_{n_k}\mathfrak{U}_{\rho_{n_k}}^{-1}\rightharpoonup \mathfrak{U}_{\rho_{0}}u_0\mathfrak{U}_{\rho_{0}}^{-1} \ \ \ \ \text{  weakly in } L^2(M;\mathrm{End}(\mathbb{C}^{d_\Gamma}))
 \end{equation*}
 for some $u_0\in L^2(M;F^\rho)$ and some finite dimensional unitary representation $\rho_0$ with the $\rho_0$-equivalent map $\mathfrak{U}_{\rho_{0}}$. Thus $\mathfrak{U}_{\rho_{n_k}}u_{n_k}\mathfrak{U}_{\rho_{n_k}}^{-1}\to \mathfrak{U}_{\rho_{0}}u_0\mathfrak{U}_{\rho_{0}}^{-1}$ strongly in $H^{-4}(M;\mathrm{End}(\mathbb{C}^{d_\Gamma}))$. On the one hand, applying \eqref{eqn-obs-error} and \eqref{eqn-assump-error}, we obtain
 \begin{equation}
     1\le C\frac{1}{n_k}+C\|u_{n_k}\|^2_{H^{-4}(M;F^\rho)}.
 \end{equation}
 Taking $k\to \infty$, we obtain
 \begin{equation}
     \|(1-\Delta_\rho)^{2}\mathfrak{U}_{\rho_{0}}u_0\mathfrak{U}_{\rho_{0}}^{-1}\|_{L^2(M;\mathrm{End}(\mathbb{C}^{d_\Gamma})}\ge C^{-1/2}>0.\label{eqn-non-zero}
 \end{equation}
 On the other hand, we have
 \begin{equation*}
     \int_0^T\|e^{it\Delta_g}u_0\|^2_{L^2(\Omega;F^\rho)}\,d t=0.
 \end{equation*}
 Then $u_0\in N_T^\rho$ and thus $u_0\equiv 0$, which contradicts to \eqref{eqn-non-zero} finishing the proof of Theorem \ref{thm-uniform-obs}.  
\end{proof}
Now we complete the proof of Theorem \ref{thm-main-obs}:
\begin{proof}
According to the generalized Bloch Theorem,
\begin{equation}
\begin{aligned}
\|u\|_{L^2(X)}^2&=\int_{\widehat{\Gamma}}\|\mathcal{B}u([\rho],x)\|^2_{L^2(M,F^\rho)}d\mu([\rho])\\
&\leq C\int_{\widehat{\Gamma}}\int_{0}^{T}\|e^{it\Delta^{\rho}}\mathcal{B}u([\rho],x)\|^2_{L^2(\Omega,F^\rho)}dtd\mu([\rho])\\
&=C\int_{0}^{T}\|e^{it\Delta_g}u(x)\|^2_{L^2(\pi^{-1}(\Omega))}dt.
\end{aligned}
\end{equation}
The proof is complete.
\end{proof}

\section*{Acknowledgments}
The authors would like to thank Semyon Dyatlov, Long Jin, and Zhongkai Tao for their valuable suggestions about semiclassical analysis on flat Hilbert bundles; Qiaochu Ma for useful discussions regarding the full support of augmented semiclassical measures for eigensections; and Philippe Jaming for helpful discussions at an early stage of the project. Yulin Gong was supported by the EPSRC [EP/W007010/1].

%%%%%%%%%
\section*{Data availability statement}

Data sharing is not applicable to this article, as no datasets were generated or analyzed during the current study.

\section*{Conflict of interest}

The authors declare that they have no conflict of interest.

%%%%%%%%%

\bibliographystyle{alpha}
\bibliography{article.bib}
\end{document}